\documentclass[11pt,reqno]{amsart}

\usepackage{amsmath, amsfonts, amsthm, amssymb, stmaryrd, color, cite}

\newtheorem{theorem}{Theorem}[section]
\newtheorem{lemma}{Lemma}[section]
\newtheorem{proposition}{Proposition}[section]

\theoremstyle{definition}
\newtheorem{definition}{Definition}[section]

\theoremstyle{remark}

\numberwithin{equation}{section}
\allowdisplaybreaks

\def\f{\frac}

\def\hf1{^\f{1}{1-\xi^2}}

\def\be{\begin{equation}}
\def\en{\end{equation}}
\def\bs{\begin{split}}
\def\es{\end{split}}
\def\ba{\begin{align}}
\def\ea{\end{align}}

\author[Z. Qiu]{Zhaoyang Qiu}
\address{School of Mathematics and Statistics, Huazhong University of Science and Technology, Wuhan, 430074, China.}
\email{zhqmath@163.com}

\author[Y. Wang]{Yixuan Wang}
\address{Department of Mathematics, University of Pittsburgh, Pittsburgh, 15260, USA.}
\email{YIW119@pitt.edu}

\title[Martingale solution for stochastic active liquid crystal system]
{Martingale solution for stochastic active liquid crystal system}

\keywords{Stochastic active liquid crystal system, global martingale solution, stochastic compactness, four-level approximation}
\subjclass[2000]{35Q35, 76D05, 76A15}

\date{\today}

\begin{document}
\begin{abstract}
 The global weak martingale solution is built through a four-level approximation scheme to stochastic compressible active liquid crystal system driven by multiplicative noise in a smooth bounded domain in $\mathbb{R}^{3}$ with large initial data. The coupled structure makes the analysis challenging, and more delicate arguments are required in stochastic case compared to the deterministic one \cite{11}.
\end{abstract}

\maketitle
\section{Introduction}

 {The PDEs perturbed randomly are considered} as a primary tool in the modeling of uncertainty, especially while describing fundamental phenomenon in physics, climate dynamics, communication systems and gene regulation systems. Hence, the study of the well-posedness and dynamical behaviour of PDEs subjected to the noise which is largely applied to the theoretical and practical areas has drawn a lot of attention. In this paper, we consider the global weak martingale solution to the following stochastic compressible active liquid crystal system perturbed by a multiplicative noise in a smooth bounded domain $\mathcal{D}$, which consists of the compressible Navier-Stokes equation coupled with Q-tensor equation as well as the concentration equation of active particles,
\begin{eqnarray}\label{Equ1.1}
\left\{\begin{array}{ll}
\partial_{t}c+(u\cdot \nabla)c=\triangle c,\\
\partial_{t}\rho+{\rm div}(\rho u)=0,\\
\partial_{t}(\rho u)+{\rm div}(\rho u\otimes u)+\nabla p
 =\mu_1\Delta u+(\mu_1+\mu_2)\nabla ({\rm div} u)+\sigma^*\nabla\cdot(c^2Q)\\ \qquad\qquad+\nabla\cdot({\rm F}(Q){\rm I}_3-\nabla Q\odot \nabla Q)+\nabla\cdot(Q\triangle Q-\triangle Q Q) +\rho f(\rho,\rho u, c, Q)\frac{d\mathcal{W}}{dt}, \\
\partial_{t}Q+(u\cdot \nabla)Q+Q\Psi-\Psi Q=\Gamma {\rm H}(Q,c),\\
\end{array}\right.
\end{eqnarray}
where $c, \rho, u $ denote the concentration of active particles, the density, and the flow velocity, $p(\rho)=\rho^\gamma$ stands for the pressure with the adiabatic exponent $\gamma>\frac{3}{2}$, the nematic tensor order parameter $Q$ is a traceless and $3\times 3$ symmetric matrix, ${\rm I}_3$ is the $3\times 3$ identity matrix$, \mu_1, \mu_2$ are the viscosity coefficients satisfying the physical assumptions $\mu_1\geq 0$ and $2\mu_1+3\mu_2\geq 0$,  $\Gamma^{-1}>0$ is the rotational viscosity, $\sigma^*\in \mathbb{R}$ is the stress generated by the active particles along the director field. $\Psi=\frac{1}{2}(\nabla u-\nabla u^\bot)$ is the skew-symmetric part of the rate of strain tensor. $\mathcal{W}$ is a cylindrical Wiener process which will be introduced later. Furthermore,
\begin{eqnarray*}
{\rm F}(Q)=\frac{1}{2}|\nabla Q|^2+\frac{1}{2}{\rm tr}(Q^2)+\frac{c_*}{4}{\rm tr}^2(Q^2),
\end{eqnarray*}
and
\begin{eqnarray*}
{\rm H}(Q,c)=\triangle Q-\frac{c-c_*}{2}Q+b\left(Q^2-\frac{\rm tr(Q^2)}{3}{\rm I}_3\right)-c_*Q{\rm tr}(Q^2),
\end{eqnarray*}
where the constant $c_*$ is the critical concentration for the isotropic-nematic transition and $b$ is material-dependent constant. And the term $\nabla Q\odot \nabla Q$ stands for a $3\times 3$ matrix, its $(i,j)$-th entry is defined $(\nabla Q\odot \nabla Q)_{ij}=\sum_{k,l=1}^3\partial_iQ_{kl}\partial_jQ_{kl}$.

The system is supplied with the following initial data,
\begin{eqnarray}
\rho(0,x)=\rho_{0}(x), ~\rho u(0,x)=m_0(x),~ c(0,x)=c_{0}(x),~ Q(0,x)=Q_0(x),\label{1.2}
\end{eqnarray}
and  {the boundary conditions},
\begin{eqnarray}
\frac{\partial c}{\partial n}\bigg|_{\partial \mathcal{D}}=\frac{\partial Q}{\partial n}\bigg|_{\partial \mathcal{D}}=0, ~u|_{\partial \mathcal{D}}=0, \label{1.3}
\end{eqnarray}
here we omit the random element $\omega$.

Starting from the 1880’s, a new material that shares both the property of conventional liquids and those of solid crystals was found and named as liquid crystals. The nematic liquid crystals are one of the major types of liquid crystals, which is a representative model of the complex liquids and the anisotropic liquids crystal with different directions of average molecular alignment. Since the liquid crystal is of wide applicable, and its model is of vital importance in physics, a large amount of research concerning this has appeared starting from the early 1950's, which brought us tremendous results, see \cite{ding,ding2,lin,hu,12,yu} for compressible case and \cite{13, 15} for incompressible case. The active system is quite broadly applicable in nature, describing the collective dynamics of microscopic particles, especially in biophysical system such as swarm bacteria, vibrated granular rods, see \cite{Darnton,Sanchez}. In this paper, we would mainly focus on the active nematic liquid system, a model containing the active term to the hydrodynamic theories for the nematic liquid crystal, for further detail see \cite{11}. The idea of the mathematical description has arisen in recent years. For example, \cite{rama} includes several forms of the model according to various physical phenomenon. There are several works concerning the mathematical aspect of the system, \cite{14,11, lian} proved the existence of global weak solution for compressible and incompressible cases respectively.

Observe that the system would degenerate to the compressible Navier-Stokes equation, if the concentration $c$ and the order parameter $Q$ are absent. The historical development of the research in weak solution for compressible Navier-Stokes equation is as follows. To begin with, \cite{nish1,nish2,nish3} established the global existence with restriction on the initial data. After that, by introducing the re-normalized solution to surmount the difficulty of large oscillations, \cite{Lions} gave the global existence of weak solution for adiabatic exponent $\gamma>\frac{9}{5}$ with large initial data and the appearance of vacuum. Then, \cite{Feireisl} extended the result to adiabatic exponent $\gamma>\frac{3}{2}$, which by now is the result that allows the maximum range of $\gamma$. Furthermore, in \cite{Maslow}, authors developed the deterministic result to the stochastic case, obtained the existence of global weak pathwise solution to the equation forced by additive noise, where the special form of noise allows us to transform the stochastic system into the random equation, enabling the deterministic result to be exploited. As for the existence result of the equation driven by multiplicative noise, there are also some pioneering works, in \cite {DWang} for global weak martingale solution with finite-dimensional Brownian motion, in \cite{16,Hofmanova} for global weak martingale solution with cylindrical Wiener process, in \cite{18} for stationary solution, in \cite{Breit} for local strong pathwise solution, in \cite{17} for weak martingale solution to non-isentropic, compressible Navier-Stokes equation,  {in \cite{Breit1} for weak martingale solution} to non-isentropic, compressible Navier-Stokes-Fourier equation where energy balance equation is also forced by a random heat source.

 Note that, the main difference between the deterministic and stochastic case is that there is no compactness in random element $w$ since sample space has no topology structure. Generally speaking, it might not be the case that the embedding $L^{2}(\Omega;X)\hookrightarrow L^{2}(\Omega;Y)$ is compact, even if $X\hookrightarrow\hookrightarrow Y$. As a result, the usual compactness criteria, such as the Aubin or Arzel\`{a}-Ascoli type theorems, can not be applied directly. A common method to overcome this difficulty is to invoke the Skorokhod theorem to obtain that there exists a sequence of new random variables on a new probability space, converges almost surely, and its distribution is same as the original one, consequently, the new random variables also satisfy the system on the new probability space.

 We are devoted to establishing the existence of global weak martingale solution to system (\ref{Equ1.1})-(\ref{1.3}). Our proof mainly relies on the four level approximation developed by \cite{Feireisl} and \cite{Hofmanova} which also consists of the Galerkin approximation, the artificial viscosity and the artificial pressure. Each level approximation contains the argument of compactness and identify the limit. Here, we point out that the boundedness of concentration $c$ acquired via the maximum principle and the special construction of $Q$-tensor(symmetric and traceless) play a key role in obtaining the a priori estimates (cancel certain high-order nonlinear term) and establishing the weak continuity of the effective viscous flow. Without this remarkable property of $Q$-tensor, we are not able to handle the higher order nonlinear term $\nabla\cdot(Q\triangle Q-\triangle Q Q)$. In addition, the coupled constitution of four equations makes the analysis much more complicated and more delicate arguments including identifying the stochastic integral and showing the tightness of probability measures set are necessary.  We have reserved the further details of the idea of proof in Section 2.

The rest of the paper is organized as follows. In Section 2, we recall some deterministic and stochastic preliminaries associated with system \eqref{Equ1.1} and then state our result, followed by the idea of proofs. In Section 3, we construct the global martingale solution to a modified system by taking limit on the Galerkin approximate solution. Section 4 gives the existence of global martingale solution by passing the limit as the artificial viscosity goes to zero. We build the main result by passing the limit as the artificial pressure goes to zero in Section 5. Last, we include an Appendix stating the results used frequently in this paper.

\section{Preliminaries and main result}\label{sec2}
In this section, we begin by reviewing some deterministic and stochastic preliminaries associated with system (1.1), followed by main result.

Define the inner product between two $3\times 3$ matrices $A$ and $B$
\begin{eqnarray*}
(A, B)=\int_{\mathcal{D}}A:B dx=\int_{\mathcal{D}}{\rm tr}(AB)dx,
\end{eqnarray*}
and $S_0^3\subset \mathbb{M}^{3\times 3}$ the space of $Q$-tensor
\begin{eqnarray*}
S_0^3=\left\{Q\in \mathbb{M}^{3\times 3}:~Q_{ij}=Q_{ji},~ {\rm tr}(Q)=0, ~i,j=1,2,3\right\},
\end{eqnarray*}
and the norm of a matrix using the Frobenius norm
\begin{eqnarray*}
|Q|^2={\rm tr}(Q^2)=Q_{ij}Q_{ij}.
\end{eqnarray*}
The Sobolev space of $Q$-tensor is defined by
\begin{eqnarray*}
H^{1}(\mathcal{D};S_0^3)=\left\{Q: \mathcal{D}\rightarrow S_0^3, ~{\rm and} \int_{\mathcal{D}}|\nabla Q|^2+|Q|^2dx<\infty\right\}.
\end{eqnarray*}
Set $|\nabla Q|^2=\partial_{k}Q_{ij}\partial_{k}Q_{ij}$ and $|\triangle Q|^2=\triangle Q_{ij}\triangle Q_{ij}$.

The space $C_w([0,T]; X)$ consists of all weakly continuous functions $u:[0,T]\rightarrow X$ and $u_n\rightarrow u$ in $C_w([0,T]; X)$ if and only if $\langle u_n(t),\phi\rangle\rightarrow \langle u(t), \phi\rangle$ uniformly in $t$, $\phi\in X^*$, where $X^*$ is the dual space of $X$.

Let $(\Omega,\mathcal{F},\{\mathcal{F}_{t}\}_{t\geq0},\mathbb{P}, \mathcal{W})$ be a fixed stochastic basis and $(\Omega,\mathcal{F},\mathbb{P})$ be a complete probability space. $\mathcal{W}$ is a cylindrical Wiener process defined on the Hilbert space $\mathcal{H}$, which is adapted to the complete, right continuous filtration $\{\mathcal{F}_{t}\}_{t\geq 0}$. Namely, $\mathcal{W}=\sum_{k\geq 1}e_k\beta_{k}$ with $\{e_k\}_{k\geq 1}$ being the complete orthonormal basis of $\mathcal{H}$ and $\{\beta_{k}\}_{k\geq 1}$ being a sequence of independent standard one-dimensional Brownian motions. In addition, $L_{2}(\mathcal{H},X)$ denotes the collection of Hilbert-Schmidt operators, the set of all linear operators $G$ from $\mathcal{H}$ to $X$, with the norm $\|G\|_{L_{2}(\mathcal{H},X)}^2=\sum_{k\geq 1}\|Ge_k\|_{X}^2$.

Consider an auxiliary space $\mathcal{H}_0\supset \mathcal{H}$, define by
\begin{eqnarray*}
\mathcal{H}_0=\left\{h=\sum_{k\geq 1}\alpha_k e_k: \sum_{k\geq 1}\alpha_k^2k^{-2}<\infty\right\},
\end{eqnarray*}
 with the norm $\|h\|_{\mathcal{H}_0}^2=\sum_{k\geq 1}\alpha_k^2k^{-2}$. Observe that the mapping $\Phi:\mathcal{H}\rightarrow\mathcal{H}_0$ is Hilbert-Schmidt. We also have that $\mathcal{W}\in C([0,\infty),\mathcal{H}_0)$ almost surely, see \cite{Prato}.

For an $X$-valued predictable process $f\in L^{2}(\Omega;L^{2}_{loc}([0,\infty),L_{2}(\mathcal{H},X)))$  by taking $f_{k}=fe_{k}$, the Burkholder-Davis-Gundy inequality holds
\begin{eqnarray*}
\mathbb{E}\left[\sup_{t\in [0,T]}\left\|\int_{0}^{t}fd\mathcal{W}\right\|_{X}^{p}\right]\leq c_{p}\mathbb{E}\left(\int_{0}^{T}\|f\|_{L_{2}(\mathcal{H},X)}^{2}dt\right)^{\frac{p}{2}}
=c_{p}\mathbb{E}\left(\int_{0}^{T}\sum_{k\geq 1}\|f_k\|_{X}^{2}dt\right)^{\frac{p}{2}},
\end{eqnarray*}
for any $1\leq p<\infty$.

Next, we define the global weak martingale solution of system (\ref{Equ1.1})-(\ref{1.3}).
\begin{definition}\label{def2.1} Let $\mathcal{P}$ be a Borel probability measure on $L^{\gamma}(\mathcal{D})\times L^{\frac{2\gamma}{\gamma+1}}(\mathcal{D})\times (H^{1}(\mathcal{D}))^2$ with $\gamma>\frac{3}{2}$.  $\{(\Omega, \mathcal{F}, \{\mathcal{F}_{t}\}_{t\geq 0}, \mathbb{P}), \rho,u,c,Q,\mathcal{W}\}$ is a global weak martingale solution to system (\ref{Equ1.1})-(\ref{1.3}) if the following conditions hold:\\
(i) $(\Omega, \mathcal{F}, \{\mathcal{F}_{t}\}_{t\geq 0}, \mathbb{P})$ is a stochastic basis and $\mathcal{W}$ is an $\mathcal{F}_{t}$ cylindrical Wiener process,\\
(ii)  {the processes $\rho\in C_w([0,T]; L^\gamma(\mathcal{D})), \rho u\in C_w([0,T]; L^\frac{2\gamma}{\gamma+1}(\mathcal{D})), c\in C_w([0,T]; L^2(\mathcal{D})), Q\in C_w([0,T]; H^1(\mathcal{D}))$ are} $\mathcal{F}_t$ progressively measurable, satisfying
\begin{eqnarray*}
&&\rho\in L^{p}(\Omega; L^{\infty}(0,T; L^{\gamma}(\mathcal{D}))),\\&&\rho u\in L^{p}(\Omega; L^{\infty}(0,T; L^{\frac{2\gamma}{\gamma+1}}(\mathcal{D}))), \sqrt{\rho}u\in L^{p}(\Omega; L^{\infty}(0,T; L^{2}(\mathcal{D}))), \\ &&c\in L^{p}(\Omega; L^{\infty}(0,T; L^{2}(\mathcal{D}))\cap L^{2}(0,T; H^{1}(\mathcal{D}))),\\ &&
Q\in L^{p}(\Omega; L^{\infty}(0,T; H^{1}(\mathcal{D}))\cap L^{2}(0,T; H^{2}(\mathcal{D}))),
\end{eqnarray*}
for any $1\leq p<\infty, 0<T<\infty$, \\
(iii)  {the velocity $u$ is a random distribution} adapted to $\mathcal{F}_{t}$, for the definition see \cite[Definition 2.2.13]{Breit2}, satisfying
$$u\in L^{p}(\Omega; L^{2}(0,T; H^{1}(\mathcal{D}))),$$
for any $1\leq p<\infty, 0<T<\infty$,\\
(iv) $\mathcal{P}=\mathbb{P}\circ(\rho_{0}, m_{0},c_{0}, Q_0)^{-1}$,\\
(v)  {for $\ell\in C^{\infty}(\mathcal{D}), \phi\in C^{\infty}(\mathcal{D}), \varphi\in C^{\infty}(\mathcal{D}), \psi\in C^{\infty}(\mathcal{D})$ and $t\in [0,T]$}, $\mathbb{P}$ a.s.
\begin{eqnarray*}
&&\int_{\mathcal{D}}c(t)\ell dx=\int_{\mathcal{D}}c(0)\ell dx-\int_{0}^{t}\int_{\mathcal{D}} (u\cdot\nabla)c \cdot\ell dxds-\int_{0}^{t}\int_{\mathcal{D}} \nabla c \cdot\nabla\ell dxds,\\
&&\int_{\mathcal{D}}\rho(t)\psi dx=\int_{\mathcal{D}}\rho(0)\psi dx+\int_{0}^{t}\int_{\mathcal{D}}\rho u\cdot \nabla\psi dxds,\\
&&\int_{\mathcal{D}}\rho u(t)\phi dx=\int_{\mathcal{D}}m(0)\phi dx+\int_{0}^{t}\int_{\mathcal{D}}\rho u\otimes u\cdot \nabla\phi dxds-\mu_1\int_{0}^{t}\int_{\mathcal{D}}\nabla u\cdot \nabla\phi dxds\\ &&~\qquad\qquad\qquad-(\mu_1+\mu_2)\int_{0}^{t}\int_{\mathcal{D}}{\rm div}u\cdot {\rm div}\phi dxds+\int_{0}^{t}\int_{\mathcal{D}}\rho^{\gamma}\cdot {\rm div}\phi dxds \\ &&~\qquad\qquad\qquad-\int_{0}^{t}\int_{\mathcal{D}}(({\rm F}(Q){\rm I}_3-\nabla Q\odot \nabla Q)+(Q\triangle Q-\triangle Q Q)+\sigma^*c^2 Q)\cdot\nabla \phi dxds \\&&~\qquad\qquad\qquad+\int_{0}^{t}\int_{\mathcal{D}}\rho f(\rho,\rho u,c, Q)\phi dxd\mathcal{W},\\
&&\int_{\mathcal{D}}Q(t)\varphi dx=\int_{\mathcal{D}}Q(0)\varphi dx-\int_{0}^{t}\int_{\mathcal{D}}\varphi(u\cdot \nabla)Q+\varphi Q\Psi-\varphi\Psi Q dxds\\&&~\qquad\qquad\qquad+\int_{0}^{t}\int_{\mathcal{D}}\Gamma \varphi{\rm H}(Q,c)dxds,
\end{eqnarray*}
(vi)  for all $\psi\in C^{\infty}(\mathcal{D})$ and $t\in [0,T]$,  $\rho$ satisfies the following re-normalized equation
\begin{eqnarray*}
&&\int_{\mathcal{D}}b(\rho)\psi dx=\int_{\mathcal{D}}b(\rho(0))\psi dx+\int_{0}^{t}\int_{\mathcal{D}}b(\rho)u\cdot \nabla\psi dxds\\&&\qquad\qquad\qquad+ \int_{0}^{t}\int_{\mathcal{D}}(b'(\rho)\rho-b(\rho)){\rm div}u\cdot \psi dxds,
\end{eqnarray*}
where the function $b\in C^{1}(\mathbb{R})$ satisfies $b'(z)=0$ for all $z\in \mathbb{R}$ large enough.
\end{definition}

Throughout the paper, we assume that the operator $f$ satisfies the following conditions: there exists a constant $C$ such that
\begin{eqnarray}\label{2.1*}
\sum_{k\geq 1}|f(\rho, \rho u, c, Q)e_{k}|^2\leq C \left(|\rho|^{\gamma-1}+|c, \nabla Q|^\frac{2(\gamma-1)}{\gamma}+|u|^2\right), \label{2.1}
\end{eqnarray}
and
\begin{eqnarray}\label{2.2*}
&&\sum_{k\geq 1}|(\rho_1f(\rho_1,\rho_1 u_1, c_1, Q_1 )-\rho_2f(\rho_2,\rho_2 u_2, c_2, Q_2 ))e_{k}|^2\nonumber\\
&&\leq C|\rho_1-\rho_2, \rho_1u_1-\rho_2u_2,c_1-c_2,Q_1-Q_2|^{\frac{\gamma+1}{2\gamma}}, \label{2.2}
\end{eqnarray}
where $|u,v|:=|u|+|v|$ and $|\cdot|$ stands for the absolute value. Condition (\ref{2.1}) will be used for obtaining the a priori estimate, while Condition (\ref{2.2}) will be applied to identify the limit.

In addition, we assume that initial data satisfy the following conditions for all $1\leq p< \infty$
\begin{eqnarray}
&&\rho_0\in L^p(\Omega;L^\gamma(\mathcal{D})), ~\rho_0\geq 0 ~{\rm and }~ m_0=0~ {\rm if}~ \rho_0=0,\label{2.3}\\
&&\frac{|m_0|^2}{\rho_0}\in  L^p(\Omega;L^1(\mathcal{D})), \label{2.3*}\\
&&c_0\in  L^p(\Omega;H^1(\mathcal{D}))~~ {\rm and}~ 0<\underline{c}\leq c_0\leq \bar{c}<\infty,\label{2.4}\\
 &&Q_0\in L^p(\Omega; H^{1}(\mathcal{D}; S_0^3)), \label{2.5}
\end{eqnarray}
where the lower and upper bounds $\underline{c}, \overline{c}$ are two fixed constants.

Now, we state the main result.
\begin{theorem}\label{thm2.1} Let $\gamma>\frac{3}{2}$. Suppose that the initial data $(\rho_0, m_0, c_0, Q_0)$ satisfy the assumptions (\ref{2.3})-(\ref{2.5}), and  the operator $f$ satisfies the conditions (\ref{2.1}), (\ref{2.2}). Then, there exists a global martingale weak solution to system (\ref{Equ1.1})-(\ref{1.3}) in the sense of Definition \ref{def2.1}.
\end{theorem}

 {The proof of Theorem \ref{thm2.1} mainly relies on the four-level approximation.} The first level approximation actually contains two-step approximation. Because of the difficulty technically, we need to cut-off the approximate solution such that in a sense the bound is uniform in random element corresponding to the truncated parameter $K$ inspired by \cite{Hofmanova}. Then, for any fixed $n,K$, the existence and uniqueness of the Galerkin approximate solution in small time is established using the Banach fixed point argument. The uniform bound of $c$ obtained via maximum principle and the property of symmetric and traceless of Q-tensor can be used for cancelling the high order nonlinear term, allowing us to get the uniform a priori estimates, see Lemma \ref{lem3.1}. With the a priori estimates established, we can extend the existence time to global. Next, we let $K\rightarrow \infty$ to establish the Galerkin approximate solution for fixed $n$. Define a stopping time $\tau_{K}$, on interval $[0, \tau_K)$, we can define the approximate solution using the uniqueness of solution and the monotonicity of the sequence of stopping time $\tau_{K}$. Using the a priori estimates, it holds $\lim_{K\rightarrow\infty}\tau_{K}=T$, $\mathbb{P}$ a.s. This means no blow-up arises in finite time. Except for this level approximation, all other three level approximations contain the compactness argument. Owing to the complex structure of the system, we have to work with weak compactness, the classical Skorokhod theorem is replaced by the Skorokhod-Jakubowski theorem applying to quasi-Polish space.  To employ the theorem, it is necessary to show the tightness of the set of probability measures generated by the distribution of approximate solution, which can be achieved using the Aubin-Lions Lemma, the a priori estimates and the delicate analysis. After obtaining the compactness of the new processes on the new probability space, we identify all the nonlinear term by passing limit as $n\rightarrow\infty$.

In the second level approximation, the solution obtained in first level approximation will be used as approximate solution. Here the boundedness of $\sqrt{\epsilon}\rho_{\epsilon,\delta} \in L^{p}(\Omega; L^{2}(0,T;H^{1}(\mathcal{D})))$ is not helpful in getting the strong convergence of density, which makes it difficult to identify the nonlinear term with respect to $\rho$ (the pressure term and stochastic term). Following the ideas from \cite{Feireisl,Lions}, we shall show the strong convergence of density by improving the integrability of density, establishing the weak continuity of the effective viscous flow (here the symmetric of $Q$-tensor plays a crucial role, see Step 1 in Section 4) and using the Minty trick.

Following the same line as the second level approximation, improving the integrability of density, establishing the weak continuity of the effective viscous flow, introducing the re-normalized solution to control the large oscillations and the truncated technique are required to get the strong convergence of density. Here, the proof is standard, we only give the necessary tightness argument and  improve the integrability of density, for further details, we refer the reader to \cite{Hofmanova,16,DWang}.

\maketitle
\section{The existence of martingale solution for $n\rightarrow \infty$}

In this section, we are devoted to building the existence of global weak martingale solution to the following modified viscous system
\begin{equation}\label{Equ3.1}
    \begin{cases} \partial_{t}c+(u\cdot \nabla)c=\triangle c,\\
\partial_{t}\rho+{\rm div}(\rho u)=\epsilon \triangle \rho,\\
\partial_{t}(\rho u)+{\rm div}(\rho u\otimes u)+\nabla (\rho^\gamma+\delta{\rho}^\beta)+\epsilon\nabla\rho\cdot\nabla u
 =\mu_1\Delta u+(\mu_1+\mu_2)\nabla ({\rm div} u)\\ \qquad\qquad+\sigma^*\nabla\cdot(c^2Q)+\nabla\cdot({\rm F}(Q){\rm I}_3-\nabla Q\odot \nabla Q)+\nabla\cdot(Q\triangle Q-\triangle Q Q)\\ \qquad\qquad +\rho f(\rho,\rho u, c, Q)\frac{d\mathcal{W}}{dt},\\
\partial_{t}Q+(u\cdot \nabla)Q+Q\Psi-\Psi Q=\Gamma {\rm H}(Q,c),\\
    \end{cases}
    \end{equation}
where $\epsilon, \delta>0$ and $\beta>\{6, \gamma\}$, with the boundary conditions
\begin{eqnarray}\label{3.2}
\frac{\partial \rho}{\partial n}\bigg|_{\partial \mathcal{D}}=\frac{\partial c}{\partial n}\bigg|_{\partial \mathcal{D}}=\frac{\partial Q}{\partial n}\bigg|_{\partial \mathcal{D}}=0, ~u|_{\partial \mathcal{D}}=0,
\end{eqnarray}
and the modified initial data
\begin{eqnarray}
&&\rho(0)=\rho_{0,\delta}\in L^p(\Omega; C^{2+\alpha}(\mathcal{D})),\label{3.3}\\ && \rho u(0)=m_{0,\delta}\in L^p(\Omega; C^{2}(\mathcal{D})), \label{3.3*}\\~ &&Q(0)=Q_0\in L^p(\Omega; H^{1}(\mathcal{D}; S_0^3)), \label{3.4}\\
&&c(0)=c_0\in L^p(\Omega; H^1(\mathcal{D}))~ {\rm and}~ 0<\underline{c}\leq c_0\leq \bar{c}<\infty.\label{3.5}
\end{eqnarray}
Moreover, assume that the initial data $\rho_{0,\delta}$ satisfies the following conditions
\begin{eqnarray}
0<\delta\leq \rho_{0,\delta}\leq \delta^{-\frac{1}{\beta}}< \infty, ~~(\rho_{0,\delta})_m\leq M,~~\rho_{0,\delta}\rightarrow \rho_{0}~ {\rm in}~ L^{\gamma}(\mathcal{D})~ {\rm as} ~ \delta\rightarrow 0, \label{3.6}
\end{eqnarray}
where the $(\rho_{0,\delta})_m$ denotes the mean value of $\rho_{0,\delta}$ in domain $\mathcal{D}$, and
\begin{eqnarray*}
m_{0,\delta}=h_{\delta}\sqrt{\rho_{0,\delta}},
\end{eqnarray*}
$h_{\delta}$ is defined as follows. Let
\begin{eqnarray*}
\tilde{m}_{0,\delta}=\left\{\begin{array}{ll}
m_{0}\sqrt{\frac{\rho_{0,\delta}}{\rho_{0}}},~~ {\rm if} ~~\rho_{0}>0,\\
0, ~~{\rm if}~~ \rho_{0}=0.\\
\end{array}\right.
\end{eqnarray*}
According to the assumption (\ref{2.3*}), we have $\frac{|\tilde{m}_{0,\delta}|^{2}}{\rho_{0,\delta}}\in L^{p}(\Omega; L^{1}(\mathcal{D}))$ uniformly in $\delta$ for $1\leq p<\infty$. Therefore, we can find $C^{2}(\overline{\mathcal{D}})$-valued random variables $h_{\delta}$ such that
\begin{eqnarray*}
\left\|\frac{\tilde{m}_{0,\delta}}{\sqrt{\rho_{0,\delta}}}-h_{\delta}\right\|_{L^2(\mathcal{D})}\leq \delta.
\end{eqnarray*}
Then, we have
\begin{eqnarray*}
&&\frac{|m_{0,\delta}|^{2}}{\rho_{0,\delta}}\in L^{p}(\Omega; L^{1}(\mathcal{D}))~{\rm  uniformly ~in}~ \delta,\\
&&\frac{m_{0,\delta}}{\sqrt{\rho_{0,\delta}}}\rightarrow \frac{m_{0}}{\sqrt{\rho_{0}}}~ {\rm in}~L^{p}(\Omega; L^{2}(\mathcal{D}))~ {\rm for} ~1\leq p<\infty.
\end{eqnarray*}
Let $\mathcal{P}_{\delta}=\mathbb{P}\circ(\rho_{0,\delta}, m_{0,\delta}, c_{0}, Q_0)^{-1}$. According to the construction, we have $\mathcal{P}_{\delta}$ is Borel probability measure on $C^{2+\alpha}(\mathcal{D})\times C^{2}(\mathcal{D})\times (H^{1}(\mathcal{D}))^2$, satisfying
\begin{eqnarray*}
&&\mathcal{P}_{\delta}\big\{(\rho,\rho u, c_, Q)\in C^{2+\alpha}(\mathcal{D})\times C^{2}(\mathcal{D})\times (H^{1}(\mathcal{D}))^2;\\
&&\qquad\qquad {\rm and }~0< \delta\leq\rho\leq \delta^{-\frac{1}{\beta}}<\infty,~(\rho)_m\leq M,~ 0<\underline{c}\leq c_0\leq \bar{c}<\infty, ~Q\in S_0^3\big\}=1,
\end{eqnarray*}
and
\begin{eqnarray} \label{3.7}
\int_{C^{2+\alpha}(\mathcal{D})\times C^{2}(\mathcal{D})\times (H^{1}(\mathcal{D}))^2}\left\|\frac{|\rho u|^{2}}{2\rho}+\frac{1}{\gamma-1}\rho^{\gamma}+|\nabla c, \nabla Q|^{2}\right\|_{L^{1}(\mathcal{D})}^{p}d\mathcal{P}_{\delta}\nonumber\\ \rightarrow \int_{L^{\gamma}(\mathcal{D})\times L^{\frac{2\gamma}{\gamma+1}}(\mathcal{D})\times (H^{1}(\mathcal{D}))^2}\left\|\frac{|\rho u|^{2}}{2\rho}+\frac{1}{\gamma-1}\rho^{\gamma}
+|\nabla c, \nabla Q|^{2}\right\|_{L^{1}(\mathcal{D})}^{p}d\mathcal{P},
\end{eqnarray}
where $\mathcal{P}$ is a Borel probability measure on $L^\gamma(\mathcal{D})\times L^\frac{2\gamma}{\gamma+1}(\mathcal{D})\times (H^1(\mathcal{D}))^2$.

The proof will be divided into three subsections. For the first subsection, we establish the Galerkin approximate solution and the a priori estimates. Then, the compactness result is obtained in second subsection. In the third subsection, we get the existence of global weak martingale solution by taking the limit as $n\rightarrow\infty$.

\subsection{
The approximate solution and a priori estimates}

First of all,  {we build the approximate solution to system (\ref{Equ3.1})-(\ref{3.5}) for fixed $\varepsilon>0$ and $\delta>0$, we would need an extra approximation layer compared to the deterministic case following the ideas of \cite{Hofmanova}.} At the beginning, we introduce the following well-posedness results taken from \cite{11,Lunardi}.
\begin{lemma} \label{lem3.2}Suppose that the initial data $\rho_0$ satisfies (\ref{3.3}).  If $u\in C([0,T];C^2(\mathcal{D}))$ with $u|_{\partial \mathcal{D}}=0$, then there exists a mapping $\mathcal{S}=\mathcal{S}(u)$
\begin{eqnarray*}
\mathcal{S}: C([0,T];C^2(\overline{\mathcal{D}}))\rightarrow C([0,T];C^{2+\alpha}(\overline{\mathcal{D}})),
\end{eqnarray*}
with the following properties:

(1) $\rho=\mathcal{S}(u)$ is a unique classical solution of system $(\ref{Equ3.1})_2, (\ref{3.3}), (\ref{3.6})$ with the mapping $\mathcal{S}$ continuous on bounded subset of $C([0,T];C^{2}(\mathcal{D}))$.

 (2) It holds
\begin{eqnarray*}
\delta{\rm exp}\left(-\int_{0}^{t}\|{\rm div} u\|_{L^{\infty}}dr\right)\leq \rho \leq \delta^{-\frac{1}{\beta}}{\rm exp}\left(\int_{0}^{t}\|{\rm div} u\|_{L^{\infty}}dr\right),
\end{eqnarray*}
for all $t\in [0,T]$.
\end{lemma}
\begin{lemma} \label{lem3.3}For each $u\in C([0,T];C^2(\mathcal{D}))$ with $u|_{\partial \mathcal{D}}=0$, then there exists a unique strong solution $(c, Q)\in [L^\infty(0,T; H^1(\mathcal{D}))\cap L^2(0,T;H^2(\mathcal{D}))]^2$ to the system
\begin{eqnarray}\label{3.15*}
\left\{\begin{array}{ll}
\partial_{t}c+(u\cdot \nabla)c=\triangle c,\\
\partial_{t}Q+(u\cdot \nabla)Q+Q\Psi-\Psi Q=\Gamma {\rm H}(Q,c),\\
Q(0)=Q_0\in  H^{1}(\mathcal{D};S_0^3)~\mbox{a.e.} ~{\rm and }~\frac{\partial Q}{\partial n}\bigg|_{\partial \mathcal{D}}=0,\\
c(0)=c_0\in  H^1(\mathcal{D})~{\rm and}~ 0<\underline{c}\leq c_0\leq \bar{c}<\infty,~\frac{\partial c}{\partial n}\bigg|_{\partial \mathcal{D}}=0. \\
\end{array}\right.
\end{eqnarray}
Moreover, we have $0<\underline{c}\leq c\leq \bar{c}<\infty$. Furthermore, the mapping
\begin{eqnarray*}
\widetilde{\mathcal{S}}:  C([0,T]; C^{2}(\mathcal{D}))\rightarrow \left[L^{\infty}(0,T; H^1(\mathcal{D}))\cap L^{2}(0,T; H^{2}(\mathcal{D}))\right]^2,
\end{eqnarray*}
is continuous on set $B_R:=\left\{u\in C([0,T]; C^{2}(\mathcal{D})); \|u\|_{C([0,T]; C^{2}(\mathcal{D}))}\leq R\right\}$ and $Q\in S_0^3$ \mbox{a.e.}
\end{lemma}
{\bf Remark}.  The proof of $Q\in S_0^3$ a.e. depends on the uniqueness of solution to system (\ref{3.15*}), therefore, we have to lift the regularity of initial data $c$ to $H^1(\mathcal{D})$ rather than $L^2(\mathcal{D})$ in establishing the existence of strong solutions, making further effort to achieve the uniqueness. It will also be used for defining the Galerkin approximate solution for fixed truncation parameter $K$ introduced later.

With these results in hand, we now find the approximate velocity field $u_{n}$ satisfying the integral equation
\begin{eqnarray}\label{3.15}
&& \int_{\mathcal{D}}\rho u_{n}\phi dx-\int_{\mathcal{D}}m_{0}\phi dx\nonumber\\
&&=-\int_{0}^{t}\int_{\mathcal{D}}({\rm div} (\rho u_{n}\otimes u_{n})-\mu_1\triangle u_{n}-(\mu_1+\mu_2)\nabla ({\rm div} u_{n})+\nabla\rho^{\gamma}+\delta\nabla\rho^{\beta})\phi dxds\nonumber\\&&\quad+ \int_{0}^{t}\int_{\mathcal{D}}\sigma^*\nabla\cdot(c^2Q)\phi dxds-\int_{0}^{t}\int_{\mathcal{D}}\epsilon \phi\nabla \rho\cdot\nabla u_{n}  dxds\nonumber\\ &&\quad+\int_{0}^{t}\int_{\mathcal{D}}\nabla\cdot({\rm F}(Q){\rm I}_3-\nabla Q\odot \nabla Q)\phi+\nabla\cdot(Q\triangle Q-\triangle Q Q)\phi dxds\nonumber\\&& \quad+ \int_{0}^{t}\int_{\mathcal{D}}\sqrt{\rho}P_n(\sqrt{\rho} f(\rho,\rho u_n, c, Q))\phi dxd\mathcal{W},
\end{eqnarray}
for $t\in [0,T]$, and $\phi$ belongs to the finite dimensional space $X_n$ defined by ${\rm span} \{h_{i}\}_{i=1}^{n}$, where the family of smooth functions $\{h_{i}\}_{i=1}^{n}$ is an orthonormal basis of $H^{1}(\mathcal{D})$, and $P_{n}$ be the orthogonal projection from $L^{2}(\mathcal{D})$ into $X_{n}$.

Define by the operator $\mathcal{M}[\rho]:X_n\rightarrow X^*_n$
\begin{eqnarray*}
\langle\mathcal{M}[\rho]u,v\rangle=\int_{\mathcal{D}}\rho u\cdot v dx~{\rm for}~u,v\in X_n.
\end{eqnarray*}
From the definition, we know $\mathcal{M}[\rho]u=P_n(\rho u)$. Moreover, $\mathcal{M}[\rho]$ is a positive symmetric operator with following properties,
\begin{eqnarray}\label{3.11*}
\|\mathcal{M}[\rho]^{-1}\|_{\mathcal{L}(X^*_n, X_n)}\leq \|\rho^{-1}\|_{C(\mathcal{D})},
\end{eqnarray}
and
\begin{eqnarray}\label{3.12*}
\|\mathcal{M}[\rho_1]^{-1}-\mathcal{M}[\rho_2]^{-1}\|_{\mathcal{L}(X^*_n, X_n)}\leq \|\rho_1^{-1}\|_{C(\mathcal{D})}\|\rho_2^{-1}\|_{C(\mathcal{D})}\|\rho_1-\rho_2\|_{L^1(\mathcal{D})}.
\end{eqnarray}

Introduce the functional $\mathcal{N}[\rho,u,c, Q](\varphi)$ by
\begin{eqnarray*}
&&\mathcal{N}[\rho,u,c,Q](\varphi)=\int_{\mathcal{D}}(-{\rm div}(\rho u\otimes u)- \nabla(\rho^\gamma+\delta\rho^\beta)+\mu_1\Delta u+(\mu_1+\mu_2)\nabla ({\rm div} u)\\&&\qquad\qquad\qquad\qquad-\epsilon\nabla \rho\cdot\nabla u+\nabla\cdot({\rm F}(Q){\rm I}_3-\nabla Q\odot \nabla Q)\\ &&\qquad\qquad\qquad\qquad+\nabla\cdot(Q\triangle Q-\triangle Q Q)+\sigma^*\nabla\cdot(c^2Q))\cdot \varphi dx,
\end{eqnarray*}
for all $\varphi\in X_n$.  {Due to the technical difficulty, we need further truncation to $u_n$. Following the idea of \cite{Hofmanova}}, define the $C^\infty$-smooth cut-off function
\begin{eqnarray*}
\xi_K(z)=\left\{\begin{array}{ll}
1,~ |z|\leq K,\\
0,~ |z|>2K.
\end{array}\right.
\end{eqnarray*}
Let $u^K=\sum_{i=1}^{n}\xi_K(\alpha_i)\alpha_ih_i$, then we have $\|u^K\|_{C([0,T];C^2(\mathcal{D}))}\leq 2K$ and the truncation operator ${\rm Tr}: u\rightarrow u^K$ satisfies
\begin{eqnarray}\label{3.13*}
{\rm Tr}: X_n\rightarrow X_n, ~~~{\rm and } ~~\|{\rm Tr}(u)-{\rm Tr}(v)\|_{X_n}\leq C(n)\|u-v\|_{X_n}.
\end{eqnarray}
 {Then, we rewrite (\ref{3.15}) as}
\begin{align}\label{3.16}
u_n(t)&=\mathcal{M}^{-1}\left[\mathcal{S}(u_n^K)\right]\bigg(m_0^*+\int_0^t\mathcal{N}[\mathcal{S}(u_n^K),u_n^K, \widetilde{\mathcal{S}}(u_n^K)]ds\nonumber\\ &\quad+\mathrm{Tr}\left(\int_{0}^{t}\mathcal{M}^\frac{1}{2}(\mathcal{S}(u_n^K))P_n(\sqrt{\mathcal{S}(u_n^K)} f(\mathcal{S}(u_n^K),\mathcal{S}(u_n^K) u_n^K,\widetilde{\mathcal{S}}(u_n^K)))d\mathcal{W}\right)\bigg).
\end{align}
Here, the stochastic integral should be understood evolving on space $X_n^*$.

The mapping $\mathcal{Y}$ from $L^{2}(\Omega; C([0,T]; C^{2}(\mathcal{D})))$ into itself is defined by the right hand side of (\ref{3.16}). For fixed $n, K$, we can show the mapping $\mathcal{Y}$ is the contraction for $T^*$ small enough, for further details see \cite{12,14} for the deterministic part. Next, we give the estimate of stochastic term. Using \eqref{3.13*} and the triangle inequality, we have
\begin{align}\label{3.14*}
& {\quad\mathbb{E}\sup_{0\leq t\leq T^*}\bigg\|\mathcal{M}^{-1}(\mathcal{S}(u_n^K))}\nonumber\\ &\qquad\qquad\quad\times\mathrm{Tr}\left(\int_{0}^{t}\mathcal{M}^\frac{1}{2}(\mathcal{S}(u_n^K)) P_n(\sqrt{\mathcal{S}(u_n^K)} f(\mathcal{S}(u_n^K),\mathcal{S}(u_n^K) u_n^K,\widetilde{\mathcal{S}}(u_n^K)))d\mathcal{W}\right)\nonumber\\&\quad-\mathcal{M}^{-1}(\mathcal{S}(v_n^K))
\mathrm{Tr}\left(\int_{0}^{t}\mathcal{M}^\frac{1}{2}(\mathcal{S}(v_n^K))P_n(\sqrt{\mathcal{S}(v_n^K)} f(\mathcal{S}(v_n^K),\mathcal{S}(v_n^K) v_n^K,\widetilde{\mathcal{S}}(v_n^K)))d\mathcal{W}\right)\bigg\|^2_{X_n}\nonumber\\
&\leq \mathbb{E}\sup_{0\leq t\leq T^*}\left\|\mathcal{M}^{-1}(\mathcal{S}(u_n^K))-\mathcal{M}^{-1}(\mathcal{S}(v_n^K))\right\|^2_{\mathcal{L}(X^*_n, X_n)}\nonumber\\
&\qquad\quad\times\left\|\mathrm{Tr}\left(\int_{0}^{t}\mathcal{M}^\frac{1}{2}(\mathcal{S}(u_n^K)) P_n(\sqrt{\mathcal{S}(u_n^K)} f(\mathcal{S}(u_n^K),\mathcal{S}(u_n^K) u_n^K,\widetilde{\mathcal{S}}(u_n^K)))d\mathcal{W}\right)\right\|^2_{X_n}\nonumber\\&\quad+C(n)\mathbb{E}\sup_{0\leq t\leq T^*}\left\|\mathcal{M}^{-1}(\mathcal{S}(v_n^K))\right\|^2_{\mathcal{L}(X^*_n, X_n)}\nonumber\\ &\qquad\quad\times\bigg\|\int_{0}^{t}\mathcal{M}^\frac{1}{2}(\mathcal{S}(u_n^K)) P_n(\sqrt{\mathcal{S}(u_n^K)} f(\mathcal{S}(u_n^K),\mathcal{S}(u_n^K) u_n^K,\widetilde{\mathcal{S}}(u_n^K)))d\mathcal{W}\nonumber\\&\qquad\quad\quad-\int_{0}^{t}\mathcal{M}^\frac{1}{2}(\mathcal{S}(v_n^K)) P_n(\sqrt{\mathcal{S}(v_n^K)} f(\mathcal{S}(v_n^K),\mathcal{S}(v_n^K) v_n^K,\widetilde{\mathcal{S}}(v_n^K)))d\mathcal{W}\bigg\|^2_{X_n}\nonumber\\
&=:L_1+L_2.
\end{align}
For $L_1$, using \eqref{3.12*}, Lemma \ref{lem3.2}(2) and the property of operator $\mathrm{Tr}$, we have
\begin{align}\label{3.15&}
L_1&\leq T^*C(n,K,T^*)\mathbb{E}\sup_{0\leq t\leq T^*}\|u_n^K-v_n^K\|^2_{X_n}.
\end{align}
For $L_2$, using \eqref{3.11*}, the boundedness of $\mathcal{S}(u_n^K)$, we have from the Burkholder-Davis-Gundy inequality
\begin{align}
L_2&\leq C(K,\delta,n)\mathbb{E}\sup_{0\leq t\leq T^*}\bigg\|\int_{0}^{t}\mathcal{M}^\frac{1}{2}(\mathcal{S}(u_n^K)) P_n(\sqrt{\mathcal{S}(u_n^K)} f(\mathcal{S}(u_n^K),\mathcal{S}(u_n^K) u_n^K,\widetilde{\mathcal{S}}(u_n^K)))d\mathcal{W}\nonumber\\
&\qquad\qquad\qquad-\int_{0}^{t}\mathcal{M}^\frac{1}{2}(\mathcal{S}(v_n^K)) P_n(\sqrt{\mathcal{S}(v_n^K)} f(\mathcal{S}(v_n^K),\mathcal{S}(v_n^K) v_n^K,\widetilde{\mathcal{S}}(v_n^K)))d\mathcal{W}\bigg\|^2_{X_n}\nonumber\\
&\leq C(K,\delta,n)\mathbb{E}\int_{0}^{T^*}\sum_{k\geq 1}\bigg\|\mathcal{M}^\frac{1}{2}(\mathcal{S}(u_n^K)) P_n(\sqrt{\mathcal{S}(u_n^K)} f(\mathcal{S}(u_n^K),\mathcal{S}(u_n^K) u_n^K,\widetilde{\mathcal{S}}(u_n^K))e_k)\nonumber\\&\qquad\qquad\qquad\quad-\mathcal{M}^\frac{1}{2}(\mathcal{S}(v_n^K)) P_n(\sqrt{\mathcal{S}(v_n^K)} f(\mathcal{S}(v_n^K),\mathcal{S}(v_n^K) v_n^K,\widetilde{\mathcal{S}}(v_n^K))e_k)\bigg\|^2_{X_n}ds\nonumber\\
&\leq C(K,\delta,n)\mathbb{E}\Bigg(\int_{0}^{T^*}\sum_{k\geq 1}\bigg\|(\mathcal{M}^\frac{1}{2}(\mathcal{S}(u_n^K))-\mathcal{M}^\frac{1}{2}(\mathcal{S}(v_n^K)))\nonumber \\ &\qquad\qquad\qquad\qquad\quad\times P_n(\sqrt{\mathcal{S}(u_n^K)}f(\mathcal{S}(u_n^K),\mathcal{S}(u_n^K) u_n^K,\widetilde{\mathcal{S}}(u_n^K))e_k)\bigg\|^2_{X_n}ds\nonumber\\&\qquad\qquad\qquad+\int_{0}^{T^*}\sum_{k\geq 1}\bigg\|\mathcal{M}^\frac{1}{2}(\mathcal{S}(v_n^K)) P_n(\sqrt{\mathcal{S}(v_n^K)}f(\mathcal{S}(v_n^K),\mathcal{S}(v_n^K) v_n^K,\widetilde{\mathcal{S}}(v_n^K))e_k\nonumber\\&\qquad\qquad\quad\qquad-\sqrt{\mathcal{S}(u_n^K)}f(\mathcal{S}(u_n^K),\mathcal{S}(u_n^K) u_n^K,\widetilde{\mathcal{S}}(u_n^K))e_k)\bigg\|^2_{X_n}ds\Bigg)\nonumber\\
&=:L_{21}+L_{22}.
\end{align}
For $L_{21}$, using the continuity of $\mathcal{S}(u_n^K)$, the equivalence of norms on finite dimensional space, condition \eqref{2.1*}, the boundedness of $u_n^K, \mathcal{S}(u_n^K)$ and Lemma \ref{lem3.3},  we have
\begin{align}
L_{21}&\leq C(K,\delta,n)\mathbb{E}\int_{0}^{T^*}\left\|\mathcal{M}^\frac{1}{2}(\mathcal{S}(u_n^K))-\mathcal{M}^\frac{1}{2}(\mathcal{S}(v_n^K))\right\|^2_{\mathcal{L}(X_n, X^*_n)}\nonumber\\ &\qquad\quad\qquad\qquad\times\sum_{k\geq 1}\left\|P_n(\sqrt{\mathcal{S}(u_n^K)}f(\mathcal{S}(u_n^K),\mathcal{S}(u_n^K) u_n^K,\widetilde{\mathcal{S}}(u_n^K))e_k)\right\|_{L^2}^2ds\Bigg)\nonumber\\
&\leq T^*C(n,K,T^*,\delta)\mathbb{E}\Bigg(\sup_{0\leq t\leq T^*}\|u_n^K-v_n^K\|^2_{X_n}\nonumber\\
&\qquad\qquad\qquad\quad\times \int_{0}^{T^*} \sum_{k\geq 1}\left\|\sqrt{\mathcal{S}(u_n^K)} f(\mathcal{S}(u_n^K),\mathcal{S}(u_n^K) u_n^K,\widetilde{\mathcal{S}}(u_n^K))e_k\right\|_{L^2}^2ds\Bigg)\nonumber\\
&\leq T^*C(n,K,T^*,\delta)\mathbb{E}\Bigg(\sup_{0\leq t\leq T^*}\|u_n^K-v_n^K\|^2_{X_n}\nonumber\\ &\qquad\qquad\qquad\quad\times \int_{0}^{T^*}\int_{\mathcal{D}}|\rho_n^K|^{\gamma}+\rho_n^K|c_n^K, \nabla Q_n^K|^\frac{2(\gamma-1)}{\gamma}+\rho_n^K|u_n^K|^2dxds\Bigg)\nonumber\\
& \leq T^*C(n,K,T^*,\delta)\mathbb{E}\Bigg(\sup_{0\leq t\leq T^*}\|u_n^K-v_n^K\|^2_{X_n}\nonumber\\ &\qquad\qquad\qquad\quad\times \int_{0}^{T^*}\int_{\mathcal{D}}|\rho_n^K|^{\gamma}+|c_n^K, \nabla Q_n^K|^2+\rho_n^K|u_n^K|^2dxds\Bigg)\nonumber\\
&\leq (T^*)^2C(n,K,T^*,\delta)\mathbb{E}\sup_{0\leq t\leq T^*}\|u_n^K-v_n^K\|^2_{X_n}.
\end{align}
For $L_{22}$, using the continuity of $\mathcal{S}(u_n^K), \widetilde{\mathcal{S}}(u_n^K)$(see Lemma \ref{lem3.3}), the boundedness of $\mathcal{S}(u^K)$, condition \eqref{2.2*} and the equivalence of norms on finite dimensional space, we also have
\begin{align}\label{3.18*}
L_{22}\leq T^*C(n,K,T^*)\mathbb{E}\sup_{0\leq t\leq T^*}\|u_n^K-v_n^K\|^2_{X_n}.
\end{align}
Then, taking into account of \eqref{3.14*}-\eqref{3.18*}, we infer that there exists a sequence of approximate solutions $u^K_{n}\in L^{2}(\Omega; C([0,T_{*}]; X_{n}))$ to equation (\ref{3.16}) for small time $T^{*}$ by the Banach fixed point theorem. Here we first assume that the a priori estimates (\ref{3.8}) hold uniformly in $n, K$ which allows us to extend the existence time $T^*$ to $T$ for any $T>0$. Namely, we proved the existence and uniqueness of solution $u_n^K\in L^{2}(\Omega; C([0,T]; C^{2}(\mathcal{D})))$ to equation (\ref{3.16}) for fixed $n, K$.

Next, we build the global existence of Galerkin approximate solution to system (\ref{Equ3.1}) for any fixed $n$ by letting $K\rightarrow \infty$. Define the stopping time,
\begin{eqnarray*}
&&\tau_K=\inf\bigg\{t\geq 0;\sup_{s\in [0,t]}\|u_n^K(s)\|_{L^2}\\&&\quad\quad+\sup_{s\in [0,t]}\left\|\int_{0}^{s} \mathcal{M}^\frac{1}{2}(\mathcal{S}(u_n^K)) P_n(\sqrt{\mathcal{S}(u_n^K)}f(\mathcal{S}(u_n^K),\mathcal{S}(u_n^K) u_n^K,\widetilde{\mathcal{S}}(u_n^K)))d\mathcal{W}\right\|_{L^2}\geq K\bigg\}.
\end{eqnarray*}
Observe that the sequence of the stopping time $\tau_K$ is increasing. Define $\rho_n^K=\mathcal{S}(u_n^K)$, $(c_n^K, Q_n^K)=\widetilde{\mathcal{S}}(u_n^K)$, then $(\rho_n^K, u_n^K, c_n^K, Q_n^K)$ is the unique solution to system (\ref{Equ3.1}). Using the monotonicity of the stopping time and the uniqueness of solution, for $K_1\leq K_2$, we have $(\rho_n^{K_1}, u_n^{K_1}, c_n^{K_1}, Q_n^{K_1})=(\rho_n^{K_2}, u_n^{K_2}, c_n^{K_2}, Q_n^{K_2})$ on $[0, \tau_{K_1})$. Therefore, we could define the solution $(\rho_n, u_n, c_n, Q_n)=(\rho_n^K, u_n^K, c_n^K, Q_n^K)$ on interval $[0, \tau_{K})$. In order to extend the existence time to $[0,T]$, we show that
\begin{eqnarray*}
\mathbb{P}\left\{\sup_{K\in \mathbb{N}^+}\tau_K=T\right\}=1.
\end{eqnarray*}
Since the stopping time $\tau_{K}$ is increasing, we have
\begin{align*}
&\mathbb{P}\left\{\sup_{K\in \mathbb{N}^+}\tau_K<T\right\}<\mathbb{P}\{\tau_K<T\}\\
&\leq \mathbb{P}\left\{\sup_{t\in [0,T]}\|u_n^K\|_{L^2}>\frac{K}{2}\right\}\\
&\quad+\mathbb{P}\left\{\sup_{t\in [0,T]}\left\|\int_{0}^{t}\mathcal{M}^\frac{1}{2}(\mathcal{S}(u_n^K)) P_n(\sqrt{\mathcal{S}(u_n^K)} f(\mathcal{S}(u_n^K),\mathcal{S}(u_n^K) u_n^K,\widetilde{\mathcal{S}}(u_n^K)))d\mathcal{W}\right\|_{L^2}>\frac{K}{2}\right\}\\&=:J_1+J_2.
\end{align*}
Using the Burkholder-Davis-Gundy inequality, the Chebyshev inequality and the equivalence of norms on finite-dimensional space, and the embedding $H^{-l}(\mathcal{D})\hookrightarrow L^{1}(\mathcal{D})$ for $l>\frac{3}{2}$, and the condition (\ref{2.1}), the bound (\ref{3.8}), we have
\begin{align}\label{3.19*}
J_2&\leq \frac{2}{K}\mathbb{E}\left(\sup_{t\in [0,T]}\left\|\int_{0}^{t}\mathcal{M}^\frac{1}{2}(\mathcal{S}(u_n^K)) P_n\sqrt{\mathcal{S}(u_n^K)} f(\mathcal{S}(u_n^K),\mathcal{S}(u_n^K) u_n^K,\widetilde{\mathcal{S}}(u_n^K))d\mathcal{W}\right\|_{L^2}\right)\nonumber\\
&\leq  \frac{C}{K}\mathbb{E}\left(\sup_{t\in [0,T]}\left\|\int_{0}^{t} \mathcal{M}^\frac{1}{2}(\mathcal{S}(u_n^K)) P_n\sqrt{\mathcal{S}(u_n^K)} f(\mathcal{S}(u_n^K),\mathcal{S}(u_n^K) u_n^K,\widetilde{\mathcal{S}}(u_n^K))d\mathcal{W}\right\|_{H^{-l}}\right)\nonumber\\
&\leq \frac{C}{K}\mathbb{E}\int_0^T\sum_{k\geq 1}\left\|\mathcal{M}^\frac{1}{2}(\mathcal{S}(u_n^K)) P_n(\sqrt{\mathcal{S}(u_n^K)}f(\mathcal{S}(u_n^K),\mathcal{S}(u_n^K) u_n^K,\widetilde{\mathcal{S}}(u_n^K))e_k)\right\|_{L^1}^2dt\nonumber\\
&\leq \frac{C}{K}\mathbb{E}\int_0^T\int_{\mathcal{D}}\sum_{k\geq 1}\left|\sqrt{\mathcal{S}(u_n^K)}f(\mathcal{S}(u_n^K),\mathcal{S}(u_n^K) u_n^K,\widetilde{\mathcal{S}}(u_n^K))e_k\right|^2dx\int_{\mathcal{D}}\mathcal{S}(u_n^K)dxdt\nonumber\\ & \leq \frac{C}{K} \left(\rho_n^K(0)\right)_m\mathbb{E}\int_0^T\int_{\mathcal{D}}|\rho_n|^\gamma+\rho_n|u_n|^2+|c_n, \nabla Q_n|^2dxdt<  \frac{C}{K},
\end{align}
where $C$ is independent of $K$, leading to $J_2\rightarrow 0$ as $K\rightarrow\infty$. Corollary 3.2 in \cite{Hofmanova} given $J_1\rightarrow 0$ as $K\rightarrow\infty$. Therefore, passing $K\rightarrow \infty$, we have
$$\mathbb{P}\left\{\sup_{K\in \mathbb{N}^+}\tau_K<T\right\}=0.$$
This means that no blow up appears in a finite time, we could extend the existence time to $[0,T]$ for any $T>0$.

We next establish the necessary a priori estimates of approximate solution. To simplify the notation, we replace $(\rho_n^K, \rho_n^K u_n^K, c_n^K, Q_n^K)$ by $(\rho, \rho u, c, Q)$.
\begin{lemma}\label{lem3.1} Suppose that $(\rho, \rho u, c, Q)$ is the Galerkin approximate solution to system (\ref{Equ3.1})-(\ref{3.5}), and $f$ satisfies the condition (\ref{2.1}). Then, there exists a constant $C$ which is independent of $n, K$ but depends on $(\mu_1, \mu_2, \sigma^*, c_*,b, T, p, \Gamma)$ and initial data such that for all $1\leq p< \infty$
\begin{eqnarray}\label{3.8}
&&\mathbb{E}\left[\sup_{t\in [0,T]}\left(\|c, \nabla Q, \sqrt{\rho}u\|_{L^2}^2+\|Q\|_{L^4}^4+\frac{1}{\gamma-1}\|\rho\|_{L^\gamma}^\gamma+\frac{\delta}{\beta-1}\|\rho\|_{L^\beta}^\beta\right)\right]^p\nonumber\\
&&+\mathbb{E}\left(\int_0^T\|\nabla c, \nabla u,  \triangle Q, {\rm div}u\|_{L^2}^2+\|Q\|_{L^6}^6dt\right)^p\nonumber\\
&&+\mathbb{E}\left(\int_0^{T}\int_{\mathcal{D}}\epsilon(\gamma \rho^{\gamma-2}+\delta \beta\rho^{\beta-2})|\nabla \rho|^2dxdt\right)^p\leq C.
\end{eqnarray}
Moreover, we have
\begin{eqnarray}
&&\sqrt{\epsilon}\rho \in L^{p}(\Omega; L^{2}(0,T;H^{1}(\mathcal{D}))),\label{3.8*}\\
&&\rho u\in L^{p}(\Omega; L^{\infty}(0,T;L^\frac{2\beta}{\beta+1}(\mathcal{D}))).\label{3.9*}
\end{eqnarray}
\end{lemma}
\begin{proof}
Denote $\Phi(\rho,m)=(m, \mathcal{M}^{-1}(\rho)m)$, we obtain
$$\nabla_m\Phi(\rho,m)=2\mathcal{M}^{-1}(\rho)m,\nabla_m^2\Phi(\rho,m)=2\mathcal{M}^{-1}(\rho),
\nabla_\rho\Phi(\rho,m)=-(m,\mathcal{M}^{-1}(\rho)\mathcal{M}^{-1}(\rho)m ).$$
Applying the It\^{o} formula to function $\Phi(\rho, \rho u)$, integrating with respect to time, taking the supremum on interval $[0, t\wedge\tau_K]$, then
\begin{align}\label{3.9}
   &\sup_{s\in [t\wedge\tau_K]}\int_{\mathcal{D}}\rho|u|^2dx\leq \!\int_{\mathcal{D}}\rho_0|u_0|^2dx-\int_{0}^{t\wedge\tau_K}\int_{\mathcal{D}}\nabla|u|^{2}\cdot \rho udx ds+\!\epsilon\int_{0}^{t\wedge\tau_K}\int_{\mathcal{D}}\nabla |u|^{2}\cdot \nabla \rho dxds\nonumber\\&-2\int_{0}^{t\wedge\tau_K}\int_{\mathcal{D}}\nabla(\rho^\gamma+\delta\rho^\gamma)\cdot udxds
    +2\int_{0}^{t\wedge\tau_K}\int_{\mathcal{D}}(\mu_1\triangle u+(\mu_1+\mu_2) \nabla{\rm div}u)\cdot udxds\nonumber\\&+2\int_{0}^{t\wedge\tau_K}\int_{\mathcal{D}}\rho u\otimes u:\nabla udxds-2\epsilon\int_{0}^{t\wedge\tau_K}\int_{\mathcal{D}}\nabla u\nabla\rho\cdot udxds\nonumber\\&-2\int_{0}^{t\wedge\tau_K}\int_{\mathcal{D}}(\nabla Q\odot \nabla Q-{\rm F}(Q){\rm I}_3):\nabla udxds\nonumber\\
&+2\int_{0}^{t\wedge\tau_K}\int_{\mathcal{D}}(Q\triangle Q-\triangle Q Q):\nabla udxds
+2\int_{0}^{t\wedge\tau_K}\int_{\mathcal{D}}\sigma^* (c^2 Q):\nabla udxds\nonumber\\
&+\int_{0}^{t\wedge\tau_K}\sum_{k\geq 1}(\mathcal{M}^{-1}(\rho)\mathcal{M}^{\frac{1}{2}}(\rho)P_n(\sqrt{\rho}f(\rho,\rho u, c, Q)e_k),\mathcal{M}^{\frac{1}{2}}(\rho)P_n(\sqrt{\rho}f(\rho,\rho u, c, Q)e_k))ds\nonumber\\
&+2\sup_{s\in [0, t\wedge\tau_K]}\left|\int_{0}^{s}\sum_{k\geq 1}\int_{\mathcal{D}}\mathcal{M}^{\frac{1}{2}}(\rho)u P_n(\sqrt{\rho}f(\rho,\rho u, c, Q)e_k)dx d\beta_k\right|.
\end{align}
By equation $(\ref{Equ3.1})_2$, we have
\begin{align*}
  -2\int_{\mathcal{D}}\nabla(\rho^\gamma+\delta \rho^\beta)\cdot udx
   &=-2\int_{\mathcal{D}}\left(\frac{\gamma}{\gamma-1}\nabla\rho^{\gamma-1}+\frac{\delta\beta}{\beta-1}\nabla\rho^{\beta-1}\right)\rho u dx\\
  &=2\int_{\mathcal{D}}\left(\frac{\gamma}{\gamma-1}\rho^{\gamma-1}+\frac{\delta\beta}{\beta-1}\rho^{\beta-1}\right){\rm div}(\rho u) dx\\
  &=2\int_{\mathcal{D}}\left(\frac{\gamma}{\gamma-1}\rho^{\gamma-1}+\frac{\delta\beta}{\beta-1}\rho^{\beta-1}\right)(-\partial_t\rho +\epsilon \triangle\rho) dx\\
  &=-2\int_{\mathcal{D}}\frac{1}{\gamma-1}d\rho^{\gamma}+\frac{\delta}{\beta-1}d\rho^{\beta}dx\\
  &\quad-2\epsilon\int_{\mathcal{D}}(\gamma \rho^{\gamma-2}+\delta \beta\rho^{\beta-2})|\nabla \rho|^2dxdt.
\end{align*}
Multiplying equation $(\ref{Equ3.1})_1$ with c, integrating over $\mathcal{D}$, then
\begin{eqnarray}\label{3.10}
    d\left(\frac{1}{2}\int_{\mathcal{D}}c^2dx\right)+\int_{\mathcal{D}}|\nabla c|^2dxdt
    =-\int_{\mathcal{D}}cu\cdot\nabla cdxdt.
\end{eqnarray}
Also, multiplying equation $\eqref{Equ3.1}_4$ with $-(\triangle Q-Q-c_*Q{\rm tr}(Q^2))$, taking the trace and integrating over $\mathcal{D}$, adding (\ref{3.9}) and (\ref{3.10}), then we get
\begin{align}\label{3.11}
&\sup_{s\in [t\wedge\tau_K]}\int_{\mathcal{D}}\left(\rho|u|^2+c^2+\frac{2\delta}{\beta-1}\rho^{\beta}+\frac{2}{\gamma-1}\rho^\gamma
+|Q, \nabla Q|^2+\frac{c_*}{2}|Q|^4\right)dx\nonumber\\
&\quad+2\int_{0}^{t\wedge\tau_K}\int_{\mathcal{D}}\mu_1|\nabla u|^2+(\mu_1+\mu_2)|{\rm div} u|^2
+|\nabla c|^2\nonumber\\&\quad\qquad\qquad+\Gamma(|\nabla Q|^2+|\triangle Q|^2)
+\Gamma(c_*|Q|^4+c_*^2|Q|^6)dxds\nonumber\\
&\quad+2\epsilon\int_{0}^{t\wedge\tau_K}\int_{\mathcal{D}}(\gamma \rho^{\gamma-2}+\delta \beta\rho^{\beta-2})|\nabla \rho|^2dxds\nonumber\\
&=\int_{\mathcal{D}}\left(\rho_0|u_0|^2+c_0^2+\frac{2\delta}{\beta-1}\rho_0^{\beta}+\frac{2}{\gamma-1}\rho_0^\gamma
+|Q_0, \nabla Q_0|^2+\frac{c_*}{2}|Q_0|^4\right)dx\nonumber\\
&\quad-\int_{0}^{t\wedge\tau_K}\int_{\mathcal{D}}\nabla|u|^{2}\cdot \rho udxds+\epsilon\int_{0}^{t\wedge\tau_K}\int_{\mathcal{D}}\nabla |u|^{2}\cdot \nabla \rho dxds+\!2\int_{0}^{t\wedge\tau_K}\int_{\mathcal{D}}\rho u\otimes u:\nabla udxds\nonumber\\
&\quad-2\epsilon\int_{0}^{t\wedge\tau_K}\int_{\mathcal{D}}\nabla u\nabla\rho\cdot udxds+2\int_{0}^{t\wedge\tau_K}\int_{\mathcal{D}}(\nabla Q\odot \nabla Q-{\rm F}(Q){\rm I}_3):\nabla udxds\nonumber\\
&\quad-2\int_{0}^{t\wedge\tau_K}\int_{\mathcal{D}}(Q\triangle Q-\triangle Q Q):\nabla udxds-2\int_{0}^{t\wedge\tau_K}\int_{\mathcal{D}}\sigma^* (c^2 Q):\nabla udxds\nonumber\\& \quad-2\int_{0}^{t\wedge\tau_K}\int_{\mathcal{D}}cu\cdot\nabla cdxds
+2\int_{0}^{t\wedge\tau_K}\int_{\mathcal{D}}u\cdot\nabla Q:(\triangle Q-Q-c_*Q{\rm tr}(Q^2))dxds\nonumber\\&\quad-2\int_{0}^{t\wedge\tau_K}\int_{\mathcal{D}}(\Psi Q-Q\Psi):(\triangle Q-Q-c_*Q{\rm tr}(Q^2))dxds\nonumber\\
&\quad+\int_{0}^{t\wedge\tau_K}\int_{\mathcal{D}}\Gamma(c-c_*) Q:(\triangle Q-Q-c_*Q{\rm tr}(Q^2))dxds\nonumber\\
&\quad-2\int_{0}^{t\wedge\tau_K}\int_{\mathcal{D}}b\Gamma Q^2:(\triangle Q-Q-c_*Q{\rm tr}(Q^2))dxds+\int_{0}^{t\wedge\tau_K}\int_{\mathcal{D}}2c_*\Gamma Q|Q|^2:\triangle Qdxds\nonumber\\
&\quad+\int_{0}^{t\wedge\tau_K}\sum_{k\geq 1}(\mathcal{M}^{-1}(\rho)\mathcal{M}^{\frac{1}{2}}(\rho)P_n(\sqrt{\rho}f(\rho,\rho u, c, Q)e_k),\mathcal{M}^{\frac{1}{2}}(\rho)P_n(\sqrt{\rho}f(\rho,\rho u, c, Q)e_k))ds\nonumber\\
&\quad+2\sup_{s\in [0, t\wedge\tau_K]}\left|\int_{0}^{s}\sum_{k\geq 1}\int_{\mathcal{D}}\mathcal{M}^{\frac{1}{2}}(\rho)u P_n(\sqrt{\rho}f(\rho,\rho u, c, Q)e_k)dx d\beta_k\right|\nonumber\\
&=:\int_{0}^{t\wedge\tau_K}\sum_{i=1}^{14}J_ids+2\sup_{s\in [0, t\wedge\tau_K]}\left|\int_{0}^{s}\sum_{k\geq 1}J_{15}d\beta_k\right|.
\end{align}

Next, we control all the right hand side terms of (\ref{3.11}). Note that $J_1+J_3=0, J_2+J_4=0$ and
\begin{eqnarray*}
&&J_6+J_{10}=-2\int_{\mathcal{D}}(Q\triangle Q-\triangle Q Q):\nabla udx\nonumber\\
 &&\qquad\qquad\quad+2\int_{\mathcal{D}}(\Psi Q-Q\Psi):(\triangle Q-Q-c_*Q{\rm tr}(Q^2))dx\\
&&\qquad\qquad=-2\int_{\mathcal{D}}{\rm tr}(Q\triangle Q\nabla u-\triangle Q Q\nabla u)dx
 +2\int_{\mathcal{D}}{\rm tr}(\Psi Q\triangle Q-Q\Psi\triangle Q)dx\\
&&\qquad\qquad\quad+2\int_{\mathcal{D}}{\rm tr}((\Psi Q-Q\Psi)(-Q-c_*Q{\rm tr}(Q^2)))dx=\mathcal{J}_1+\mathcal{J}_2+\mathcal{J}_3.
\end{eqnarray*}
Due to the fact that $Q$ is symmetric and traceless and $\Psi$ is skew-symmetric we have $\mathcal{J}_3=0, \mathcal{J}_1+\mathcal{J}_2=0$, also $J_5+J_9=0$, see \cite{14}.

Applying Young's inequality and the boundedness of $c$,  we have
\begin{eqnarray*}
  &&|J_7|=\left|2\int_{\mathcal{D}}\sigma^* c^2 Q:\nabla udx\right|\leq C\|c\|^2_{L^\infty([0,T]\times \mathcal{D})}\|\nabla u\|_{L^2}\|Q\|_{L^2}\leq \|\nabla u\|_{L^2}^2+C\|Q\|_{L^2}^2, \\
 && |J_8|=\left|2\int_{\mathcal{D}}cu\cdot\nabla cdx\right|\leq C\|c\|_{L^\infty([0,T]\times \mathcal{D})}\|\nabla c\|_{L^2}\|u\|_{L^2}\leq \|\nabla c\|_{L^2}^2+C\|u\|_{L^2}^2\\
  &&\qquad\quad\leq \|\nabla c\|_{L^2}^2+\mu_1\|\nabla u\|_{L^2}^2+C.
\end{eqnarray*}
In addition
 \begin{align*}
  |J_{11}|&=\left|\int_{\mathcal{D}}\Gamma(c-c_*) Q:(\triangle Q-Q-c_*Q{\rm tr}(Q^2))dx\right|\\ &\quad\leq \frac{\Gamma}{2}\|\triangle Q\|_{L^2}^2
  +C\|Q\|_{L^2}^2+C\|Q\|_{L^4}^4,\\
  |J_{12}|&=\left|\int_{\mathcal{D}}2b\Gamma Q^2:(\triangle Q-Q-c_*Q{\rm tr}(Q^2))dx\right|\\
  &\quad\leq b\Gamma\|\triangle Q\|_{L^2}\|Q^2\|_{L^2}+\frac{c_*\Gamma}{2}\|Q\|_{L^6}^6+C\|Q\|_{L^4}^4\\
  &\quad\leq \frac{\Gamma}{2}\|\triangle Q\|^2+C\|Q\|_{L^4}^4+\frac{c_*\Gamma}{2}\|Q\|_{L^6}^6
  +C\|Q\|^2,\\
  J_{13}&=\int_{\mathcal{D}}2c_*\Gamma Q|Q|^2:\triangle Qdx=-2c_*\Gamma\int_{\mathcal{D}}|\nabla Q|^2|Q|^2dx-c_*\Gamma\int_{\mathcal{D}}|\nabla{\rm tr}(Q)|^2dx\leq 0.
\end{align*}
Using the condition (\ref{2.1}), $J_{14}$ can be treated as
\begin{align*}
J_{14}&\leq\|\sqrt{\rho}f(\rho,\rho u, c, Q)\|_{L_{2}(\mathcal{H};L^2(\mathcal{D}))}^{2}\nonumber\\&\leq \int_{\mathcal{D}}\sum_{k\geq 1}|\sqrt{\rho} f(\rho,\rho u, c, Q)e_{k}|^2dx\nonumber\\
&\leq C\int_{\mathcal{D}}\rho^\gamma+|\sqrt{\rho}u|^2+\rho |c|^\frac{2(\gamma-1)}{\gamma}+\rho|\nabla Q|^\frac{2(\gamma-1)}{\gamma}dx\nonumber\\
&\leq C\int_{\mathcal{D}}\rho^\gamma+|\sqrt{\rho}u|^2+ |c|^2+|\nabla Q|^2dx.
\end{align*}
Define the stopping time $\tau_R$
\begin{eqnarray}
\tau_R=\inf\left\{t\geq 0; \sup_{s\in [0,t]}\|\sqrt{\rho }u\|_{L^2}^2\geq R\right\}\wedge \tau_K,
\end{eqnarray}
if the set is empty, taking $\tau_{R}=T$. Note that, $\tau_{R}$ is an increasing sequence with $\lim_{R\wedge K\rightarrow \infty}\tau_{R}= T$. Regarding the stochastic term, by the Burkholder-Davis-Gundy inequality and condition (\ref{2.1}) for all $1\leq p<\infty$
\begin{eqnarray*}
  &&\mathbb{E}\left[\sup_{s\in[0,t\wedge \tau_R]}\left|\int_{0}^{s}\sum_{k\geq 1}J_{15}d\beta_k\right|\right]^p \\
  &&\leq  C\mathbb{E}\left(\int_{0}^{t\wedge \tau_R}\sum_{k\geq 1}\left(\int_{\mathcal{D}}\mathcal{M}^{\frac{1}{2}}(\rho)uP_n(\sqrt{\rho} f(\rho,\rho u, c, Q)e_k) dx\right)^2ds\right)^\frac{p}{2}\\
  &&\leq C\mathbb{E}\left(\int_{0}^{t\wedge \tau_R}\|\mathcal{M}^{\frac{1}{2}}(\rho)u\|_{L^2}^2\int_{\mathcal{D}}\sum_{k\geq 1}|\rho f^2(\rho,\rho u, c, Q)e_k^2|dxds\right)^\frac{p}{2}\\
  &&\leq C\mathbb{E}\left[\sup_{s\in [0,t\wedge \tau_R ]}\|\sqrt{\rho}u\|_{L^2}^p\right]\left(\int_{0}^{t\wedge \tau_R}\int_{\mathcal{D}}\rho^\gamma+|\sqrt{\rho}u|^2+\rho |c|^\frac{2(\gamma-1)}{\gamma}+\rho|\nabla Q|^\frac{2(\gamma-1)}{\gamma}dxds\right)^\frac{p}{2}\\
  &&\leq \mathbb{E}\left[\sup_{s\in [0,t\wedge \tau_R ]}\|\sqrt{\rho}u\|_{L^2}^{2p}\right]+C\mathbb{E}\left(\int_{0}^{t\wedge \tau_R}\int_{\mathcal{D}}\rho^\gamma+|\sqrt{\rho}u|^2+|c|^2+|\nabla Q|^2dxds\right)^p.
\end{eqnarray*}
Considering all these estimates, taking the integral with respect to time, taking the supremum on interval $[0, t\wedge \tau_R]$, then power $p$ and taking expectation on both sides, the Gronwall lemma yields
\begin{eqnarray}
&&\mathbb{E}\left[\sup_{s\in [0,t\wedge \tau_R]}\left(\|c, Q, \nabla Q, \sqrt{\rho}u\|_{L^2}^2+\|Q\|_{L^4}^4+\frac{1}{\gamma-1}\|\rho\|_{L^\gamma}^\gamma+\frac{\delta}{\beta-1}\|\rho\|_{L^\beta}^\beta\right)\right]^p\nonumber\\
&&+\mathbb{E}\left(\int_0^{t\wedge \tau_R}\mu_1\|\nabla u\|_{L^2}^2+(\mu_1+\mu_2)\|{\rm div}u\|_{L^2}^2+\Gamma\|\triangle Q\|_{L^2}^2+\|\nabla c\|_{L^2}^2+c_*^2\Gamma\|Q\|_{L^6}^6ds\right)^p\nonumber\\
&&+\mathbb{E}\left(\int_0^{t\wedge \tau_R}\int_{\mathcal{D}}\epsilon(\gamma \rho^{\gamma-2}+\delta \beta\rho^{\beta-2})|\nabla \rho|^2dxds\right)^p\leq C,
\end{eqnarray}
where $C$ is constant independent of $n$. Finally, we get the bound (\ref{3.8}) by the monotone convergence theorem.

Using the bound (\ref{3.8}), we have
\begin{eqnarray*}
&&\mathbb{E}\int_{\mathcal{D}}|\rho|^{2}dx+2\epsilon\mathbb{E}\int_{0}^{t}\int_{\mathcal{D}}|\nabla \rho|^{2}dxds=\mathbb{E}\int_{\mathcal{D}}|\rho_{0}|^{2}dx-\mathbb{E}\int_{0}^{t}\int_{\mathcal{D}}{\rm div} u|\rho|^{2}dxds\\ &&\leq
\mathbb{E}\int_{\mathcal{D}}|\rho_{0}|^{2}dx+\mathbb{E}\left(\int_{0}^{t}\int_{\mathcal{D}}|{\rm div} u|^{2}+|\rho|^{4}dxds\right)^{4}\leq C,
\end{eqnarray*}
consequently, (\ref{3.8*}) holds. We can obtain the bound (\ref{3.9*}) using the fact that $\sqrt{\rho}u\in L^p(\Omega; L^\infty(0,T; L^2(\mathcal{D})))$ and $\rho\in L^p(\Omega; L^\infty(0,T; L^\beta(\mathcal{D})))$. This completes the proof.
\end{proof}
{\bf Remark}. By taking inner product with $-(\triangle Q-Q-c_*Q{\rm tr}(Q^2))$ in equation $(\ref{Equ3.1})_4$ in place of ${\rm H}(Q,c)$, we can prevent the interaction term of $c$ and $Q$-tensor arising, making the estimates concise.

\subsection{The compactness of approximate solution}

Unlike the deterministic case, it may not be the case that the embedding $L^{2}(\Omega; X)$ into $L^{2}(\Omega; Y)$ is compact, even if $X\hookrightarrow Y$ is compact. Therefore, in order to obtain the compactness of approximate solution, the key point is to obtain the compactness of the set of probability measures generated by the approximate solution sequences. Define the path space
\begin{eqnarray*}
\mathcal{X}=\mathcal{X}_{u}\times \mathcal{X}_{\rho}\times \mathcal{X}_{\rho u}\times   \mathcal{X}_{c}\times\mathcal{X}_{Q}\times\mathcal{X}_{\mathcal{W}},
\end{eqnarray*}
where
\begin{eqnarray*}
&&\mathcal{X}_{u}:=L_{w}^{2}(0,T;H^{1}(\mathcal{D})),\mathcal{X}_{\rho u}:=C([0,T];H^{-1}(\mathcal{D})),\\
&&\mathcal{X}_{\rho}:=L^\infty(0,T; H^{-\frac{1}{2}}(\mathcal{D}))\cap L^{2}(0,T; L^{2}(\mathcal{D}))\cap L_{w}^{2}(0,T;H^{1}(\mathcal{D})),\\
&& \mathcal{X}_{c}:= L_{w}^{2}(0,T;H^{1}(\mathcal{D}))\cap L^2(0,T; L^2(\mathcal{D})),\\
&& \mathcal{X}_{Q}:=L_{w}^{2}(0,T;H^2(\mathcal{D}))\cap L^2(0,T; H^1(\mathcal{D})), ~\mathcal{X}_{\mathcal{W}}:=C([0,T];\mathcal{H}_{0}).
\end{eqnarray*}
Define the probability measures
\begin{eqnarray}\label{3.17}
\nu^{n}=\nu^{n}_{u}\otimes \nu^{n}_{\rho}\otimes \nu^{n}_{\rho u}\otimes \nu^{n}_{c}\otimes \nu^{n}_{Q}\otimes\nu_{\mathcal{W}},
\end{eqnarray}
where $\nu^n_{(\cdot)}(B)=\mathbb{P}\{\cdot\in B\}$ for any $B\in \mathcal{B}(\mathcal{X}_{(\cdot)})$, $\mathcal{X}_{(\cdot)}$ is the path space defined above, respectively.

Next, we establish the following compactness result.
\begin{proposition}\label{pro3.1} There exists a subsequence of probability measures $\{\nu^{n}\}_{n\geq 1}$ still denoted by $\{\nu^{n}\}_{n\geq 1}$, a probability space $(\widetilde{\Omega},\widetilde{\mathcal{F}},\widetilde{\mathbb{P}})$ with $\mathcal{X}$-valued measurable random variables
\begin{eqnarray*}
(\tilde{u}_{n},\tilde{\rho}_{n},P_n(\tilde{\rho}_{n}\tilde{u}_{n}), \tilde{c}_n, \widetilde{Q}_n, \widetilde{\mathcal{W}}_{n}) ~{\rm and} ~(\tilde{u},\tilde{\rho}, \tilde{\rho}\tilde{u}, \tilde{c}, \widetilde{Q}, \widetilde{\mathcal{W}}),
\end{eqnarray*}
such that
\begin{eqnarray}&&(\tilde{u}_{n},\tilde{\rho}_{n}, P_n(\tilde{\rho}_{n}\tilde{u}_{n}), \tilde{c}_n, \widetilde{Q}_n, \widetilde{\mathcal{W}}_{n})
\rightarrow(\tilde{u},\tilde{\rho},\tilde{\rho}\tilde{u}, \tilde{c}, \widetilde{Q}, \widetilde{\mathcal{W}}), ~\widetilde{\mathbb{P}}~\mbox{a.s.} \label{3.18}
\end{eqnarray}
in the topology of $\mathcal{X}$ and
\begin{eqnarray}
&&\widetilde{\mathbb{P}}\{(\tilde{u}_{n},\tilde{\rho}_{n},P_n(\tilde{\rho}_{n}\tilde{u}_{n}), \tilde{c}_n, \widetilde{Q}_n, \widetilde{\mathcal{W}}_{n})\in \cdot\}=\nu^{n}(\cdot),\label{3.19}\\
&&\widetilde{\mathbb{P}}\{(\tilde{u},\tilde{\rho}, \tilde{\rho}\tilde{u}, \tilde{c}, \widetilde{Q}, \widetilde{\mathcal{W}})\in \cdot\}=\nu(\cdot),\label{3.20}
\end{eqnarray}
where $\nu$ is a Radon measure and $\widetilde{\mathcal{\mathcal{W}}}_{n}$ is cylindrical Wiener process, relative to the filtration $\widetilde{\mathcal{F}}_{t}^{n}$ generated by the completion of $\sigma(\tilde{u}_{n}(s),\tilde{\rho}_{n}(s), \tilde{c}_n(s), \widetilde{Q}_n(s), \widetilde{\mathcal{W}}_{n}(s);s\leq t)$. Moreover, the process $(\tilde{u}_n,\tilde{\rho}_n, \tilde{\rho}_n\tilde{u}_n, \tilde{c}_n, \widetilde{Q}_n, \widetilde{\mathcal{W}}_n)$ also satisfies the system (\ref{Equ3.1}) and shares the following uniform a priori estimates
\begin{eqnarray}
&&\tilde{\rho}_n\in L^{p}(\widetilde{\Omega}; L^{\infty}(0,T;L^{\beta}(\mathcal{D}))\cap L^{2}(0,T;H^1(\mathcal{D}))),\label{3.21}\\
&&\tilde{u}_n\in  L^{p}(\widetilde{\Omega}; L^{2}(0,T;H^{1}(\mathcal{D}))), \label{3.22}\\
&& \sqrt{\tilde{\rho}_n} \tilde{u}_n\in L^{p}(\widetilde{\Omega}; L^{\infty}(0,T;L^{2}(\mathcal{D}))),\label{3.23}\\
&&\tilde{\rho}_n \tilde{u}_n\in L^{p}(\widetilde{\Omega}; L^{\infty}(0,T;L^{\frac{2\beta}{\beta+1}}(\mathcal{D}))),\label{3.24}\\
&& \tilde{c}_n \in L^{p}(\widetilde{\Omega}; L^{\infty}(0,T;L^2(\mathcal{D}))\cap L^2(0,T;H^1(\mathcal{D}))),\label{3.25}\\
&& \widetilde{Q}_n \in L^{p}(\widetilde{\Omega}; L^{\infty}(0,T;H^1(\mathcal{D}))\cap L^2(0,T;H^2(\mathcal{D}))).\label{3.26}
\end{eqnarray}
\end{proposition}
Combining the bound (\ref{3.21}) and strong convergence $\tilde{\rho}_n$ in $L^2(0,T; L^2(\mathcal{D}))$, the Vitali convergence theorem \ref{thm6.1} implies that $\widetilde{\mathbb{P}}$ a.s.
\begin{eqnarray*}
\tilde{\rho}_n\rightarrow\tilde{\rho}~{\rm in} ~L^4(0,T; L^4(\mathcal{D})).
\end{eqnarray*}

In order to employ the Skorokhod-Jakubowski theorem, we next show the tightness of set $\{\nu^n\}_{n\geq 1}$.
\begin{lemma} \label{lem3.4} The set of probability measures $\{\nu^{n}\}_{n\geq 1}$ is tight on path space $\mathcal{X}$.
\end{lemma}
\begin{proof} It is enough to show that each set of probability measures $\{\nu^n_{(\cdot)}\}_{n\geq 1}$ is tight on the corresponding path space $\mathcal{X}_{(\cdot)}$.\\
\noindent{\bf Claim 1}. The set $ \{\nu_{P_n(\rho u)}^{n}\}_{n\geq 1}$ is tight on path spaces $\mathcal{X}_{\rho u}$.

Decompose $P_n(\rho_nu_n)= X_n+Y_n$, where
\begin{eqnarray*}
&&X_n=m_{0,n}+P_n\int_{0}^{t}-{\rm div}(\rho_n u_n\otimes u_n)-\nabla (\rho^\gamma_n+\delta\rho_n^\beta)
 +\mu_1\Delta u_n+(\mu_1+\mu_2)\nabla ({\rm div} u_n)\\&&\qquad+\nabla\cdot({\rm F}(Q_n){\rm I}_3-\nabla Q_n\odot \nabla Q_n) +\nabla\cdot(Q_n\triangle Q_n-\triangle Q_n Q_n)+\sigma^*\nabla\cdot(c_n^2Q_n)ds\\&&\qquad+\int_{0}^{t}\mathcal{M}^\frac{1}{2}(\rho_n)P_n\sqrt{\rho_n} f(\rho_n,\rho_n u_n, c_n, Q_n)d\mathcal{W},
\end{eqnarray*}
and
\begin{eqnarray*}
Y_n=\epsilon\int_{0}^{t}P_n(\nabla \rho_n\cdot\nabla u_n) ds.
\end{eqnarray*}
The main goal is to get
\begin{eqnarray}\label{3.27}
\mathbb{E}\|P_n(\rho_n u_n)\|_{C^\alpha([0,T];H^{-k}(\mathcal{D}))}\leq C,
\end{eqnarray}
where $C$ is independent of $n$ for $\alpha\in [0,\frac{1}{2})$ and $k\geq \frac{5}{2}$.

Regarding the stochastic term, similar to \eqref{3.19*}, using the Burkholder-Davis-Gundy inequality and condition (\ref{2.1}), we get for all $\alpha\in [0,\frac{1}{2})$
\begin{eqnarray*}
&&\mathbb{E}\left\|\int_{0}^{t}\mathcal{M}^\frac{1}{2}(\rho_n)P_n\sqrt{\rho_n}f(\rho_n,\rho_n u_n, c_n, Q_n)d\mathcal{W}\right\|_{C^\alpha([0,T];H^{-k}(\mathcal{D}))}\\ &&\leq
\mathbb{E}\left[\sup_{t,t'\in [0,T]}\frac{\left\|\int_{t'}^{t}\mathcal{M}^\frac{1}{2}(\rho_n)P_n\sqrt{\rho_n}f(\rho_n,\rho_n u_n, c_n, Q_n)d\mathcal{W}\right\|_{H^{-k}}}{|t-t'|^\alpha}\right]\\
&&\leq \frac{\mathbb{E}\left\|\int_{t'}^{t}\mathcal{M}^\frac{1}{2}(\rho_n)P_n\sqrt{\rho_n}f(\rho_n,\rho_n u_n, c_n, Q_n)d\mathcal{W}\right\|_{H^{-k}}}{|t-t'|^\alpha}+\delta'\\
&&\leq \frac{\mathbb{E}\left(\int_{t'}^{t}\|\mathcal{M}^\frac{1}{2}(\rho_n)\sum_{k\geq 1}P_n\sqrt{\rho_n} f(\rho_n,\rho_n u_n, c_n, Q_n)e_k\|_{L^1}^2dr\right)^\frac{1}{2}}{|t-t'|^\alpha}+\delta'\\
&&\leq C|t-t'|^{\frac{1}{2}-\alpha}\mathbb{E}\left[\sup_{t\in[0,T]}(\|\rho_n\|_{L^\gamma}+\|\sqrt{\rho_n}u_n\|^2_{L^2}+\|c_n\|^2_{L^2}+\|\nabla Q_n\|^2_{L^2})\right]
+\delta'\leq C.
\end{eqnarray*}
Using (\ref{3.8}) and the H\"{o}lder inequality, we have
\begin{align*}
\mathbb{E}\|\nabla\cdot (Q_n\triangle Q_n-\triangle Q_n Q_n)\|_{L^2(0,T;H^{-k}(\mathcal{D}))}^p&\leq \mathbb{E}\left(\int_0^T\|Q_n\triangle Q_n-\triangle Q_n Q_n\|_{L^1}^2dt\right)^\frac{p}{2}\\
&\leq \mathbb{E}\left(\int_0^T\|Q_n\|_{L^2}^2\|\triangle Q_n\|_{L^2}^2dt\right)^\frac{p}{2}\\&\leq \mathbb{E}\left[\sup_{t\in[0,T]}\|Q_n\|_{L^2}^{2p}\right]\mathbb{E}\left(\int_0^T\|\triangle Q_n\|_{L^2}^2dt\right)^p\\ &\leq C,
\end{align*}
and
\begin{eqnarray*}
\mathbb{E}\|\sigma^*\nabla\cdot(c_n^2Q_n)\|_{L^\infty(0,T;H^{-1}(\mathcal{D}))}^p\leq C\|c_n\|_{L^\infty([0,T]\times \mathcal{D})}^{2p}\mathbb{E}\|Q_n\|_{L^\infty(0,T;L^2(\mathcal{D}))}^p\leq C,
\end{eqnarray*}
moreover
\begin{eqnarray*}
\mathbb{E}\|\nabla\cdot({\rm F}(Q_n){\rm I}_3-\nabla Q_n\odot \nabla Q_n)\|_{L^\infty(0,T;H^{-k}(\mathcal{D}))}^p\leq \mathbb{E}\left[\sup_{t\in [0,T]}\|\nabla Q_n\|_{L^2}^p\right]\leq C,
\end{eqnarray*}
where the constant $C$ is independent of $n$.

Furthermore, the bound (\ref{3.8}) together with the H\"{o}lder inequality yields, ${\rm div} (\rho_{n} u_{n}\otimes u_{n}) \in L^{p}(\Omega;L^{2}(0,T; H^{-1, \frac{6\beta}{4\beta+3}}(\mathcal{D})))$, then the Sobolev embedding $H^{-1, \frac{6\beta}{4\beta+3}}(\mathcal{D})\hookrightarrow H^{-k}(\mathcal{D})$ for $k\geq \frac{5}{2}$ implies that ${\rm div} (\rho_{n} u_{n}\otimes u_{n}) \in L^{p}(\Omega;L^{2}(0,T; H^{-k}(\mathcal{D})))$. Also, the bound (\ref{3.8}) implies $\nabla (\rho^\gamma_n+\delta\rho_n^\beta)\in L^{p}(\Omega;L^{\frac{\beta+1}{\beta}}(0,T; H^{-k}(\mathcal{D})))$ using the Sobolev embedding $H^{-1, \frac{\beta+1}{\beta}}(\mathcal{D})\hookrightarrow H^{-k}(\mathcal{D})$.

To find the boundedness of $Y_n$, we need to improve the time integrability of $\rho_n$ following Lemma 2.4 in \cite{Maslow}. By (\ref{3.8}) and the H\"{o}lder inequality, we have
\begin{eqnarray*}
\rho_n u_n\in L^p(\Omega; L^2(0,T; L^\frac{6\beta}{\beta+6}(\mathcal{D})))\cap L^p(\Omega; L^\infty(0,T; L^\frac{2\beta}{\beta+1}(\mathcal{D}))).
\end{eqnarray*}
The interpolation lemma \ref{lem6.2} implies that there exists $q>2$ such that
\begin{eqnarray*}
\rho_n u_n\in L^p(\Omega; L^q(0,T; L^2(\mathcal{D}))),
\end{eqnarray*}
 this estimate together with the bound $\rho_n \in L^p(\Omega; L^\infty(0,T; L^\beta(\mathcal{D})))$ and equation $(\ref{Equ3.1})_2$ yields
\begin{eqnarray}\label{3.31*}
\rho_n\in L^p(\Omega; L^q(0,T; H^1(\mathcal{D}))),~{\rm for} ~q>2.
\end{eqnarray}
Using (\ref{3.31*}) and (\ref{3.8}), we have
\begin{eqnarray}\label{3.29}
&&\mathbb{E}\left\|\epsilon\int_{0}^{t}P_n(\nabla\rho_{n}\cdot\nabla u_{n}) ds\right\|_{C^{\alpha}([0,T];H^{-k}(\mathcal{D}))}\\ &&\leq \mathbb{E}\frac{\left\|\epsilon\int_{t'}^{t} P_n(\nabla\rho_{n}\cdot\nabla u_{n}) ds\right\|_{H^{-k}}}{|t-t'|^\alpha}+\delta'
\leq\mathbb{E}\frac{\epsilon\int_{t'}^{t}\|\nabla\rho_{n}\cdot\nabla u_{n}\|_{L^1} ds}{|t-t'|^\alpha}+\delta'\nonumber\\
&&\leq \frac{1}{|t-t'|^\alpha}\mathbb{E}\int_{0}^t\|\nabla u\|_{L^2}^2ds \mathbb{E}\int_{0}^t\epsilon\|\nabla \rho\|_{L^2}^2ds+\delta' \nonumber\\
&&\leq C|t-t'|^{\frac{q-2}{q}-\alpha} \epsilon\mathbb{E}\left(\int_{0}^t\|\nabla \rho\|_{L^2}^qds\right)^\frac{2}{q}+\delta'\leq C,
\end{eqnarray}
for any $\alpha\in [0,\frac{q-2}{q}]$ and $k\geq\frac{5}{2}$.

Combining all estimates, we get the desired bound (\ref{3.27}). For any $R>0$, define the set
\begin{eqnarray*}
&&B_{1,R}=\bigg\{P_n(\rho_{n}u_n)\in L^{\infty}(0,T;L^{\frac{2\beta}{\beta+1}}(\mathcal{D}))\cap C^{\alpha}([0,T];H^{-k}(\mathcal{D})) :\\&&\qquad\qquad\qquad\qquad\qquad \|\rho_{n}u_n\|_{ L^{\infty}(0,T;L^{\frac{2\beta}{\beta+1}}(\mathcal{D}))}+\|\rho_n u_n\|_{C^{\alpha}([0,T];H^{-k}(\mathcal{D}))}\leq R\bigg\},
\end{eqnarray*}
By the Aubin-Lions lemma \ref{lem6.1}, we know
\begin{eqnarray}
 L^{\infty}(0,T;L^{\frac{2\beta}{\beta+1}}(\mathcal{D}))\cap C^{\alpha}([0,T];H^{-k}(\mathcal{D}))\hookrightarrow L^\infty(0,T; H^{-1}(\mathcal{D})),
\end{eqnarray}
is compact, therefore, the set $B_{1,R}$ is relatively compact in $L^\infty(0,T; H^{-1}(\mathcal{D}))$.  Considering (\ref{3.27}), (\ref{3.9*}) and the Chebyshev inequality, to conclude
\begin{eqnarray*}
&&\nu_{\rho u}^{n}(B_{1,R}^\mathfrak{c})\leq \mathbb{P}\left(\|\rho_n u_n\|_{ L^{\infty}(0,T;L^{\frac{2\beta}{\beta+1}}(\mathcal{D}))}>\frac{R}{2}\right)\\
&&~\qquad\qquad\quad+\mathbb{P}\left(\|\rho_n u_n\|_{C^{\alpha}([0,T];H^{-k}(\mathcal{D}))}>\frac{R}{2}\right)\\
&&~\qquad\qquad\leq \frac{2}{R}\mathbb{E}\left(\|\rho_n u_n\|_{ L^{\infty}(0,T;L^{\frac{2\beta}{\beta+1}}(\mathcal{D}))}+\|\rho_n u_n\|_{C^{\alpha}([0,T];H^{-k}(\mathcal{D}))}\right)\leq\frac{C}{R},
\end{eqnarray*}
leading to the tightness of set $\{\nu_{\rho u}^{n}\}_{n\geq 1}$.

\noindent{\bf Claim 2}. The set $ \{\nu_{c}^{n}\}_{n\geq 1}$ is tight on path space $\mathcal{X}_{c}$.

Note that, for any $R>0$, by the Banach-Alaoglu theorem, the set
\begin{eqnarray*}
B_{2,R}:=\left\{c_{n}\in L^{2}(0,T; H^1(\mathcal{D})): \|c_{n}\|_{L^{2}(0,T;H^1(\mathcal{D}))}\leq R\right\},
\end{eqnarray*}
is relatively compact on path space $L_{w}^{2}(0,T;H^1(\mathcal{D}))$. On the other hand, we have
\begin{eqnarray}\label{3.30}
\partial_t c_n\in L^p(\Omega; L^2(0,T;H^{-1}(\mathcal{D}))).
\end{eqnarray}
Define the set
\begin{eqnarray*}
&&B_{3,R}=\big\{c_{n}\in L^{2}(0,T; H^1(\mathcal{D}))\cap W^{1,2}(0,T;H^{-1}(\mathcal{D})):\\ &&\qquad\qquad\qquad\qquad\|c_{n}\|_{L^{2}(0,T;H^1(\mathcal{D}))}+\|c_n\|_{W^{1,2}(0,T;H^{-1}(\mathcal{D}))}\leq R\big\},
\end{eqnarray*}
which is compact on $L^{2}(0,T; L^2(\mathcal{D}))$. The bounds (\ref{3.8}), (\ref{3.30}) and the Chebyshev inequality imply
\begin{eqnarray*}
&&\nu_{c}^n\left((B_{2,R}\cap B_{3,R})^\mathfrak{c}\right)\leq \nu_c^{n}(B_{2,R}^\mathfrak{c})+ \nu_c^{n}(B_{3,R}^\mathfrak{c})\leq \frac{C}{R}.
\end{eqnarray*}

\noindent{\bf Claim 3}. The set $ \{\nu_{Q}^{n}\}_{n\geq 1}$ is tight on path space $\mathcal{X}_{Q}$.

The proof follows the same line as above, here we only give the necessary estimates. Using (\ref{3.8}) and the H\"{o}lder inequality again, we have
\begin{eqnarray*}
&&\mathbb{E}\|-(u_n\cdot \nabla)Q_n-(Q_n\Psi_n-\Psi_n Q_n)+\Gamma {\rm H}(Q_n,c_n)\|_{L^2(0,T;L^\frac{3}{2}(\mathcal{D}))}\\
&&\leq C \mathbb{E}\left[\int_0^t\|u_n\|_{L^6}^2\|\nabla Q_n\|_{L^2}^2+\|Q_n\|_{L^6}^2\|\nabla u_n\|_{L^2}^2+\|\triangle Q_n\|_{L^2}^2+\|c_n\|_{L^2}^2\|Q_n\|_{L^6}^2ds\right]^\frac{1}{2}\\
&&\leq C\mathbb{E}\left[\sup_{t\in[0,T]}\|Q_n\|_{H^1}^2\right]\mathbb{E}\int_0^t\|\nabla u_n\|_{L^2}^2ds+ \mathbb{E}\int_0^t\|\triangle Q_n\|_{L^2}^2ds\\ &&\quad+\mathbb{E}\left[\sup_{t\in [0,T]}(\|c_n\|^2_{L^2}+\|Q_n\|_{H^1}^2)\right]\leq C,
\end{eqnarray*}
leading to
\begin{eqnarray*}
&&\mathbb{E}\|Q_n\|_{C^\alpha([0,T];L^\frac{3}{2}(\mathcal{D}))}\\
&&\leq\mathbb{E}\frac{\int_{t'}^{t}\|-(u_n\cdot \nabla)Q_n-(Q_n\Psi_n-\Psi_n Q_n)+\Gamma {\rm H}(Q_n,c_n)\|_{L^\frac{3}{2}}ds}{|t-t'|^\alpha}+\delta'\\
&&\leq |t-t'|^{\frac{1}{2}-\alpha}\cdot\mathbb{E}\|-(u_n\cdot \nabla)Q_n-(Q_n\Psi_n-\Psi_n Q_n)+\Gamma {\rm H}(Q_n,c_n)\|_{L^2(0,T;L^\frac{3}{2}(\mathcal{D}))}+\delta'\\ &&\leq C,
\end{eqnarray*}
where $C$ is independent of $n$.

\noindent{\bf Claim 4}. The sets $ \{\nu_{u}^{n}\}_{n\geq 1}$ and $ \{\nu_{\rho}^{n}\}_{n\geq 1}$  are tight on path spaces $\mathcal{X}_{u}, \mathcal{X}_{\rho}$.

Here, we only focus on the tightness of $\{\nu_{\rho}^{n}\}_{n\geq 1}$ on space $L^2([0,T]\times \mathcal{D})$. Since $\rho_n\in L^2(0,T;H^1(\mathcal{D}))$ and $\partial_t\rho_n\in L^2(0,T;H^{-1}(\mathcal{D}))$, then we can show the tightness using the same argument as Claim 2.

Finally, Lemma \ref{lem3.4} follows the result of Claims 1-4.
\end{proof}

\begin{proof}[{\rm Proof of Proposition \ref{pro3.1}}] With the tightness established,  the Skorokhod-Jakubowski theorem is invoked to get that there exists a probability space $(\widetilde{\Omega},\widetilde{\mathcal{F}},\widetilde{\mathbb{P}})$ with $\mathcal{X}$-valued measurable random variables
\begin{eqnarray*}
(\tilde{u}_{n},\tilde{\rho}_{n}, \tilde{q}_{n}, \tilde{c}_n, \widetilde{Q}_n, \widetilde{\mathcal{W}}_{n})~{\rm and }~(\tilde{u},\tilde{\rho}, \tilde{q}, \tilde{c}, \widetilde{Q}, \widetilde{\mathcal{W}}),
\end{eqnarray*}
such that
\begin{eqnarray*}
(\tilde{u}_{n},\tilde{\rho}_{n}, \tilde{q}_{n}, \tilde{c}_n, \widetilde{Q}_n, \widetilde{\mathcal{W}}_{n})\rightarrow(\tilde{u},\tilde{\rho}, \tilde{q}, \tilde{c}, \widetilde{Q}, \widetilde{\mathcal{W}}),~\widetilde{\mathbb{P}}~~ \mbox{a.s.}
\end{eqnarray*}
in the topology of $\mathcal{X}$. Moreover, the joint distribution of $(\tilde{u}_{n},\tilde{\rho}_{n}, \tilde{q}_{n}, \tilde{c}_n, \widetilde{Q}_n, \widetilde{\mathcal{W}}_{n})$ is the same as the law of $(u_n, \rho_{n}, P_n(\rho_{n}u_n), c_n, Q_n, \mathcal{W}_{n})$, consequently, we have $\tilde{q}_{n}=P_n(\tilde{\rho}_n\tilde{u}_n)$, $\widetilde{Q}_n\in S_0^3$, a.s. and the energy estimates (\ref{3.21})-(\ref{3.26}) hold. Moreover, the process $(\tilde{\rho}_n,  P_n(\tilde{\rho}_n\tilde{u}_n), \tilde{c}_n, \widetilde{Q}_n, \widetilde{\mathcal{W}}_n)$ also satisfies the system (\ref{Equ3.1}) using the same argument as \cite{DWang}.

It remains to identify $\tilde{q}=\tilde{\rho}\tilde{u}$.  On the one hand, $ P_n(\tilde{\rho}_n \tilde{u}_n)\rightarrow \tilde{q}$ in $C(0,T;H^{-1}(\mathcal{D}))$, $\widetilde{\mathbb{P}}$ a.s. On the other hand, $\tilde{\rho}_n\rightarrow \tilde{\rho}$ in $L^\infty(0,T; H^{-\frac{1}{2}}(\mathcal{D}))$ and $\tilde{u}_n\rightharpoonup \tilde{u}$ in $L^2(0,T; H^{1}(\mathcal{D}))$ imply $ P_n(\tilde{\rho}_n \tilde{u}_n)\rightharpoonup\tilde{\rho}\tilde{u}$ in $L^2(0,T;H^{-1}(\mathcal{D}))$, $\widetilde{\mathbb{P}}$ a.s. Then, we infer $\tilde{q}=\tilde{\rho}\tilde{u}$.
\end{proof}

\subsection{Taking the limit for $n\rightarrow \infty$}
Based on the Proposition \ref{pro3.1}, we identify the limit of the nonlinear term.
\begin{lemma}\label{lem3.5} For any $\phi\in L^\infty(0,T;H^{1,6}(\mathcal{D}))$ and $t\in [0,T]$, the following convergence holds $\widetilde{\mathbb{P}}$ \mbox{a.s.}
\begin{eqnarray*}
&&\int_0^t\langle {\rm F}(\widetilde{Q}_n){\rm I}_3-\nabla \widetilde{Q}_n\odot\nabla \widetilde{Q}_n+\widetilde{Q}_n\triangle \widetilde{Q}_n-\triangle \widetilde{Q}_n\widetilde{Q}_n+\sigma^*\tilde{c}_n^2\widetilde{Q}_n, \nabla \phi\rangle ds\\ &&\rightarrow \int_0^t\langle {\rm F}(\widetilde{Q}){\rm I}_3-\nabla \widetilde{Q}\odot\nabla \widetilde{Q}+\widetilde{Q}\triangle \widetilde{Q}-\triangle\widetilde{ Q} \widetilde{Q}+\sigma^*\tilde{c}^2\widetilde{Q}, \nabla \phi\rangle ds,
\end{eqnarray*}
as $n\rightarrow \infty$.
\end{lemma}
\begin{proof} Decompose
\begin{eqnarray*}
&&\int_0^t\langle \widetilde{Q}_n\triangle \widetilde{Q}_n-\triangle \widetilde{Q}_n\widetilde{Q}_n-(\widetilde{Q}\triangle \widetilde{Q}-\triangle\widetilde{ Q} \widetilde{Q}), \nabla \phi \rangle ds\\
&&=\int_0^t\langle (\widetilde{Q}_n-\widetilde{Q})\triangle \widetilde{Q}_n, \nabla \phi\rangle ds+\int_0^t\langle  \widetilde{Q}(\triangle \widetilde{Q}_n-\triangle \widetilde{Q}), \nabla \phi\rangle ds\\
&&\quad+\int_0^t\langle (\triangle \widetilde{Q}-\triangle \widetilde{Q}_n) \widetilde{Q}, \nabla \phi\rangle ds+\int_0^t\langle \triangle \widetilde{Q}_n(\widetilde{Q}-\widetilde{Q}_n), \nabla \phi\rangle ds\\
&&=:J_1+J_2+J_3+J_4.
\end{eqnarray*}
For $J_1, J_4$, by Proposition \ref{pro3.1}(\ref{3.18}) and (\ref{3.26}), we have $\widetilde{\mathbb{P}}$ a.s.
\begin{eqnarray*}
&&|J_1+J_4|\leq \int_0^t \|\nabla \varphi\|_{L^3}\|\widetilde{Q}_n-\widetilde{Q}\|_{L^6}\|\triangle\widetilde{Q}_n\|_{L^2}ds\\
&&\qquad\qquad\leq \|\nabla \varphi\|_{L^\infty(0,T;L^3(\mathcal{D}))}\left(\int_0^t\|\widetilde{Q}_n-\widetilde{Q}\|_{H^1}^2ds\right)^\frac{1}{2}
\left(\int_0^t\|\triangle\widetilde{Q}_n\|_{L^2}^2ds\right)^\frac{1}{2}\rightarrow 0.
\end{eqnarray*}
Also, we have $J_2, J_3\rightarrow 0$ as $n\rightarrow\infty$ using the fact $\triangle \widetilde{Q}_n\rightharpoonup\triangle \widetilde{Q}$ in $L^2([0,T]\times\mathcal{D})$.

On the other hand,  by Proposition \ref{pro3.1}(\ref{3.18}), (\ref{3.25}) and (\ref{3.26}), the following convergences hold $\widetilde{\mathbb{P}}$ a.s.
\begin{eqnarray*}
&&\int_0^t\langle \nabla \widetilde{Q}_n\odot\nabla \widetilde{Q}_n-\nabla \widetilde{Q}\odot\nabla \widetilde{Q}, \nabla \phi\rangle ds\\
&&\leq \int_0^t \|\nabla\widetilde{Q}_n-\nabla\widetilde{Q}\|_{L^2}\|\nabla \widetilde{Q},\nabla \widetilde{Q}_n\|_{L^6}\|\nabla \phi\|_{L^3}ds\\
&&\leq \|\nabla \phi\|_{L^\infty(0,T;L^3(\mathcal{D}))}\left(\int_0^t\|\widetilde{Q}_n-\widetilde{Q}\|_{H^1}^2ds\right)^\frac{1}{2}
\left(\int_0^t\|\nabla\widetilde{Q}_n, \nabla\widetilde{Q}\|_{H^1}^2ds\right)^\frac{1}{2}\rightarrow 0,
\end{eqnarray*}
and
\begin{eqnarray*}
&&\int_0^t\langle \tilde{c}_n^2\widetilde{Q}_n-\tilde{c}^2\widetilde{Q}, \nabla \phi\rangle ds=\int_0^t\langle (\tilde{c}_n^2-\tilde{c}^2)\widetilde{Q}_n+\tilde{c}^2(\widetilde{Q}_n-\widetilde{Q}), \nabla \phi\rangle ds\\
&&\leq \int_0^t \|\tilde{c}_n-\tilde{c}\|_{L^2}\|\tilde{c}_n,\tilde{c}\|_{L^6}\|\widetilde{Q}_n\|_{L^6}\|\nabla \phi\|_{L^6}+\|\tilde{c}\|_{L^6}\|\tilde{c}\|_{L^2}\|\widetilde{Q}_n-\widetilde{Q}\|_{L^6}\|\nabla \phi\|_{L^6}ds\\
&&\leq \left(\|\widetilde{Q}_n\|_{L^\infty(0,T;L^6(\mathcal{D}))}+\|\tilde{c}\|_{L^\infty(0,T;L^2(\mathcal{D}))}\right)\|\nabla \phi\|_{L^\infty(0,T; L^6(\mathcal{D}))}\\ &&\quad\times \left(\int_0^t \|\tilde{c}_n-\tilde{c}\|_{L^2}^2+\| \widetilde{Q}_n-\widetilde{Q}\|_{H^1}^2ds\right)^{\frac{1}{2}}\left(\int_0^t (1+\|\tilde{c}_n, \tilde{c} \|_{L^6})^2ds\right)^{\frac{1}{2}}\rightarrow 0.
\end{eqnarray*}
Similarly, we have $\widetilde{\mathbb{P}}$ a.s.
\begin{eqnarray*}
\int_0^t\langle {\rm F}(\widetilde{Q}_n){\rm I}_3-{\rm F}(\widetilde{Q}){\rm I}_3,\nabla \phi\rangle ds\rightarrow 0.
\end{eqnarray*}
This completes the proof.
\end{proof}
\begin{lemma}\label{lem3.6}  For any $\varphi\in L^\infty(0,T;L^{3}(\mathcal{D}))$ and $t\in [0,T]$, the following convergence holds $\widetilde{\mathbb{P}}$ \mbox{a.s.}
\begin{eqnarray*}
&&\int_0^t\langle(\tilde{u}_n\cdot \nabla)\widetilde{Q}_n+\widetilde{Q}_n\widetilde{\Psi}_n-\widetilde{\Psi}_n\widetilde{ Q}_n-\Gamma {\rm H}(\widetilde{Q}_n,\tilde{c}_n),\varphi\rangle ds\\
&&\rightarrow \int_0^t\langle(\tilde{u}\cdot \nabla)\widetilde{Q}+\widetilde{Q}\widetilde{\Psi}-\widetilde{\Psi} \widetilde{Q}-\Gamma {\rm H}(\widetilde{Q},\tilde{c}), \varphi\rangle ds,
\end{eqnarray*}
as $n\rightarrow \infty$.
\end{lemma}
\begin{proof} Decompose
\begin{eqnarray*}
&&\int_0^t\langle(\tilde{u}_n\cdot \nabla)\widetilde{Q}_n-(\tilde{u}\cdot \nabla)\widetilde{Q},\varphi\rangle ds\\&&=\int_0^t\langle(\tilde{u}_n\cdot \nabla)(\widetilde{Q}_n-\widetilde{Q}),\varphi\rangle ds+\int_0^t\langle(\tilde{u}_n-\tilde{u})\cdot \nabla)\widetilde{Q},\varphi\rangle ds\\&&=:J_1+J_2.
\end{eqnarray*}
For $J_1$, using the Proposition \ref{pro3.1}(\ref{3.18}), (\ref{3.22}) and the H\"{o}lder inequality, to get $\widetilde{\mathbb{P}}$ a.s.
\begin{eqnarray*}
&&~|J_1|\leq \int_0^t\|\tilde{u}_n\|_{L^6}\|\nabla(\widetilde{Q}_n-\widetilde{Q})\|_{L^2}\|\varphi\|_{L^3}ds\\
&&\quad\quad\leq \|\varphi\|_{L^{\infty}(0,T;L^3(\mathcal{D}))}\left(\int_0^t\|\nabla(\widetilde{Q}_n-\widetilde{Q})\|_{L^2}^2ds\right)^\frac{1}{2}
\left(\int_0^t\|\tilde{u}_n\|_{H^1}^2ds\right)^\frac{1}{2}\rightarrow 0.
\end{eqnarray*}
Also, we have $J_2\rightarrow 0$ as $n\rightarrow\infty$ which is the result of the convergence of $\tilde{u}_n$ in $\mathcal{X}_u$, $\widetilde{\mathbb{P}}$ a.s.
By the similar argument as the first term, using the strong convergences of $\widetilde{Q}_n$ in $\mathcal{X}_Q$ and $\tilde{c}_n$ in $\mathcal{X}_c$, $\widetilde{\mathbb{P}}$ a.s. and $(\ref{3.22}), (\ref{3.25}), (\ref{3.26})$, we have $\widetilde{\mathbb{P}}$ a.s.
\begin{eqnarray*}
\int_0^t\langle\widetilde{Q}_n\widetilde{\Psi}_n-\widetilde{\Psi}_n\widetilde{ Q}_n-\Gamma {\rm H}(\widetilde{Q}_n,\tilde{c}_n),\varphi\rangle ds
\rightarrow \int_0^t\langle\widetilde{Q}\widetilde{\Psi}-\widetilde{\Psi} \widetilde{Q}-\Gamma {\rm H}(\widetilde{Q},\tilde{c}), \varphi\rangle ds.
\end{eqnarray*}
as $n\rightarrow\infty$.
\end{proof}

For the sake of elaborating the convergence of term $\epsilon\nabla \tilde{\rho}_n\cdot\nabla \tilde{u}_n$, Feireisl-Novotn\'{y}-Petzeltov\'{a} \cite{Feireisl} showed
$\tilde{\rho}_n\rightarrow \tilde{\rho}$ in $L^2(0,T;H^1(\mathcal{D}))$, $\widetilde{\mathbb{P}}$ a.s. Then, we have
\begin{eqnarray}\label{3.34}
\nabla\tilde{\rho}_n\cdot\nabla \tilde{u}_n \rightarrow \nabla\tilde{\rho}\cdot\nabla \tilde{u} ~{\rm in }~L^\infty([0,T]\times\mathcal{D})', ~ \widetilde{\mathbb{P}}~ \mbox{a.s.}
\end{eqnarray}
Moreover, \cite{Hofmanova, DWang, 16} give
\begin{eqnarray}\label{3.35}
\tilde{\rho}_{n}\tilde{u}_{n}\otimes \tilde{u}_{n}\rightarrow \tilde{\rho}\tilde{u}\otimes \tilde{u} ~{\rm in}~~L^\infty([0,T]\times\mathcal{D})',~\widetilde{\mathbb{P}}~\mbox{a.s.}
\end{eqnarray}
and for $q\in [1, \frac{2\beta}{\beta+1})$
\begin{eqnarray}\label{3.36}
\tilde{\rho}_{n}\tilde{u}_{n}\rightarrow \tilde{\rho}\tilde{u} ~{\rm in}~ L^q([0,T]\times \mathcal{D}),~\widetilde{\mathbb{P}}~\mbox{a.s.}
\end{eqnarray}
Furthermore, using the Proposition \ref{pro3.1}(\ref{3.18}), (\ref{3.22}), (\ref{3.25}), we have
\begin{eqnarray}\label{3.36*}
\tilde{u}_n\cdot\nabla \tilde{c}_n \rightarrow \tilde{u}\cdot\nabla \tilde{c} ~{\rm in }~L^\infty([0,T]\times\mathcal{D})', ~ \widetilde{\mathbb{P}}~ \mbox{a.s.}
\end{eqnarray}

Define the functional for any $\phi\in \cup X_n$
\begin{align*}
\mathcal{N}(\rho,u, c, Q)&=\int_{\mathcal{D}}\rho u\phi dx-\int_{\mathcal{D}}m(0)\phi dx\nonumber\\
&-\int_{0}^{t}\int_{\mathcal{D}}(\rho u\otimes u)-\mu_1\nabla u)\nabla\phi-((\mu_1+\mu_2){\rm div} u+\rho^{\gamma}+\delta\rho^{\beta}){\rm div}\phi dxds\nonumber\\&+ \int_{0}^{t}\int_{\mathcal{D}}\sigma^*(c^2Q)\nabla\phi dxds-\int_{0}^{t}\int_{\mathcal{D}}\epsilon  \phi\nabla \rho\cdot\nabla u dxds\nonumber\\ &+\int_{0}^{t}\int_{\mathcal{D}}({\rm F}(Q){\rm I}_3-\nabla Q\odot \nabla Q+Q\triangle Q-\triangle Q Q)\nabla\phi dxds.
\end{align*}
Following ideas of \cite{Hofmanova, ZM}, we are able to obtain the limit $(\tilde{c}, \tilde{\rho}, \tilde{u}, \widetilde{Q}, \widetilde{\mathcal{W}})$ satisfies the momentum equation once we show that the process $\mathcal{N}(\tilde{c},\tilde{\rho}, \tilde{u},\widetilde{Q})_t$ is a square integral martingale and its quadratic and cross variations satisfy
\begin{eqnarray}
&& \ll \mathcal{N}(\tilde{c},\tilde{\rho}, \tilde{u},\widetilde{Q})_t\gg=\sum_{k\geq 1}\int_{0}^{t}\langle \tilde{\rho}f(\tilde{\rho}, \tilde{\rho} \tilde{u},\tilde{c}, \widetilde{Q})\beta_k,\phi\rangle ds,\label{3.41}\\
&&\ll \mathcal{N}(\tilde{c},\tilde{\rho}, \tilde{u},\widetilde{Q})_t, \tilde{\beta}_k\gg=\int_{0}^{t}\langle \tilde{\rho}f(\tilde{\rho}, \tilde{\rho} \tilde{u},\tilde{c}, \widetilde{Q})\beta_k,\phi\rangle ds.\label{3.42}
\end{eqnarray}
Here, we only focus on the noise term. It is enough to show that $\widetilde{\mathbb{P}}\otimes \mathcal{L}~\mbox{a.e.}$
\begin{align}
&&\langle \mathcal{M}^{\frac{1}{2}}(\tilde{\rho}_n)P_n(\sqrt{\tilde{\rho}_n} f(\tilde{\rho}_n, \tilde{\rho}_n \tilde{u}_n,\tilde{c}_n, \widetilde{Q}_n)\cdot),\phi\rangle\rightarrow \langle \tilde{\rho}f(\tilde{\rho}, \tilde{\rho} \tilde{u},\tilde{c}, \widetilde{Q})\cdot,\phi\rangle ~{\rm in} ~L_{2}(\mathcal{H};\mathbb{R}).  \label{3.39}
\end{align}
Toward proving the convergence, we estimate by the Minkowski inequality
\begin{eqnarray*}
&&\left\|\langle \mathcal{M}^{\frac{1}{2}}(\tilde{\rho}_n)P_n(\sqrt{\tilde{\rho}_n} f(\tilde{\rho}_n, \tilde{\rho}_n \tilde{u}_n,\tilde{c}_n, \widetilde{Q}_n)\cdot,\phi\rangle-\langle \tilde{\rho}f(\tilde{\rho}, \tilde{\rho} \tilde{u},\tilde{c}, \widetilde{Q})\cdot,\phi\rangle\right\|_{L_{2}(\mathcal{H};\mathbb{R})}\\
&&\leq C\left\|\mathcal{M}^{\frac{1}{2}}(\tilde{\rho}_n)P_n(\sqrt{\tilde{\rho}_n} f(\tilde{\rho}_n, \tilde{\rho}_n \tilde{u}_n,\tilde{c}_n, \widetilde{Q}_n)-\tilde{\rho}f(\tilde{\rho}, \tilde{\rho} \tilde{u},\tilde{c}, \widetilde{Q})\right\|_{L_{2}(\mathcal{H};H^{-k})}\\
&&\leq C\left(\sum_{k\geq 1}\|\tilde{\rho}_n f(\tilde{\rho}_n, \tilde{\rho}_n \tilde{u}_n,\tilde{c}_n, \widetilde{Q}_n)e_k-\tilde{\rho}f(\tilde{\rho}, \tilde{\rho} \tilde{u},\tilde{c}, \widetilde{Q})e_k\|_{L^1}^2\right)^\frac{1}{2}\\
&&\quad+\left\|\mathcal{M}^{\frac{1}{2}}(\tilde{\rho}_n)P_n(\sqrt{\tilde{\rho}_n}f(\tilde{\rho}_n, \tilde{\rho}_n \tilde{u}_n,\tilde{c}_n, \widetilde{Q}_n)-\tilde{\rho}_nf(\tilde{\rho}_n, \tilde{\rho}_n \tilde{u}_n,\tilde{c}_n, \widetilde{Q}_n)\right\|_{L_{2}(\mathcal{H};H^{-k})}\\
&&\leq C\int_{\mathcal{D}}\left(\sum_{k\geq 1}|\tilde{\rho}_n f(\tilde{\rho}_n, \tilde{\rho}_n \tilde{u}_n,\tilde{c}_n, \widetilde{Q}_n)e_k-\tilde{\rho}f(\tilde{\rho}, \tilde{\rho} \tilde{u},\tilde{c}, \widetilde{Q})e_k|^2\right)^\frac{1}{2}dx\\
&&\quad+\left\|\mathcal{M}^{\frac{1}{2}}(\tilde{\rho}_n)P_n(\sqrt{\tilde{\rho}_n} f(\tilde{\rho}_n, \tilde{\rho}_n \tilde{u}_n,\tilde{c}_n, \widetilde{Q}_n))-\tilde{\rho}_nf(\tilde{\rho}_n, \tilde{\rho}_n \tilde{u}_n,\tilde{c}_n, \widetilde{Q}_n)\right\|_{L_{2}(\mathcal{H};H^{-k})}\\
&&\leq \mathcal{J}_{1}+\mathcal{J}_{2}.
\end{eqnarray*}
Next, we show that $\mathcal{J}_{1}, \mathcal{J}_2\rightarrow 0$, as $n\rightarrow \infty$, $~\widetilde{\mathbb{P}}\otimes \mathcal{L}~\mbox{a.e.}$ Indeed,
by condition (\ref{2.2}) as well as Proposition \ref{pro3.1}, we have
\begin{eqnarray*}
&&|\mathcal{J}_{1}|\leq C\|\tilde{\rho}_n-\tilde{\rho}, \tilde{\rho}_n \tilde{u}_n-\tilde{\rho}\tilde{u}, \tilde c_n-\tilde c,  \widetilde Q_n-\widetilde Q\|_{L^\frac{2\gamma}{\gamma+1}}\rightarrow 0.
\end{eqnarray*}
Also, using the H\"{o}lder inequality, condition (\ref{2.1}), the bound \eqref{3.8} and Proposition \ref{pro3.1}, we have $\mathcal{J}_2\rightarrow 0$, as $n\rightarrow \infty$, see also \cite[Proposition 4.11]{Hofmanova}. Then, (\ref{3.39}) follows. We could obtain equalities (\ref{3.41}), \eqref{3.42} by combining (\ref{3.34})-(\ref{3.36}), (\ref{3.39}), Proposition \ref{pro3.1} and the Vitali convergence theorem 6.1.

Using the same argument as above, we infer that it holds $\widetilde{\mathbb{P}}$ a.s.
\begin{align*}
&\int_{\mathcal{D}}\tilde{c}(t)\ell dx=\int_{\mathcal{D}}\tilde{c}(0)\ell dx-\int_{0}^{t}\int_{\mathcal{D}} (\tilde{u}\cdot\nabla)\tilde{c} \cdot\ell dxds-\int_{0}^{t}\int_{\mathcal{D}} \nabla \tilde{c} \cdot\nabla\ell dxds,\\
&\int_{\mathcal{D}}\widetilde{Q }(t)\varphi dx=\int_{\mathcal{D}}\widetilde{Q}(0) \varphi dx-\int_{0}^{t}\int_{\mathcal{D}}((\tilde{u}\cdot \nabla)\widetilde{Q}+\widetilde{Q}\widetilde{\Psi}-\widetilde{\Psi} \widetilde{Q})\varphi dxds\nonumber\\&\quad\qquad\qquad\quad+\int_{0}^{t}\int_{\mathcal{D}}\Gamma\varphi {\rm H}(\widetilde{Q},\tilde{c})dxds,
\end{align*}
for $\ell\in C^\infty(\mathcal{D}), \varphi\in C^\infty(\mathcal{D})$, $t\in [0,T]$. We summarize the result for this section,
\begin{proposition} For $\beta>{\rm max}\{6,\gamma\}$, fixed $\delta>0$. If conditions (\ref{2.1}), (\ref{2.2}) hold. There exists a global weak martingale solution to modified system (\ref{Equ3.1})-(\ref{3.5}).
\end{proposition}

\maketitle
\section{The existence of martingale solution for vanishing artificial viscosity}
In this section, we let $\epsilon\rightarrow 0$ to build the existence of global weak martingale solution to the following system
\begin{eqnarray}\label{Equ4.1}
\left\{\begin{array}{ll}
\partial_{t}c+(u\cdot \nabla)c=\triangle c,\\
\partial_{t}\rho+{\rm div}(\rho u)=0,\\
\partial_{t}(\rho u)+{\rm div}(\rho u\otimes u)+\nabla (\rho^\gamma+\delta\rho^\beta)
 =\mu_1\Delta u+(\mu_1+\mu_2)\nabla ({\rm div} u)+\sigma^*\nabla\cdot(c^2Q)\\ \qquad\qquad+\nabla\cdot({\rm F}(Q){\rm I}_3-\nabla Q\odot \nabla Q)+\nabla\cdot(Q\triangle Q-\triangle Q Q) +\rho f(\rho,\rho u, c, Q)\frac{d\mathcal{W}}{dt},\\
\partial_{t}Q+(u\cdot \nabla)Q+Q\Psi-\Psi Q=\Gamma {\rm H}(Q,c).
\end{array}\right.
\end{eqnarray}

The solutions $(\rho_{\epsilon,\delta}, u_{\epsilon,\delta}, c_{\epsilon,\delta}, Q_{\epsilon,\delta})$ obtained in the first level approximation will be used for the approximate solution in this section, which shares the same energy bounds with (\ref{3.21})-(\ref{3.26}). Namely,
\begin{eqnarray}
&&\rho_{\epsilon,\delta} u_{\epsilon,\delta}\in L^{p}(\Omega; L^{\infty}(0,T; L^{\frac{2\beta}{\beta+1}}(\mathcal{D}))),\label{4.2}\\
&&\rho_{\epsilon,\delta}\in L^{p}(\Omega; L^{\infty}(0,T; L^{\beta}(\mathcal{D}))),\label{4.3}\\
&&u_{\epsilon,\delta}\in L^{p}(\Omega; L^{2}(0,T; H^{1}(\mathcal{D}))),\label{4.4}\\
&&\sqrt{\rho_{\epsilon,\delta}} u_{\epsilon,\delta}\in L^{p}(\Omega; L^{\infty}(0,T;L^{2}(\mathcal{D}))),\label{4.4*}\\
&&c_{\epsilon,\delta}\in L^{p}(\Omega; L^{\infty}(0,T; L^{2}(\mathcal{D}))\cap L^{2}(0,T; H^{1}(\mathcal{D}))),\label{4.5}\\
&&Q_{\epsilon,\delta}\in L^{p}(\Omega; L^{\infty}(0,T; H^{1}(\mathcal{D}))\cap L^{2}(0,T; H^{2}(\mathcal{D}))).\label{4.6}
\end{eqnarray}
The proof also consists of the argument of tightness and identifying the limit. Note that, here we can not make use of the a priori bound $\sqrt{\epsilon}\rho_{\epsilon,\delta} \in L^{p}(\Omega; L^{2}(0,T;H^{1}(\mathcal{D})))$ to gain the tightness of the distribution of density on path space $L^2(0,T;L^2(\mathcal{D}))$ . Therefore, we are not able to identify the pressure and stochastic term. To overcome this difficulty, we first improve the integrability of density. We replace $(\rho_{\epsilon,\delta}, u_{\epsilon,\delta}, c_{\epsilon,\delta}, Q_{\epsilon, \delta})$ by $(\rho_{\epsilon}, u_{\epsilon}, c_{\epsilon}, Q_\epsilon)$ to simplify the notation.

Recall the operator $\mathcal{T}$ constructed by Bogovskii \cite{19} related to the problem
\begin{eqnarray*}
{\rm div}v=f,~~v|_{\partial \mathcal{D}}=0,
\end{eqnarray*}
with the following properties:

1.  $\mathcal{T}:\left\{f\in L^p: \int_{\mathcal{D}}fdx=0\right\}\rightarrow H^{1,p}_0(\mathcal{D})$ is a bounded linear operator such that for all $p>1$
\begin{eqnarray}\label{4.7}
\|\mathcal{T}(f)\|_{H^{1,p}_0(\mathcal{D})}\leq C\|f\|_{L^p(\mathcal{D})}.
\end{eqnarray}

2. $v=\mathcal{T}(f)$ is a solution to above equation.

3. For any function $g\in L^r(\mathcal{D})$ with $g\cdot \vec{n}|_{\partial \mathcal{D}}=0$, it holds
\begin{eqnarray}\label{4.8}
\|\mathcal{T}({\rm div}g)\|_{L^{r}(\mathcal{D})}\leq C\|g\|_{L^r(\mathcal{D})}.
\end{eqnarray}
The proof of above properties, we refer the readers to \cite{20,21} and the references therein for details.
\begin{lemma} \label{lem4.1} The approximate sequence $\rho_{\epsilon}$ satisfies the following estimate
\begin{eqnarray*}
\mathbb{E}\int_{0}^{T}\int_{\mathcal{D}}\rho_{\epsilon}^{\gamma+1}+\delta\rho_{\epsilon}^{\beta+1}dxdt\leq C,
\end{eqnarray*}
where the constant $C$ is independent of $\epsilon$.
\end{lemma}
\begin{proof} The proof is similar to that of \cite{Feireisl}. Applying the It\^{o} formula to function $\Phi(\rho_\epsilon u_\epsilon,\\ \mathcal{T}[\rho_\epsilon-(\rho_0)_m] )=\int_{\mathcal{D}}\rho_\epsilon u_\epsilon\cdot \mathcal{T}[\rho_\epsilon-(\rho_0)_m] dx$,
\begin{eqnarray}\label{4.9}
&&\int_{\mathcal{D}}\rho_\epsilon u_\epsilon\cdot \mathcal{T}[\rho_\epsilon-(\rho_0)_m] dx=\int_{0}^{t}\int_{\mathcal{D}}\rho_{\epsilon}^{\gamma+1}+\delta\rho_{\epsilon}^{\beta+1}dxds+\int_{\mathcal{D}}m_{0}\cdot \mathcal{T}[\rho_0-(\rho_0)_m]dx\nonumber
\\&&\quad\quad-\mu_1\int_{0}^{t}\int_{\mathcal{D}}\nabla u_{\epsilon}:\nabla \mathcal{T}[\rho_\epsilon-(\rho_0)_m]dxds -(\mu_1+\mu_2)\int_{0}^{t}\int_{\mathcal{D}}{\rm div}u_{\epsilon}\cdot [\rho_\epsilon-(\rho_0)_m] dxds\nonumber
\\&&\quad\quad+\int_{0}^{t}\int_{\mathcal{D}}\rho_{\epsilon} u_{\epsilon}\otimes u_{\epsilon}:\nabla \mathcal{T}[\rho_\epsilon-(\rho_0)_m]dxds
-\int_{0}^{t}(\rho_0)_m\int_{\mathcal{D}}\rho_{\epsilon}^{\gamma}
+\delta\rho_{\epsilon}^{\beta}dxds\nonumber \\
&&\quad\quad-\epsilon\int_{0}^{t}\int_{\mathcal{D}}\nabla u_{\epsilon}\nabla \rho_{\epsilon}\cdot\mathcal{T}[\rho_\epsilon-(\rho_0)_m]dx ds-\int_{0}^{t}\int_{\mathcal{D}} \rho_{\epsilon}u_{\epsilon} \cdot \mathcal{T} [{\rm div} (\rho_{\epsilon}u_{\epsilon})]dxds\nonumber \\
&&\quad\quad+\epsilon\int_{0}^{t}\int_{\mathcal{D}} \rho_{\epsilon}u_{\epsilon}\cdot\mathcal{T}[\triangle\rho_{\epsilon}] dxds-\int_{0}^{t}\int_{\mathcal{D}}\sigma^*c_\epsilon^2Q_\epsilon:\nabla\mathcal{T}[\rho_\epsilon-(\rho_0)_m]dxds\nonumber\\
&&\quad\quad-\int_{0}^{t}\int_{\mathcal{D}}(\nabla Q_\epsilon\odot \nabla Q_\epsilon-{\rm F(Q_\epsilon)}{\rm I}_3):\nabla\mathcal{T}[\rho_\epsilon-(\rho_0)_m]dxds\nonumber\\
&&\quad\quad-\int_{0}^{t}\int_{\mathcal{D}}(Q_\epsilon \triangle Q_\epsilon-\triangle Q_\epsilon Q_\epsilon):\nabla\mathcal{T}[\rho_\epsilon-(\rho_0)_m]dxds\nonumber\\
&&\quad\quad+\int_{0}^{t}\int_{\mathcal{D}}\rho_{\epsilon} f(\rho_{\epsilon},\rho_{\epsilon}u_{\epsilon}, c_{\epsilon}, Q_\epsilon)\cdot \mathcal{T}[\rho_\epsilon-(\rho_0)_m] dxd\mathcal{W}\nonumber\\
&&\qquad=\int_{0}^{t}\int_{\mathcal{D}}\rho_{\epsilon}^{\gamma+1}+\delta\rho_{\epsilon}^{\beta+1}dxds
+J_{0}+\int_{0}^{t}J_1+\cdots+J_{10}ds+\int_{0}^{t}J_{11}d\mathcal{W}.
\end{eqnarray}
Taking expectation on both sides of (\ref{4.9}), rearranging and to obtain
\begin{eqnarray}\label{4.10}
&&\mathbb{E}\int_{0}^{t}\int_{\mathcal{D}}\rho_{\epsilon}^{\gamma+1}+\delta\rho_{\epsilon}^{\beta+1}dxds\nonumber\\
&&=-\mathbb{E}\left(J_0+\sum_{i=1}^{10}\int_{0}^{t}J_{i}ds\right)
+\mathbb{E}\int_{\mathcal{D}}\rho_{\epsilon}u_{\epsilon}\cdot \mathcal{T}[\rho_\epsilon-(\rho_0)_m] dx,
\end{eqnarray}
in fact $\mathbb{E}\int_{0}^{t}J_{11}d\mathcal{W}=0$, since the process $\int_{0}^{t}J_{11}d\mathcal{W}$ is a square integrable martingale. Indeed, using condition (\ref{2.1})
\begin{eqnarray*}
&&\mathbb{E}\int_{0}^{t}\sum_{k\geq 1}\left(\int_{\mathcal{D}}\rho_{\epsilon} f(\rho_{\epsilon},\rho_{\epsilon}u_{\epsilon}, c_{\epsilon}, Q_\epsilon)\cdot \mathcal{T}[\rho_\epsilon-(\rho_0)_m]e_kdx\right)^2ds\\
&&\leq C\mathbb{E}\left(\|\mathcal{T}[\rho_\epsilon-(\rho_0)_m]\|_{L^\infty([0,T]\times \mathcal{D})}^2\int_{0}^{t}\int_{\mathcal{D}}\sum_{k\geq 1}\rho_{\epsilon}| f(\rho_{\epsilon},\rho_{\epsilon}u_{\epsilon}, c_{\epsilon}, Q_\epsilon)e_k|^2dx\int_{\mathcal{D}}\rho_\epsilon dxds\right)\\
&&\leq C\mathbb{E}\|\mathcal{T}[\rho_\epsilon-(\rho_0)_m]\|_{L^\infty([0,T]\times \mathcal{D})}^4\\&&\quad+C\mathbb{E}\left(\int_{\mathcal{D}}\rho_\epsilon(0) dx\right)^4\mathbb{E}\left(\int_{0}^{t}\int_{\mathcal{D}}\sum_{k\geq 1}\rho_{\epsilon}| f(\rho_{\epsilon},\rho_{\epsilon}u_{\epsilon}, c_{\epsilon}, Q_\epsilon)e_k|^2dxds\right)^4\\
&&\leq C\mathbb{E}\|\mathcal{T}[\rho_\epsilon-(\rho_0)_m]\|_{L^\infty([0,T]\times \mathcal{D})}^4+C\delta^{-\frac{1}{\beta}}\mathbb{E}\left(\int_{0}^{t }\int_{\mathcal{D}}\rho_\epsilon^\gamma+|\sqrt{\rho_\epsilon} u_\epsilon|^2+c_\epsilon^2+|Q_\epsilon|^2dxds\right)^4\\
&&\leq C(\delta).
\end{eqnarray*}
The desired result follows once each term on the right hand side of (\ref{4.10}) can be controlled. By the H\"{o}lder inequality, (\ref{4.2}) and (\ref{4.8}), obtain
\begin{eqnarray*}
&&\left|-\mathbb{E}\int_{0}^{t}J_{3}+J_6ds\right|\\ &&\leq C\mathbb{E}\int^T_0\|\rho_{\epsilon}u_{\epsilon}\|_{L^\frac{2\beta}{\beta+1}}\|u_\epsilon\|_{L^6}\|\rho_\epsilon-(\rho_0)_m\|_{L^\beta}
+\|\rho_\epsilon\|_{L^\beta}\|u_\epsilon\|_{L^6}\| \mathcal{T} [{\rm div} (\rho_{\epsilon}u_{\epsilon})]\|_{L^\frac{2\beta}{\beta+1}}dt\\
&&\leq C\mathbb{E}\int^T_0\|\rho_{\epsilon}u_{\epsilon}\|_{L^\frac{2\beta}{\beta+1}}\|\nabla u_\epsilon\|_{L^2}\|\rho_\epsilon-(\rho_0)_m\|_{L^\beta}dt\\
&&\leq C\mathbb{E}\left(\sup_{t\in [0,T]}\|\rho_{\epsilon}u_{\epsilon}\|_{L^\frac{2\beta}{\beta+1}}\|\rho_\epsilon-(\rho_0)_m\|_{L^\beta}\int^T_0\|\nabla u_\epsilon\|_{L^2}dt\right)\leq C.
\end{eqnarray*}
Again, using the H\"{o}lder inequality, the Sobolev embedding $H^{1,p}(\mathcal{D})\hookrightarrow L^{\infty}(\mathcal{D})$ for $p>3$ and (\ref{4.7})
\begin{eqnarray*}
&&\left|\mathbb{E}\int_{0}^{t}-J_{5}ds\right|\leq \epsilon C\mathbb{E}\int^T_0\|\nabla u_\epsilon\|_{L^2}\|\nabla \rho_\epsilon\|_{L^2}\|\mathcal{T}[\rho_\epsilon-(\rho_0)_m]\|_{L^\infty}dt\\
&&\qquad\qquad\quad\quad\leq \epsilon C\mathbb{E}\int^T_0\|\nabla u_\epsilon\|_{L^2}\|\nabla \rho_\epsilon\|_{L^2}\|\mathcal{T}(\rho_\epsilon)\|_{H^{1,\beta}}dt\\
&&\qquad\qquad\quad\quad\leq \epsilon C\mathbb{E}\left[\sup_{t\in [0,T]}\|\rho_\epsilon\|_{L^\beta}\right]\mathbb{E}\int^T_0\|\nabla u_\epsilon\|_{L^2}^2dt\mathbb{E}\int^T_0\|\nabla \rho_\epsilon\|_{L^2}^2dt\leq C.
\end{eqnarray*}
Moreover, by the bounds (\ref{4.3}), (\ref{4.4}), (\ref{4.5}), (\ref{4.6}), we have
\begin{align*}
\left|\mathbb{E}\int_{0}^{t}-J_{7}ds\right|&\leq \epsilon C\mathbb{E}\int^T_0\|\nabla u_\epsilon\|_{L^2}\|\rho_\epsilon\|_{L^\beta}\|\nabla \rho_\epsilon\|_{L^2}dt\leq C,\\
\left|\mathbb{E}\int_{0}^{t}-J_{8}ds\right|&\leq C\mathbb{E}\int^T_0\|c_\epsilon\|_{L^6}^2\|Q_\epsilon\|_{L^6}\|\nabla \mathcal{T}[\rho_\epsilon-(\rho_0)_m]\|_{L^2}dt\\
&\leq C \mathbb{E}\left[\sup_{t\in[0,T]}(\|Q_\epsilon\|_{L^6}^4+\|\rho_\epsilon-(\rho_0)_m\|^4_{L^2})\right] \mathbb{E}\left(\int^T_0\|c_\epsilon\|_{L^6}^2dt\right)^2\leq C,\\
\left|\mathbb{E}\int_{0}^{t}-J_{9}ds\right|&\leq C\mathbb{E}\int_0^T\|\nabla Q_\epsilon\|_{L^\frac{10}{3}}^2\|\nabla \mathcal{T}[\rho_\epsilon-(\rho_0)_m]\|_{L^\frac{5}{2}}+(\|Q_\epsilon\|_{L^3}^2\\&\quad+\|Q_\epsilon\|_{L^{6}}^4)\|\rho_\epsilon-(\rho_0)_m\|_{L^3}dt\\
&\leq C\mathbb{E}\int_0^T\| Q_\epsilon\|_{H^2}^2\|\rho_\epsilon-(\rho_0)_m\|_{L^\beta}+(1+\|Q_\epsilon\|_{L^{6}}^4)\|\rho_\epsilon-(\rho_0)_m\|_{L^\beta}dt\\
&\leq C\mathbb{E}\left[\sup_{t\in[0,T]}\|\rho_\epsilon-(\rho_0)_m\|^2_{L^\beta}\right]\mathbb{E}\left(\int_0^T\|\triangle Q_\epsilon\|_{L^2}^2+(1+\|Q_\epsilon\|_{L^{6}}^4)dt\right)^2\\&\leq C,\\
\left|\mathbb{E}\int_{0}^{t}-J_{10}ds\right|&\leq C\mathbb{E}\int_{0}^{T}\|Q_\epsilon\|_{L^4}\|\triangle Q_\epsilon\|_{L^2}\|\nabla \mathcal{T}[\rho_\epsilon-(\rho_0)_m]\|_{L^\beta}dt\\
&\leq C\mathbb{E}\left[\sup_{t\in [0,T]}\|\rho_\epsilon-(\rho_0)_m\|_{L^\beta}^4\right]\mathbb{E}\left[\sup_{t\in [0,T]}\|Q_\epsilon\|_{H^1}^4\right]\mathbb{E}\left(\int_{0}^{T}\|\triangle Q_\epsilon\|_{L^2}^2dt\right)^2\\
&\leq C,
\end{align*}
where $C$ is independent of $\epsilon$. The proof is complete.
\end{proof}

\subsection{Compactness argument}

In order to acquire the compactness of the approximate sequence, we implement the same procedures as the first level approximation.  Define the path space
\begin{eqnarray*}
\mathcal{X}=\mathcal{X}_{u}\times\mathcal{X}_{\rho}\times\mathcal{X}_{\rho u}\times\mathcal{X}_{c}\times\mathcal{X}_{Q}\times\mathcal{X}_{\mathcal{W}},
\end{eqnarray*}
where
\begin{eqnarray*}
&&\mathcal{X}_{\rho}:=L^\infty(0,T; H^{-\frac{1}{2}}(\mathcal{D}))\cap L_w^{\beta+1}(0,T; L^{\beta+1}(\mathcal{D})),
\end{eqnarray*}
and $\mathcal{X}_{u}, \mathcal{X}_{\rho u},  \mathcal{X}_{c}, \mathcal{X}_{Q}, \mathcal{X}_{\mathcal{W}}$ are same as the definition in subsection 3.2. Let $\widetilde{\mathcal{X}}=\mathcal{X}\times L_w^{\frac{\beta+1}{\beta}}((0,T)\times\mathcal{D})$. The set of probability measures $\{\nu^\epsilon\}_{\epsilon>0}$ is constructed similarly to (\ref{3.17}). We also have the following result.
\begin{proposition}\label{pro4.1} There exists a new probability space $(\widetilde{\Omega}, \widetilde{\mathcal{F}}, \widetilde{\mathbb{P}})$, a subsequence $\{\nu^{\epsilon_{k}}\}_{k\geq 1}$ {\rm (still denote by $\epsilon$)} and $\widetilde{\mathcal{X}}$-valued measurable random variables
\begin{eqnarray*}
(\tilde{u}_{\epsilon},\tilde{\rho}_{\epsilon},\tilde{\rho}_{\epsilon}\tilde{u}_{\epsilon}, \tilde{\rho}_{\epsilon}^\gamma+\delta\tilde{\rho}_{\epsilon}^{\beta}, \tilde{c}_{\epsilon}, \widetilde{Q}_{\epsilon},  \widetilde{\mathcal{W}}_{\epsilon})~~ {\rm and}~~
(\tilde{u},\tilde{\rho}, \tilde{\rho}\tilde{u}, \overline{\tilde{\rho}^\gamma}+\delta\overline{\tilde{\rho}^{\beta}}, \tilde{c}, \widetilde{Q}, \widetilde{\mathcal{W}}),
\end{eqnarray*}
such that
\begin{eqnarray}\label{4.11}&&(\tilde{u}_{\epsilon},\tilde{\rho}_{\epsilon},\tilde{\rho}_{\epsilon}\tilde{u}_{\epsilon}, \tilde{c}_{\epsilon}, \widetilde{Q}_{\epsilon}, \widetilde{\mathcal{W}}_{\epsilon})
\rightarrow(\tilde{u},\tilde{\rho}, \tilde{\rho}\tilde{u}, \tilde{c}, \widetilde{Q}, \widetilde{\mathcal{W}}), ~\widetilde{\mathbb{P}}~\mbox{a.s.}
\end{eqnarray}
in the topology of $\mathcal{X}$, moreover, we have $\widetilde{\mathbb{P}}$ a.s.
\begin{eqnarray}\label{4.12}
\tilde{\rho}_{\epsilon}^\gamma+\delta\tilde{\rho}_{\epsilon}^{\beta}\rightarrow \overline{\tilde{\rho}^\gamma}+\delta\overline{\tilde{\rho}^{\beta}}~{\rm in}~L_w^{\frac{\beta+1}{\beta}}((0,T)\times\mathcal{D} ),
\end{eqnarray}
and
\begin{eqnarray*}
&&\widetilde{\mathbb{P}}\left\{(\tilde{u}_{\epsilon},\tilde{\rho}_{\epsilon},\tilde{\rho}_{\epsilon}\tilde{u}_{\epsilon}, \tilde{\rho}_{\epsilon}^\gamma+\delta\tilde{\rho}_{\epsilon}^{\beta}, \tilde{c}_{\epsilon}, \widetilde{Q}_{\epsilon}, \widetilde{\mathcal{W}}_{\epsilon})\in \cdot\right\}=\nu^{\epsilon}(\cdot),\\
&&\widetilde{\mathbb{P}}\left\{(\tilde{u},\tilde{\rho}, \tilde{\rho}\tilde{u}, \overline{\tilde{\rho}^\gamma}+\delta\overline{\tilde{\rho}^{\beta}}, \tilde{c}, \widetilde{Q}, \widetilde{\mathcal{W}})\in \cdot\right\}=\nu(\cdot),
\end{eqnarray*}
where $\nu$ is a Radon measure and $\widetilde{\mathcal{\mathcal{W}}}_{\epsilon}$ is cylindrical Wiener process, relative to the filtration $\widetilde{\mathcal{F}}_{t}^{\epsilon}$ generated by the completion of $\sigma(\tilde{u}_{\epsilon}(s),\tilde{\rho}_{\epsilon}(s), \tilde{c}_\epsilon(s), \widetilde{Q}_\epsilon(s), \widetilde{\mathcal{W}}_{\epsilon}(s);s\leq t)$.
In addition,
\begin{align}
&\tilde{\rho}_{\epsilon}{\rm div} \tilde{u}_{\epsilon}\rightarrow \overline{{\tilde{\rho}_{\epsilon}{\rm div} \tilde{u}_{\epsilon}}}~{\rm weak ly~in}~L^p(\widetilde{\Omega}; L^2(0,T;L^\frac{2\beta}{\beta+2})),\label{4.14*}\\
&\tilde{\rho}_{\epsilon}\ln \tilde{\rho}_{\epsilon}\rightarrow\overline{\tilde{\rho}_{\epsilon}\ln \tilde{\rho}_{\epsilon}}~{\rm weakly ~star~in}~L^p(\widetilde{\Omega}; L^2(0,T;L^\frac{2\beta}{\beta+2})).\label{4.15*}
\end{align}
Furthermore, the process $(\tilde{u}_{\epsilon},\tilde{\rho}_{\epsilon},\tilde{\rho}_{\epsilon}\tilde{u}_{\epsilon}, \tilde{c}_{\epsilon}, \widetilde{Q}_{\epsilon}, \widetilde{\mathcal{W}}_{\epsilon})$ also satisfies the system (\ref{Equ4.1}) and shares the uniform bounds with (\ref{4.2})-(\ref{4.6}).
\end{proposition}
\begin{lemma}\label{lem4.2}The sequence of probability measures $\{\nu^{\epsilon}\}_{\epsilon>0}$ is tight on path space $\widetilde{\mathcal{X}}$.
\end{lemma}
\begin{proof} In order not to repeat the trivial procedures, we mainly show limited parts that are different from the Lemma \ref{lem3.4}.  Decompose $\rho_{\epsilon}u_{\epsilon}=X^{\epsilon}+Y^{\epsilon}$, where
\begin{eqnarray*}
&&X_\epsilon=m_{0}+\int_{0}^{t}-{\rm div}(\rho_\epsilon u_\epsilon\otimes u_\epsilon)-\nabla (\rho^\gamma_\epsilon+\delta\rho_\epsilon^\beta)
 +\mu_1\Delta u_\epsilon+(\mu_1+\mu_2)\nabla ({\rm div} u_\epsilon)\\&&\qquad\qquad+\nabla\cdot({\rm F}(Q_\epsilon){\rm I}_3-\nabla Q_\epsilon\odot \nabla Q_\epsilon)+\nabla\cdot(Q_\epsilon\triangle Q_\epsilon-\triangle Q_\epsilon Q_\epsilon)+\sigma^*\nabla\cdot(c_\epsilon^2Q_\epsilon)ds\\ &&\qquad\qquad+\int_{0}^{t}\rho_\epsilon f(\rho_\epsilon,\rho_\epsilon u_\epsilon, c_\epsilon, Q_\epsilon)d\mathcal{W},
\end{eqnarray*}
and
\begin{eqnarray*}
Y_\epsilon=\epsilon\int_{0}^{t}\nabla \rho_\epsilon\cdot\nabla u_\epsilon ds.
\end{eqnarray*}

For the process $X_{\epsilon}$, one can treat it by the same argument as Lemma \ref{lem3.4} Claim 1, obtaining $\mathbb{E}\|X_{\epsilon}(t)\|_{C^{\alpha}([0,T]; H^{-k}(\mathcal{D}))}\leq C$ for any $\alpha\in [0,\frac{1}{2}), k\geq\frac{5}{2}$.

For process $Y_{\epsilon}$, Using the bound (\ref{4.4}),
we have $\mathbb{P}~\mbox{a.s.}$
\begin{align*}
\int_{0}^{t}\int_{\mathcal{D}}|\epsilon\nabla\rho_{\epsilon}\cdot\nabla u_{\epsilon}|dxds&\leq \sqrt{\epsilon}\left(\int_{0}^t\|\nabla u_{\epsilon}\|_{L^2}^2ds\right)^{\frac{1}{2}}\left(\int_{0}^{t}\|\sqrt{\epsilon}\nabla \rho_{\epsilon}\|_{L^2}^2ds\right)^{\frac{1}{2}}\\
&\leq C\sqrt{\epsilon}\rightarrow 0,~ {\rm as}~ \epsilon\rightarrow 0,
\end{align*}
which leads to $Y_{\epsilon}\rightarrow 0 ~{\rm in} ~C([0,T];L^{1}(\mathcal{D})),~ \mathbb{P}~\mbox{a.s.}$ Note that, the convergence a.s. implies the convergence in distribution, therefore, we have
\begin{eqnarray*}
Y_{\epsilon}\rightarrow 0,~ {\rm in} ~C([0,T];L^{1}(\mathcal{D})),
\end{eqnarray*}
in the sense of distribution. The Sobolev compactness embedding $L^{1}(\mathcal{D})\hookrightarrow H^{-k}(\mathcal{D})$ for $k>\frac{3}{2}$, implies that there exists a compact set $\mathcal{K}\subset C([0,T];H^{-k}(\mathcal{D}))$ such that $\mathbb{P}(Y_{\epsilon}\subset \mathcal{K}^{\mathfrak{c}})<\epsilon$. We obtain the law of set $\{\mathbb{P}\circ(Y_{\epsilon})^{-1}\}$ is tight on space $C([0,T];H^{-k}(\mathcal{D}))$.

Define the set $\widetilde{\mathcal{K}}=\mathcal{K}_{1}\cap(\mathcal{K}_{2}+\mathcal{K})$, where
\begin{eqnarray*}
&&\mathcal{K}_{1}:=\left\{\varphi\in L^{\infty}(0,T;L^{\frac{2\beta}{\beta+1}}(\mathcal{D})): \|\varphi\|_{L^{\infty}(0,T;L^{\frac{2\beta}{\beta+1}}(\mathcal{D}))}\leq R\right\},\\
&&\mathcal{K}_{2}:=\left\{\varphi\in C^{\alpha}([0,T];H^{-k}(\mathcal{D})): \|\varphi\|_{C^{\alpha}([0,T];H^{-k}(\mathcal{D}))}\leq R\right\}.
\end{eqnarray*}
Then, by the Aubin-Lions lemma \ref{lem6.1}, the set $\mathcal{K}$ is relatively compact in $C([0,T];H^{-1}(\mathcal{D}))$. Using the bound (\ref{4.2}) and the Chebyshev inequality, to conclude
\begin{eqnarray*}
&&\mu_{\rho u}^{\epsilon}(\widetilde{\mathcal{K}}^{\mathfrak{c}})\leq \mathbb{P}\left(\|\rho_{\epsilon}u_{\epsilon}\|_{L^{\infty}(0,T;L^{\frac{2\beta}{\beta+1}}(\mathcal{D}))}>R\right)
+\mathbb{P}\left(\|X_{\epsilon}\|_{C^{\alpha}([0,T];H^{-k}(\mathcal{D}))}>R\right)\\
&&\qquad\qquad+\mathbb{P}(Y_{\epsilon}\subset \mathcal{K}^{\mathfrak{c}})
\leq \frac{C}{R}+\epsilon.
\end{eqnarray*}
Thus, we obtain the tightness of $\{\mu_{\rho u}^{\epsilon}\}_{\epsilon>0}$ on path space $\mathcal{X}_{\rho u}$.

The tightness of $\{\nu_{\rho}^{\epsilon}\}_{\epsilon>0}$ on path space $L_w^{\beta+1}(0,T; L^{\beta+1}(\mathcal{D}))$ and $\mathbb{P}\circ (\tilde{\rho}_{\epsilon}^\gamma+\delta\tilde{\rho}_{\epsilon}^{\beta})^{-1}$ on path space $L_w^{\frac{\beta+1}{\beta}}([0,T]\times\mathcal{D})$ is a result of the bound (\ref{4.3}) and the Banach-Alagolu theorem using the same argument as Lemma \ref{lem3.4} Claim 2.
\end{proof}

\textit{Proof of Proposition 4.1.} The proofs follow the same manner as Proposition \ref{pro3.1}.  {The convergence results} (\ref{4.14*}), \eqref{4.15*} follows from the bounds \eqref{4.2}-\eqref{4.4*} after using the Banach-Alagolu theorem, see also \cite[Proposition 6.3]{16}.

\subsection{Taking the limit for $\epsilon\rightarrow 0$}
Now, we can pass to the limit $\epsilon\rightarrow 0$ for fixed $\delta$, to obtain:

\begin{proposition} For all $\ell\in C^{\infty}(\mathcal{D}), \phi\in C^{\infty}(\mathcal{D}), \varphi\in C^{\infty}(\mathcal{D}), \psi\in C^{\infty}(\mathcal{D})$, $t\in [0,T]$, there exist pressure $\overline{\tilde{\rho}^\gamma}+\delta\overline{\tilde{\rho}^{\beta}}$ and an $L_2(\mathcal{H};H^{-l})$-valued martingale $\widetilde{W}$ such that the process $(\tilde{\rho},\tilde{u},\tilde{c}, \widetilde{Q}, \widetilde{W})$ satisfies equations, $\widetilde{\mathbb{P}}$~ \mbox{a.s.}
\begin{eqnarray*}
&&\int_{\mathcal{D}}\tilde{c}(t)\ell dx=\int_{\mathcal{D}}\tilde{c}(0)\ell dx-\int_{0}^{t}\int_{\mathcal{D}} (\tilde{u}\cdot\nabla)\tilde{c} \cdot\ell dxds-\int_{0}^{t}\int_{\mathcal{D}} \nabla \tilde{c} \cdot\nabla\ell dxds,\nonumber\\
&&\int_{\mathcal{D}}\tilde{\rho}(t)\psi dx=\int_{\mathcal{D}}\tilde{\rho}(0)\psi dx+\int_0^t\int_{\mathcal{D}}\tilde{\rho}\tilde{u}\cdot\nabla \psi dxds,\nonumber\\
&&\int_{\mathcal{D}}\tilde{\rho} \tilde{u}(t)\phi dx=\int_{\mathcal{D}}\tilde{m}(0)\phi dx+\int_0^t\int_{\mathcal{D}}(\tilde{\rho}\tilde{u}\otimes \tilde{u})\cdot\nabla \phi dxds-\int_0^t\int_{\mathcal{D}}\mu_1\nabla \tilde{u}\nabla \phi dxds\nonumber\\
&&\qquad\qquad\qquad-\int_0^t\int_{\mathcal{D}}(\mu_1+\mu_2){\rm div} \tilde{u}{\rm div}\phi dxds+\int_0^t\int_{\mathcal{D}}\left( \overline{\tilde{\rho}^{\gamma}}+\delta \overline{\tilde{\rho}^{\beta}}\right){\rm div}\phi dxds\nonumber\\
&&\qquad\qquad\qquad -\int_0^t\int_{\mathcal{D}}\bigg(({\rm F}(\widetilde{Q}){\rm I}_3-\nabla \widetilde{Q}\odot \nabla \widetilde{Q})+(\widetilde{Q}\triangle \widetilde{Q}-\triangle \widetilde{Q} \widetilde{Q})\nonumber\\&&\qquad\qquad\qquad\quad+\sigma^*(\tilde{c}^2\widetilde{Q})\bigg)\nabla\phi dxds
+\int_{\mathcal{D}}\phi\widetilde{W}dx,\\
&&\int_{\mathcal{D}}\widetilde{Q }(t)\varphi dx=\int_{\mathcal{D}}\widetilde{Q}(0)\varphi dx-\int_{0}^{t}\int_{\mathcal{D}}((\tilde{u}\cdot \nabla)\widetilde{Q}+\widetilde{Q}\widetilde{\Psi}-\widetilde{\Psi} \widetilde{Q})\varphi dxds\nonumber\\
&&\qquad\qquad\qquad+\int_{0}^{t}\int_{\mathcal{D}}\Gamma \varphi{\rm H}(\widetilde{Q},\tilde{c})dxds.
\end{eqnarray*}
\end{proposition}
\begin{proof} The argument is similar to that one in subsection 3.3. Note that due to the lack of strong convergence of density, we can not identify the specific form of stochastic term, but we could verify that $\widetilde{W}$ is still a martingale process, see \cite{Hofmanova}.
\end{proof}
In order to identify the nonlinear term of $\rho$ (the pressure term and the stochastic term), the strong convergence of density is necessary, which can be acquired by two steps following the idea of \cite{Feireisl,Lions}.

\noindent{\bf Step 1}. Weak convergence of the effective viscous flux.

The quantity $\rho^{\gamma}+\delta\rho^{\beta}-(\mu_2+2\mu_1){\rm div}u$ usually called the effective viscous flux which enjoys many remarkable properties, see \cite{hoff,Lions}. Introduce the operator $\mathcal{A}$
\begin{eqnarray*}
\mathcal{A}[f]=\nabla \triangle^{-1}f, ~~\mathcal{A}_j[f]=\partial_j \triangle^{-1}f,
\end{eqnarray*}
with the following properties:
\begin{eqnarray}
&&{\rm div} \mathcal{A}[f]=f, ~~\triangle \mathcal{A}[f]=\nabla f,\label{4.24}\\
&&\|\mathcal{A}[f]\|_{H^{1,p}(\mathcal{D})}\leq C\|f\|_{L^p(\mathcal{D})},~ {\rm for ~all }~p\geq 1,\label{4.25}\\
&&\|\mathcal{A}[f]\|_{L^\infty(\mathcal{D})}\leq C\|f\|_{L^p(\mathcal{D})},~ {\rm for ~all }~p>3.\label{4.26}
\end{eqnarray}

 {Note that} $\tilde{\rho}_\epsilon, \tilde{u}_\epsilon$ could be extended to zero outside $\mathcal{D}$ satisfying
$$\partial_t\tilde{\rho}_\epsilon+{\rm div}(\tilde{\rho}_\epsilon \tilde{u}_\epsilon)=\epsilon{\rm div}(1_\mathcal{D}\nabla\tilde{\rho}_\epsilon),$$
in the weak sense, where $1_{\mathcal{D}}$ stands for the indicator function. We could also do the zero extension to limit function $\tilde{\rho}, \tilde{u}$ to $\mathbb{R}^3$ which satisfies the equation $\eqref{Equ4.1}_2$ in the weak sense, for further detail see \cite{Feireisl,11,DWang}.

Using the It\^{o} formula to functions $f_1(\tilde{\rho}_{\epsilon},\tilde{\rho}_{\epsilon} \tilde{u}_{\epsilon})=\int_{\mathcal{D}}\tilde{\rho}_{\epsilon} \tilde{u}_{\epsilon}\cdot \tilde{\psi}\tilde{\phi}\mathcal{A}[\tilde{\rho}_{\epsilon} ]dx$ and $f_2(\tilde{\rho},\tilde{\rho} \tilde{u})=\int_{\mathcal{D}}\tilde{\rho} \tilde{u}\cdot \tilde{\psi}\tilde{\phi}\mathcal{A}[\tilde{\rho}]dx$ where the functions $\tilde{\psi}\in C_c^\infty(0, T), \tilde{\phi}\in C_c^\infty(\mathcal{D})$, taking expectation, to obtain
\begin{eqnarray}\label{4.27}
&& {\widetilde{\mathbb{E}}\int_{0}^{t}\tilde{\psi}(s)\int_{\mathcal{D}}\tilde{\phi}(\tilde{\rho}_{\epsilon}^{\gamma+1}
+\delta \tilde{\rho}_{\epsilon}^{\beta+1}-(\mu_2+2\mu_1){\rm div}\tilde{u}_{\epsilon})\cdot \tilde{\rho}_{\epsilon}dxds}\nonumber\\
&&=-\widetilde{\mathbb{E}}\int_{0}^{t}\tilde{\psi}(s)\int_{\mathcal{D}}\tilde{\phi}(\tilde{\rho}_{\epsilon}^{\gamma}
+\delta \tilde{\rho}_{\epsilon}^{\beta})\partial_i\tilde{\phi}\mathcal{A}_i [\tilde{\rho}_{\epsilon}]dxds\nonumber\\
&&\quad+(\mu_1+\mu_2)\widetilde{\mathbb{E}}\int_{0}^{t}\tilde{\psi}(s)\int_{\mathcal{D}}{\rm div}\tilde{u}_{\epsilon}\partial_i\tilde{\phi}\mathcal{A}_i [\tilde{\rho}_{\epsilon}]dxds+\mu_1\widetilde{\mathbb{E}}\int_{0}^{t}\tilde{\psi}(s)\int_{\mathcal{D}}\partial_j\tilde{u}^i_\epsilon\partial_j\tilde{\phi}\mathcal{A}_i [\tilde{\rho}_{\epsilon}]dxds\nonumber\\
&&\quad-\mu_1\widetilde{\mathbb{E}}\int_{0}^{t}\tilde{\psi}(s)\int_{\mathcal{D}}\tilde{u}^i_\epsilon\partial_j\tilde{\phi}\partial_j\mathcal{A}_i [\tilde{\rho}_{\epsilon}]dxds+
\mu_1\widetilde{\mathbb{E}}\int_{0}^{t}\tilde{\psi}(s)\int_{\mathcal{D}}\tilde{u}^i_\epsilon\partial_i\tilde{\phi}\tilde{\rho}_{\epsilon}dxds\nonumber\\
&&\quad-\epsilon\widetilde{\mathbb{E}}\int_{0}^{t}\tilde{\psi}(s)\int_{\mathcal{D}}\tilde{\phi}\tilde{\rho}_\epsilon \tilde{u}_\epsilon^i \mathcal{A}_i [{\rm div}(1_\mathcal{D}\nabla\tilde{\rho}_{\epsilon})]dxds\nonumber\\
&&\quad+\widetilde{\mathbb{E}}\int_{0}^{t}\tilde{\psi}(s)\int_{\mathcal{D}}
\tilde{\phi}\tilde{u}^{i}_{\epsilon}(\tilde{\rho}_{\epsilon}\partial_i\mathcal{A}_j[\tilde{\rho}_{\epsilon}\tilde{u}_{\epsilon}^{j}]
-\tilde{\rho}_{\epsilon}\tilde{u}^{i}_{\epsilon}\partial_i\mathcal{A}_j[\tilde{\rho}_{\epsilon}])dxds\nonumber\\
&&\quad-\widetilde{\mathbb{E}}\int_{0}^{t}\tilde{\psi}_s(s)\int_{\mathcal{D}}\tilde{\phi}\tilde{\rho}_{\epsilon}\tilde{u}_{\epsilon}^{i}\mathcal{A}_i [\tilde{\rho}_{\epsilon}]dxds
-\widetilde{\mathbb{E}}\int_{0}^{t}\tilde{\psi}(s)\int_{\mathcal{D}}\tilde{\rho}_{\epsilon}\tilde{u}_{\epsilon}^{i}\tilde{u}_{\epsilon}^{j}\partial_j\tilde{\phi}\mathcal{A}_i [\tilde{\rho}_{\epsilon}]dxds\nonumber\\
&&\quad+\epsilon\widetilde{\mathbb{E}}\int_{0}^{t}\tilde{\psi}(s)\int_{\mathcal{D}}\tilde{\phi} \partial_j\tilde{u}_{\epsilon}^{i}\partial_j\tilde{\rho}_{\epsilon}\mathcal{A}_i [\tilde{\rho}_{\epsilon}]dxds\nonumber\\
&&\quad+\widetilde{\mathbb{E}}\int_0^t\tilde{\psi}(s)\int_{\mathcal{D}}
\sigma^*(\tilde{c}_\epsilon^2\widetilde{Q}_\epsilon):(\tilde{\phi}\nabla\mathcal{A}[\tilde{\rho}_{\epsilon}]+\nabla \tilde{\phi}\otimes\mathcal{A}[\tilde{\rho}_{\epsilon}]) dxds\nonumber\\
&&\quad+\widetilde{\mathbb{E}}\int_0^t\tilde{\psi}(s)\int_{\mathcal{D}}({\rm F}(\widetilde{Q}_\epsilon){\rm I}_3-\nabla \widetilde{Q}_\epsilon\odot \nabla \widetilde{Q}_\epsilon):(\tilde{\phi}\nabla\mathcal{A}[\tilde{\rho}_{\epsilon}]+\nabla \tilde{\phi}\otimes\mathcal{A}[\tilde{\rho}_{\epsilon}]) dxds\nonumber\\
&&\quad+\widetilde{\mathbb{E}}\int_0^t\tilde{\psi}(s)\int_{\mathcal{D}}(\widetilde{Q}_\epsilon\triangle \widetilde{Q}_\epsilon-\triangle \widetilde{Q}_\epsilon \widetilde{Q}_\epsilon):(\tilde{\phi}\nabla\mathcal{A}[\tilde{\rho}_{\epsilon}]+\nabla \tilde{\phi}\otimes\mathcal{A}[\tilde{\rho}_{\epsilon}]) dxds\nonumber\\
&&=:J_1^\epsilon+\cdots+J_{13}^\epsilon,
\end{eqnarray}
and
\begin{eqnarray}\label{4.28}
&&\widetilde{\mathbb{E}}\int_{0}^{t}\tilde{\psi}(s)\int_{\mathcal{D}}\tilde{\phi}(\overline{\tilde{\rho}^\gamma}
+\delta\overline{\tilde{\rho}^{\beta}})\tilde{\rho}-\tilde{\phi}(\mu_2+2\mu_1){\rm div}\tilde{u}\cdot \tilde{\rho}dxds\nonumber\\
&&=-\widetilde{\mathbb{E}}\int_{0}^{t}\tilde{\psi}(s)\int_{\mathcal{D}}\tilde{\phi}(\overline{\tilde{\rho}^\gamma}
+\delta\overline{\tilde{\rho}^{\beta}})\partial_i\tilde{\phi}\mathcal{A}_i [\tilde{\rho}]dxds\nonumber\\
&&\quad+(\mu_1+\mu_2)\widetilde{\mathbb{E}}\int_{0}^{t}\tilde{\psi}(s)\int_{\mathcal{D}}{\rm div}\tilde{u}\partial_i\tilde{\phi}\mathcal{A}_i [\tilde{\rho}]dxds+\mu_1\widetilde{\mathbb{E}}\int_{0}^{t}\tilde{\psi}(s)\int_{\mathcal{D}}\partial_j\tilde{u}^i\partial_j\tilde{\phi}\mathcal{A}_i [\tilde{\rho}]dxds\nonumber\\
&&\quad-\mu_1\widetilde{\mathbb{E}}\int_{0}^{t}\tilde{\psi}(s)\int_{\mathcal{D}}\tilde{u}^i\partial_j\tilde{\phi}\partial_j\mathcal{A}_i [\tilde{\rho}]dxds+
\mu_1\widetilde{\mathbb{E}}\int_{0}^{t}\tilde{\psi}(s)\int_{\mathcal{D}}\tilde{u}^i\partial_i\tilde{\phi}\tilde{\rho}dxds\nonumber\\
&&\quad+\widetilde{\mathbb{E}}\int_{0}^{t}\tilde{\psi}(s)\int_{\mathcal{D}}
\tilde{\phi}\tilde{u}^{i}(\tilde{\rho}\partial_i\mathcal{A}_j[\tilde{\rho}\tilde{u}^{j}]
-\tilde{\rho}_{\epsilon}\tilde{u}^{i}\partial_i\mathcal{A}_j[\tilde{\rho}])dxds\nonumber\\
&&\quad-\widetilde{\mathbb{E}}\int_{0}^{t}\tilde{\psi}_s(s)\int_{\mathcal{D}}\tilde{\phi}\tilde{\rho}\tilde{u}^{i}\mathcal{A}_i [\tilde{\rho}]dxds
-\widetilde{\mathbb{E}}\int_{0}^{t}\tilde{\psi}(s)\int_{\mathcal{D}}\tilde{\rho}\tilde{u}^{i}\tilde{u}^{j}\partial_j\tilde{\phi}\mathcal{A}_i [\tilde{\rho}]dxds\nonumber\\
&&\quad+\widetilde{\mathbb{E}}\int_0^t\tilde{\psi}(s)\int_{\mathcal{D}}
\sigma^*(\tilde{c}^2\widetilde{Q}):(\tilde{\phi}\nabla\mathcal{A}[\tilde{\rho}]+\nabla \tilde{\phi}\otimes\mathcal{A}[\tilde{\rho}]) dxds\nonumber\\
&&\quad+\widetilde{\mathbb{E}}\int_0^t\tilde{\psi}(s)\int_{\mathcal{D}}({\rm F}(\widetilde{Q}){\rm I}_3-\nabla \widetilde{Q}\odot \nabla \widetilde{Q}):(\tilde{\phi}\nabla\mathcal{A}[\tilde{\rho}]+\nabla \tilde{\phi}\otimes\mathcal{A}[\tilde{\rho}]) dxds\nonumber\\
&&\quad+\widetilde{\mathbb{E}}\int_0^t\tilde{\psi}(s)\int_{\mathcal{D}}(\widetilde{Q}\triangle \widetilde{Q}-\triangle \widetilde{Q} \widetilde{Q}):(\tilde{\phi}\nabla\mathcal{A}[\tilde{\rho}]+\nabla \tilde{\phi}\otimes\mathcal{A}[\tilde{\rho}]) dxds\nonumber\\
&&=:J_1+\cdots+J_{11},
\end{eqnarray}
 here the stochastic integral is also cancelled resulting from the property of martingale.

 Our main goal is to get for all $t\in [0,T]$
\begin{eqnarray}\label{4.29}
&&\lim_{\epsilon\rightarrow 0}\widetilde{\mathbb{E}}\int_{0}^{t}\tilde{\psi}(s)\int_{\mathcal{D}}\tilde{\phi}(\tilde{\rho}_{\epsilon}^{\gamma+1}
+\delta \tilde{\rho}_{\epsilon}^{\beta+1}-(\mu_2+2\mu_1){\rm div}\tilde{u}_{\epsilon}\cdot \tilde{\rho}_{\epsilon})dxds\nonumber\\
&&=\widetilde{\mathbb{E}}\int_{0}^{t}\tilde{\psi}(s)\int_{\mathcal{D}}\tilde{\phi}(\overline{\tilde{\rho}^\gamma}
+\delta\overline{\tilde{\rho}^{\beta}})\tilde{\rho}-\tilde{\phi}(\mu_2+2\mu_1){\rm div}\tilde{u}\cdot \tilde{\rho}dxds,
\end{eqnarray}
it suffices to show that all right hand side terms of (\ref{4.27}) converges to the right hand side terms of (\ref{4.28}). Denote $J^\epsilon_{11}=J^\epsilon_{11,a}+J^\epsilon_{11,b}$, $J_{9}=J_{9,a}+J_{9,b}$, decompose
\begin{eqnarray*}
&&J_{11,a}^\epsilon-J_{9,a}=\widetilde{\mathbb{E}}\int_0^t\tilde{\psi}(s)\int_{\mathcal{D}}
\sigma^*(\tilde{c}_\epsilon^2\widetilde{Q}_\epsilon-\tilde{c}^2\widetilde{Q}):\tilde{\phi}\nabla\mathcal{A}[\tilde{\rho}_{\epsilon}]dxds\nonumber\\
&&\qquad\quad\quad\quad\quad+\widetilde{\mathbb{E}}\int_0^t\tilde{\psi}(s)\int_{\mathcal{D}}
\sigma^*(\tilde{c}^2\widetilde{Q}):\tilde{\phi}\nabla\mathcal{A}[\tilde{\rho}_{\epsilon}-\tilde{\rho}]dxds=\mathcal{J}_{1,a}+\mathcal{J}_{2,a},\nonumber\\
&&J_{11,b}^\epsilon-J_{9,b}=\widetilde{\mathbb{E}}\int_0^t\tilde{\psi}(s)\int_{\mathcal{D}}
\sigma^*(\tilde{c}_\epsilon^2\widetilde{Q}_\epsilon-\tilde{c}^2\widetilde{Q}):\nabla\tilde{\phi}\otimes\mathcal{A}[\tilde{\rho}_{\epsilon}]dxds\nonumber\\
&&\qquad\quad\quad\quad\quad+\widetilde{\mathbb{E}}\int_0^t\tilde{\psi}(s)\int_{\mathcal{D}}
\sigma^*(\tilde{c}^2\widetilde{Q}):\nabla\tilde{\phi}\otimes\mathcal{A}[\tilde{\rho}_{\epsilon}-\tilde{\rho}]dxds=\mathcal{J}_{1,b}+\mathcal{J}_{2,b}.
\end{eqnarray*}
By Proposition \ref{pro4.1}(\ref{4.12}) and the bounds (\ref{4.5}), (\ref{4.6}), we have
\begin{eqnarray}\label{4.30}
\mathcal{J}_{2,a}\rightarrow 0, ~{\rm as}~ \epsilon\rightarrow 0.
\end{eqnarray}
$\mathcal{J}_{1,a}$ can be handled as follows, by Proposition \ref{pro4.1}(\ref{4.11}), the bounds (\ref{4.3}), (\ref{4.5}), (\ref{4.6}) and the assumption $\beta>6$
\begin{align}\label{4.31}
&|\mathcal{J}_{1,a}|\leq \left|\widetilde{\mathbb{E}}\int_0^t\tilde{\psi}(s)\int_{\mathcal{D}}
\sigma^*(\tilde{c}_\epsilon^2-\tilde{c}^2)\widetilde{Q}_\epsilon:\tilde{\phi}\nabla\mathcal{A}[\tilde{\rho}_{\epsilon}]dxds\right|\nonumber\\
&\qquad\quad+\left|\widetilde{\mathbb{E}}\int_0^t\tilde{\psi}(s)\int_{\mathcal{D}}
\sigma^*(\tilde{c}^2(\widetilde{Q}_\epsilon-\widetilde{Q}):\tilde{\phi}\nabla\mathcal{A}[\tilde{\rho}_{\epsilon}]dxds\right|\nonumber\\
&~\qquad\leq \sigma^*C
\widetilde{\mathbb{E}}\int_0^t
\|\tilde{c}_\epsilon-\tilde{c}\|_{L^2}\|\tilde{c}_\epsilon,\tilde{c}\|_{L^6}\|\widetilde{Q}_\epsilon\|_{L^6}\|\tilde{\rho}_{\epsilon}\|_{L^\beta}ds\nonumber\\
&~\qquad\quad+\sigma^*C
\widetilde{\mathbb{E}}\int_0^t\|\tilde{c}\|_{L^2}\|\tilde{c}\|_{L^6}\|\widetilde{Q}_\epsilon-\widetilde{Q}\|_{L^6}\|\tilde{\rho}_{\epsilon}\|_{L^\beta}ds\nonumber\\
&~\qquad\leq \sigma^*C\widetilde{\mathbb{E}}\int_0^t\|\tilde{c}_\epsilon-\tilde{c}\|_{L^2}^2+\|\widetilde{Q}_\epsilon-\widetilde{Q}\|_{H^1}^2ds
\widetilde{\mathbb{E}}\left[\sup_{t\in[0,T]}(\|\tilde{\rho}_{\epsilon}\|_{L^\beta}^2+\|\tilde{c}, \widetilde{Q}_\epsilon, \nabla \widetilde{Q}_\epsilon\|^2_{L^2})\right]\nonumber\\
&~\qquad\quad\times\widetilde{\mathbb{E}}\int_0^t\|\tilde{c}_\epsilon, \tilde{c}\|_{L^6}^2ds \rightarrow 0, ~~{\rm as}~\epsilon\rightarrow 0.
\end{align}
In addition, we could get $\lim_{\epsilon\rightarrow 0}J_{11,b}^\epsilon=J_{9,b}$, then it follows $\lim_{\epsilon\rightarrow 0}J_{11}^\epsilon=J_9$ using (\ref{4.30}) and (\ref{4.31}). Similarly, decompose
\begin{align*}
&J_{12,a}^\epsilon-J_{10,a}\nonumber\\&=\widetilde{\mathbb{E}}\int_0^t\tilde{\psi}(s)\int_{\mathcal{D}}({\rm F}(\widetilde{Q}_\epsilon){\rm I}_3-{\rm F}(\widetilde{Q}){\rm I}_3-(\nabla \widetilde{Q}_\epsilon\odot \nabla \widetilde{Q}_\epsilon-\nabla \widetilde{Q}\odot \nabla \widetilde{Q})):\tilde{\phi}\nabla\mathcal{A}[\tilde{\rho}_{\epsilon}]dxds\nonumber\\
&\quad+\widetilde{\mathbb{E}}\int_0^t\tilde{\psi}(s)\int_{\mathcal{D}}({\rm F}(\widetilde{Q}){\rm I}_3-\nabla \widetilde{Q}\odot \nabla \widetilde{Q}):\tilde{\phi}\nabla\mathcal{A}[\tilde{\rho}_\epsilon-\tilde{\rho}]dxds=\mathcal{I}_{1,a}+\mathcal{I}_{2,a}.
\end{align*}
Due to the same arguments as (\ref{4.30}) and (\ref{4.31}), as $\epsilon\rightarrow 0$
\begin{align}
\mathcal{I}_{2,a}&\rightarrow 0, \label{4.32}\\
|\mathcal{I}_{1,a}|&\leq \left|\widetilde{\mathbb{E}}\int_0^t\tilde{\psi}(s)\int_{\mathcal{D}}(\nabla (\widetilde{Q}_\epsilon-\widetilde{Q})\odot \nabla \widetilde{Q}_\epsilon-\nabla \widetilde{Q}\odot \nabla (\widetilde{Q}_\epsilon-\widetilde{Q})):\tilde{\phi}\nabla\mathcal{A}[\tilde{\rho}_{\epsilon}]dxds\right|\nonumber\\ & \leq C \widetilde{\mathbb{E}}\int_0^t\|\nabla\widetilde{Q}_\epsilon, \nabla\widetilde{Q}\|_{L^6}\|\widetilde{Q}_\epsilon-\widetilde{Q}\|_{L^2}\|\nabla\mathcal{A}[\tilde{\rho}_{\epsilon}]\|_{L^\beta}ds\nonumber\\
&\leq C\widetilde{\mathbb{E}}\left[\sup_{t\in [0,T]}\|\tilde{\rho}_{\epsilon}\|_{L^\beta}\right]\widetilde{\mathbb{E}}\int_0^t\|\widetilde{Q}_\epsilon-\widetilde{Q}\|_{L^2}^2ds
\widetilde{\mathbb{E}}\int_0^t\|\nabla\widetilde{Q}_\epsilon, \nabla\widetilde{Q}\|_{L^6}^2ds\rightarrow 0, \label{4.33}\\
&\widetilde{\mathbb{E}}\int_0^t\tilde{\psi}(s)\int_{\mathcal{D}}({\rm F}(\widetilde{Q}_\epsilon)-{\rm F}(\widetilde{Q})){\rm I}_3:\tilde{\phi}\nabla\mathcal{A}[\tilde{\rho}_{\epsilon}]dxds\rightarrow 0.\label{4.34}
\end{align}
Combining the convergence (\ref{4.32})-(\ref{4.34}), we get $\lim_{\epsilon\rightarrow 0}J_{12,a}^\epsilon= J_{10,a}$. Using similar estimate, we could get  $\lim_{\epsilon\rightarrow 0}J_{12,b}^\epsilon= J_{10,b}$, it follows $\lim_{\epsilon\rightarrow 0}J_{12}^\epsilon= J_{10}$.

Denote $J^\epsilon_{13}=J^\epsilon_{13,a}+J^\epsilon_{13,b}$, $J_{11}=J_{11,a}+J_{11,b}$, decompose
\begin{align}\label{4.27*}
J^\epsilon_{13,b}-J_{11,b}&=\widetilde{\mathbb{E}}\int_0^t\tilde{\psi}(s)\int_{\mathcal{D}}(\widetilde{Q}_\epsilon\triangle \widetilde{Q}_\epsilon-\triangle \widetilde{Q}_\epsilon \widetilde{Q}_\epsilon):\nabla \tilde{\phi}\otimes(\mathcal{A}[\tilde{\rho}_{\epsilon}]-\mathcal{A}[\tilde{\rho}]) dxds\nonumber\\
&\quad+\widetilde{\mathbb{E}}\int_0^t\tilde{\psi}(s)\int_{\mathcal{D}}((\widetilde{Q}_\epsilon-\widetilde{Q})\triangle \widetilde{Q}_\epsilon-\triangle \widetilde{Q}_\epsilon (\widetilde{Q}_\epsilon-\widetilde{Q})):\nabla \tilde{\phi}\otimes\mathcal{A}[\tilde{\rho}]dxds\nonumber\\
&\quad+\widetilde{\mathbb{E}}\int_0^t\tilde{\psi}(s)\int_{\mathcal{D}}(\widetilde{Q}_\epsilon\triangle (\widetilde{Q}_\epsilon-\widetilde{Q})-\triangle (\widetilde{Q}_\epsilon -\widetilde{Q})\widetilde{Q}_\epsilon):\nabla \tilde{\phi}\otimes\mathcal{A}[\tilde{\rho}]dxds.
\end{align}
Using Proposition \ref{pro4.1}(\ref{4.11}), the bound (\ref{4.6}), the estimate \eqref{4.26} and the H\"{o}lder inequality, we deduce that the right hand terms of \eqref{4.27*} go to $0$, (similar to \eqref{4.33}).

 Since the matrices $\widetilde{Q}_\epsilon$, $\widetilde{Q}$ are symmetric, hence $\widetilde{Q}_\epsilon\triangle \widetilde{Q}_\epsilon-\triangle \widetilde{Q}_\epsilon \widetilde{Q}_\epsilon$, $\widetilde{Q}\triangle \widetilde{Q}-\triangle \widetilde{Q} \widetilde{Q}$ are skew-symmetric, and note that $\nabla\mathcal{A}[\tilde{\rho}_{\epsilon}], \nabla\mathcal{A}[\tilde{\rho}]$ are symmetric, to conclude that $J_{13,a}^\epsilon, J_{11,a}=0$. (The special structure of $Q$-tensor makes the weak convergence possible, otherwise we are not able to handle the high-order nonlinear term).

We also have $J_{10}^\epsilon\rightarrow 0$ as $\epsilon\rightarrow 0$. For $J_{6}^\epsilon$, as $\epsilon\rightarrow 0$
\begin{eqnarray*}
&&~|J^\epsilon_3|\leq \sqrt{\epsilon}C\widetilde{\mathbb{E}}\int_0^t \|\sqrt{\epsilon} \nabla \tilde{\rho}_\epsilon\|_{L^2}\|\tilde{u}_\epsilon\|_{L^6}\|\tilde{\rho}_\epsilon\|_{L^\beta}ds\\
&&\qquad\leq \sqrt{\epsilon}C\widetilde{\mathbb{E}}\left[\sup_{t\in[0,T]}\|\tilde{\rho}_\epsilon\|_{L^\beta}\right] \widetilde{\mathbb{E}}\int_0^t\|\sqrt{\epsilon} \nabla \tilde{\rho}_\epsilon\|_{L^2}^2ds \widetilde{\mathbb{E}}\int_0^t\|\nabla \tilde{u}_\epsilon\|_{L^2}^2ds\\
&&\qquad\leq C\sqrt{\epsilon}\rightarrow 0.
\end{eqnarray*}
The proofs of convergence of rest terms are standard, we refer the readers to \cite{Hofmanova, DWang,16}. Finally, we obtain the convergence result (\ref{4.29}).

\noindent{\bf Step 2}. Strong convergence of density.

In this step, we could show the strong convergence of density using the re-normalized mass equation and the Minty idea, for further details see \cite{Feireisl}.

Now we can pass the limit to identify the stochastic term and nonlinear pressure term using the same argument as ({\ref{3.39}}), obtaining the following result,

\begin{proposition} For $\beta>{\rm max}\{6,\gamma\}$, $\delta>0$ and  if conditions (\ref{2.1}), (\ref{2.2}) hold. There exists a global weak martingale solution to the modified system (\ref{Equ4.1}).
\end{proposition}

\maketitle
\section{Vanishing artificial pressure}

In this section, we shall pass the artificial pressure coefficient $\delta\rightarrow 0$ to establish the Theorem \ref{thm2.1}. Also, the following uniform bounds hold for process $(\rho_{\delta}, u_{\delta}, c_{\delta}, Q_{\delta})$
\begin{eqnarray}
&&\rho_{\delta}\in L^{p}(\Omega; L^{\infty}(0,T; L^{\gamma}(\mathcal{D}))), \label{5.1}\\
&&\sqrt[\beta]{\delta}\rho_\delta\in L^{p}(\Omega; L^{\infty}(0,T; L^{\beta}(\mathcal{D}))),\label{5.2}\\
&&u_{\delta}\in L^{p}(\Omega; L^{2}(0,T; H^{1}(\mathcal{D}))),\label{5.3}\\
&&\sqrt{\rho_{\delta} }u_{\delta}\in L^{p}(\Omega; L^\infty(0,T; L^{2}(\mathcal{D}))),\label{5.4}\\
&&c_{\delta}\in L^{p}(\Omega; L^{\infty}(0,T; L^{2}(\mathcal{D}))\cap L^{2}(0,T; H^{1}(\mathcal{D}))),\label{5.5}\\
&&Q_{\delta}\in L^{p}(\Omega; L^{\infty}(0,T; H^{1}(\mathcal{D}))\cap L^{2}(0,T; H^{2}(\mathcal{D}))).\label{5.6}
\end{eqnarray}
\begin{lemma}
The approximate sequence $\rho_{\delta}$ satisfies the following estimate
\begin{eqnarray*}
\mathbb{E}\int_{0}^{T}\int_{\mathcal{D}}\rho_{\delta}^{\gamma+\theta}+\delta\rho_{\delta}^{\beta+\theta}dxdt\leq C,
\end{eqnarray*}
where the constant $C$ is independent of $\delta$ and the constant $\theta\in\left(0, {\rm min}\left\{\frac{2\gamma-3}{3}, \frac{\gamma}{3}\right\}\right)$.
\end{lemma}
\begin{proof} The proof follows the same line as Lemma \ref{lem4.1}. By the Di Perna-Lions commutator lemmas, we infer that the following equation holds in the weak sense
\begin{eqnarray}
d\left[\rho^\theta_{\delta}-(\rho^\theta_{\delta})_m\right]+{\rm div}(\rho^\theta_{\delta}u_{\delta})dt+\left[(\theta-1)\rho^\theta_{\delta}{\rm div}u_{\delta}-(\rho^\theta_{\delta}{\rm div}u_{\delta})_m\right]dt=0,~ \mathbb{P}~ \mbox{a.s.} \label{5.7}
\end{eqnarray}
Then, applying the operator $\mathcal{T}$ on both sides of (\ref{5.7}), to obtain
\begin{eqnarray}
d\mathcal{T}[\rho^\theta_{\delta}-(\rho^\theta_{\delta})_m]+\mathcal{T}[{\rm div}(\rho^\theta_{\delta}u_{\delta})]dt+(\theta-1)\mathcal{T}[\rho^\theta_{\delta}{\rm div}u_{\delta}-(\rho^\theta_{\delta}{\rm div}u_{\delta})_m]dt=0.\label{5.8}
\end{eqnarray}
Applying the It\^{o} product formula to function $\Phi(\rho_{\delta}u_{\delta},\rho_{\delta}^{\theta})
=\int_{\mathcal{D}}\rho_{\delta}u_{\delta}\cdot\mathcal{T}[\rho_{\delta}^{\theta}-(\rho^\theta_{\delta})_m]dx$, and taking expectation, we have
\begin{eqnarray*}
&&\mathbb{E}\int_{\mathcal{D}}\rho_{\delta}u_{\delta}\cdot \mathcal{T}[ \rho_{\delta}^{\theta}-(\rho^\theta_{\delta})_m]dx\\
&&=\mathbb{E}\int_{0}^{t}\int_{\mathcal{D}}\rho_{\delta}^{\gamma+\theta}+\delta\rho_{\delta}^{\beta+\theta}dxds
-\mathbb{E}\int_{0}^{t}(\rho_{\delta}^{\theta})_{m}\int_{\mathcal{D}}\rho_{\delta}^{\gamma}
+\delta\rho_{\delta}^{\beta}dxds\nonumber \\
&&\quad+\mathbb{E}\int_{\mathcal{D}}m_{0,\delta}\cdot \mathcal{T}[\rho_{0,\delta}^{\theta}-(\rho^\theta_{0,\delta})_m]dx-(\mu_1+\mu_2)\mathbb{E}\int_{0}^{t}\int_{\mathcal{D}}{\rm div} u_{\delta}\cdot\rho_{\delta}^{\theta}dxds\nonumber\\
&&\quad+\mathbb{E}\int_{0}^{t}\int_{\mathcal{D}}\rho_{\delta}u_{\delta}\otimes u_{\delta}:\nabla\mathcal{T}[ \rho_{\delta}^{\theta}-(\rho^\theta_{\delta})_m]dxds-\mu_1 \mathbb{E}\int_{0}^{t}\int_{\mathcal{D}} \nabla u_{\delta}:\nabla  \mathcal{T}[\rho_{\delta}^{\theta}-(\rho^\theta_{\delta})_m]dxds\nonumber \\
&&\quad+\mathbb{E}\int_{0}^{t}\int_{\mathcal{D}} \rho_{\delta}^{\theta}u_{\delta}\mathcal{T}[{\rm div}(\rho_{\delta}u_{\delta})]dxds
-\mathbb{E}\int_{0}^{t}\int_{\mathcal{D}}\sigma^*c_\delta^2Q_\delta:\nabla\mathcal{T}[\rho_{\delta}^{\theta}-(\rho^\theta_{\delta})_m]dxds\nonumber\\
&&\quad+(1-\theta)\mathbb{E}\int_{0}^{t}\int_{\mathcal{D}}\rho_{\delta}u_{\delta}\mathcal{T}[\rho_{\delta}^{\theta}{\rm div}u_{\delta}-(\rho^\theta_{\delta}{\rm div}u_{\delta})_m]dxds\nonumber\\
&&\quad-\mathbb{E}\int_{0}^{t}\int_{\mathcal{D}}(\nabla Q_\delta\odot \nabla Q_\delta-{\rm F(Q_\delta)}{\rm I}_3):\nabla\mathcal{T}[\rho_{\delta}^{\theta}-(\rho^\theta_{\delta})_m]dxds\nonumber\\
&&\quad-\mathbb{E}\int_{0}^{t}\int_{\mathcal{D}}(Q_\delta \triangle Q_\delta-\triangle Q_\delta Q_\delta):\nabla\mathcal{T}[\rho_{\delta}^{\theta}-(\rho^\theta_{\delta})_m]dxds.
\end{eqnarray*}
Our goal is to get the bound of $\mathbb{E}\int_{0}^{t}\int_{\mathcal{D}}\rho_{\delta}^{\gamma+\theta}+\delta\rho_{\delta}^{\beta+\theta}dxds$, which can be achieved after all other terms get controlled. Here we just give a limit amount of details.

For $\theta\leq \frac{2\gamma-3}{3}$, the bounds (\ref{5.1}), (\ref{5.3}), the H\"{o}lder inequality and the Sobolev embedding $H^{1, \frac{6\gamma}{7\gamma-6}}(\mathcal{D})\hookrightarrow L^{\frac{6\gamma}{5\gamma-6}}(\mathcal{D})$ imply
\begin{eqnarray*}
&&\left|\mathbb{E}\int_{0}^{t}\int_{\mathcal{D}}\rho_{\delta}u_{\delta}\mathcal{T}[\rho_{\delta}^{\theta}{\rm div}u_{\delta}]dxds\right|\nonumber\\
&&\leq C\mathbb{E}\int_{0}^{t}\|\rho_{\delta}\|_{L^\gamma}\|u_\delta\|_{L^6}\left\|\mathcal{T}[\rho_{\delta}^{\theta}{\rm div}u_{\delta}]\right\|_{L^{\frac{6\gamma}{5\gamma-6}}}ds\\
&&\leq C\mathbb{E}\left[\sup_{t\in [0,T]}\|\rho_{\delta}\|_{L^\gamma}\int_0^t\|u_\delta\|_{L^6}\|\rho_{\delta}^{\theta}{\rm div}u_{\delta}\|_{L^\frac{6\gamma}{7\gamma-6}}ds\right]\\
&&\leq  C\mathbb{E}\left[\sup_{t\in [0,T]}\|\rho_{\delta}\|_{L^\gamma}\int_0^t\|u_\delta\|_{L^6}\|{\rm div}u_{\delta}\|_{L^2}\|\rho_{\delta}^{\theta}\|_{L^\frac{3\gamma}{2\gamma-3}}ds\right]\\
&&\leq  C\mathbb{E}\left[\sup_{t\in [0,T]}\|\rho_{\delta}\|_{L^\gamma}\|\rho_{\delta}\|_{L^\gamma}^\theta\int_0^t\|u_\delta\|_{L^6}\|{\rm div}u_{\delta}\|_{L^2}ds\right]\leq C,
\end{eqnarray*}
where $C$ is independent of $\delta$. For $\theta<\frac{\gamma}{3}$, using the H\"{o}lder inequality and the bounds (\ref{5.1}), (\ref{5.6}), to get
\begin{eqnarray*}
&&\left|\mathbb{E}\int_{0}^{t}\int_{\mathcal{D}}\sigma^*c_\delta^2Q_\delta:\nabla\mathcal{T}[\rho_{\delta}^{\theta}]dxds\right|\\
&&\leq C\mathbb{E}\int_{0}^{t}\|c_\delta\|_{L^6}^2\|Q_\delta\|_{L^6}\|\rho_{\delta}\|^{\theta}_{L^{2\theta}}ds\\
&&\leq C\mathbb{E}\left(\int_{0}^{t}\|c_\delta\|_{L^6}^2ds\right)^2\mathbb{E}\left[\sup_{t\in [0,T]}\left(\|Q_\delta\|_{L^6}^4+\|\rho_{\delta}\|^{4\theta}_{L^{\gamma}}\right)\right]\leq C,\\
&&\left|\mathbb{E}\int_{0}^{t}\int_{\mathcal{D}}(\nabla Q_\delta\odot \nabla Q_\delta-{\rm F(Q_\delta)}{\rm I}_3):\nabla\mathcal{T}[\rho_{\delta}^{\theta}]dxds\right|\\
&&\leq C\mathbb{E}\int_{0}^{t}(\|\nabla Q_\delta\|_{L^2}\|\nabla Q_\delta\|_{L^6}+\|Q_\delta\|_{L^6})\|\rho_\delta^\theta\|_{L^3}ds\\
&&\leq C\mathbb{E}\left[\sup_{t\in [0,T]}(\| Q_\delta\|_{H^1}^4+\|\rho_{\delta}\|^{4\theta}_{L^{\gamma}})\right]\mathbb{E}\int_{0}^{t}\|\nabla Q_\delta\|_{L^6}^2ds\leq C,\\
&&\left|\mathbb{E}\int_{0}^{t}\int_{\mathcal{D}}(Q_\delta \triangle Q_\delta-\triangle Q_\delta Q_\delta):\nabla\mathcal{T}[\rho_{\delta}^{\theta}]dxds\right|\\
&&\leq C\mathbb{E}\int_{0}^{t}\|Q_\delta\|_{L^6}\|\triangle Q_\delta\|_{L^2}\|\rho_\delta^\theta\|_{L^3}ds\\
&&\leq C\mathbb{E}\left[\sup_{t\in [0,T]}\left(\|Q_\delta\|_{L^6}^4+\|\rho_{\delta}\|^{4\theta}_{L^{\gamma}}\right)\right]\mathbb{E}\int_{0}^{t}\|\triangle Q_\delta\|_{L^2}^2ds\leq C,
\end{eqnarray*}
where $C$ is independent of $\delta$. This completes the proof.
\end{proof}

\subsection{Compactness argument}
Define the cut off functions
\begin{eqnarray*}
T_{k}(z)=kT(\frac{z}{k}),~ k=1,2,3\cdots
\end{eqnarray*}
where $T(z)$ is a smooth concave function on $\mathbb{R}$ such that $T(z)=z$ if $z\leq 1$ and $T(z)=2$ if $z\geq 3$. The definition of $T_k(z)$ implies that \begin{eqnarray}\label{5.9}
T_{k}(z)=\left\{\begin{array}{ll}
z, ~z\leq k,\\
2k, z\geq 3k.
\end{array}\right.
\end{eqnarray}

Here, we define the path space $\mathcal{X}_1=\mathcal{X}_{u}\times\mathcal{X}_{\rho}\times\mathcal{X}_{\rho u}\times \mathcal{X}_c\times \mathcal{X}_{Q}\times\mathcal{X}_{\mathcal{W}}$,
where
\begin{eqnarray*}
&&\mathcal{X}_{\rho}:=L^\infty(0,T; H^{-\frac{1}{2}}(\mathcal{D}))\cap L_w^{\gamma+\theta}(0,T; L^{\gamma+\theta}(\mathcal{D})),
\end{eqnarray*}
and $\mathcal{X}_{u}, \mathcal{X}_{\rho u},  \mathcal{X}_{c}, \mathcal{X}_{Q}, \mathcal{X}_{\mathcal{W}}$ are same as the definition in subsection 3.2. Let $\mathcal{X}=\mathcal{X}_1\times C_{w}([0,T]; L^p(\mathcal{D}))\times L_w^{2}(0,T;L^2(\mathcal{D}))\times L_w^{\frac{\gamma+\theta}{\gamma}}((0,T)\times\mathcal{D})$ for all $1\leq p<\infty$. Similarly, we can define the set of probability measures $\{\nu^\delta\}_{\delta>0}$ as before. Following the same line as previous section to build the compactness result,

\begin{proposition}\label{pro5.1} There exists a new probability space $(\widetilde{\Omega}, \widetilde{\mathcal{F}}, \widetilde{\mathbb{P}})$, a subsequence of $\{\nu^{\delta}\}_{\delta>0 }$ {\rm(still denoted by $\nu^\delta$)} and $\mathcal{X}$-valued measurable random variables
\begin{eqnarray*}
(\tilde{u}_{\delta},\tilde{\rho}_{\delta},\tilde{\rho}_{\delta}\tilde{u}_{\delta}, \tilde{\rho}_{\delta}^\gamma, \tilde{c}_{\delta}, \widetilde{Q}_\delta, T_k(\tilde{\rho}_{\delta}), ~~(\tilde{\rho}_{\delta}T^{'}_k(\tilde{\rho}_{\delta})-T_k(\tilde{\rho}_{\delta})){\rm div}\tilde{u}_{\delta}, \widetilde{\mathcal{W}}_{\delta}),
\end{eqnarray*}
and $(\tilde{u},\tilde{\rho}, \tilde{\rho}\tilde{u}, \tilde{c}, \widetilde{Q}, \widetilde{T}_{1, k}, \widetilde{T}_{2, k}, \widetilde{\mathcal{W}})$ such that
\begin{eqnarray*}
&&\widetilde{\mathbb{P}}\left\{(\tilde{u}_{\delta},\tilde{\rho}_{\delta},\tilde{\rho}_{\delta}\tilde{u}_{\delta}, \tilde{\rho}_{\delta}^\gamma, \tilde{c}_{\delta}, \widetilde{Q}_\delta, T_k(\tilde{\rho}_{\delta}), (\tilde{\rho}_{\delta}T^{'}_k(\tilde{\rho}_{\delta})-T_k(\tilde{\rho}_{\delta})){\rm div}\tilde{u}_{\delta}, \widetilde{\mathcal{W}}_{\delta})\in \cdot\right\}=\nu^{\delta}(\cdot),\\
&&\widetilde{\mathbb{P}}\left\{(\tilde{u},\tilde{\rho}, \tilde{\rho}\tilde{u}, \rho^*, \tilde{c}, \widetilde{Q}, \widetilde{T}_{1, k}, \widetilde{T}_{2, k}, \widetilde{\mathcal{W}})\in \cdot\right\}=\nu(\cdot),
\end{eqnarray*}
where $\nu$ is a Radon measure and $\widetilde{\mathcal{\mathcal{W}}}_{\delta}$ is cylindrical Wiener process, relative to the filtration $\widetilde{\mathcal{F}}_{t}^{\delta}$ generated by the completion of $\sigma(\tilde{u}_{\delta}(s),\tilde{\rho}_{\delta}(s), \tilde{c}_\delta(s), \widetilde{Q}_\delta(s), \widetilde{\mathcal{W}}_{\delta}(s);s\leq t)$, and the following convergence results hold, $\widetilde{\mathbb{P}}$ \mbox{a.s.}
\begin{eqnarray}\label{5.10}(\tilde{u}_{\delta},\tilde{\rho}_{\delta},\tilde{\rho}_{\delta}\tilde{u}_{\delta}, \tilde{c}_{\delta}, \widetilde{Q}_{\delta}, \widetilde{\mathcal{W}}_{\delta})
\rightarrow(\tilde{u},\tilde{\rho}, \tilde{\rho}\tilde{u}, \tilde{c}, \widetilde{Q}, \widetilde{\mathcal{W}}),
\end{eqnarray}
in the topology of $\mathcal{X}_1$, in addition
\begin{eqnarray}
&&\tilde{\rho}_{\delta}^\gamma\rightarrow \rho^*~{\rm in}~ L_w^{\frac{\gamma+\theta}{\gamma}}((0,T)\times\mathcal{D}),\label{5.11}\\
&&T_k(\tilde{\rho}_{\delta})\rightarrow \widetilde{T}_{1, k} ~{\rm in}~C_{w}([0,T]; L^p(\mathcal{D})), ~{\rm for~ all}~1\leq p<\infty,\label{5.12}\\
&&(\tilde{\rho}_{\delta}T^{'}_k(\tilde{\rho}_{\delta})-T_k(\tilde{\rho}_{\delta})){\rm div}\tilde{u}_{\delta}\rightharpoonup\widetilde{T}_{2, k} ~{\rm in}~L^{2}(0,T;L^2(\mathcal{D})).\label{5.13}
\end{eqnarray}
Moreover, the bounds \eqref{5.1}-\eqref{5.6} still hold for $(\tilde{u}_{\delta},\tilde{\rho}_{\delta},\tilde{\rho}_{\delta}\tilde{u}_{\delta}, \tilde{c}_{\delta}, \widetilde{Q}_{\delta})$ uniformly in $\delta$.
\end{proposition}

\begin{lemma}\label{lem5.2} The set of induced laws $\{\nu^{\delta}\}_{\delta>0}$ is tight on path space $\mathcal{X}$.
\end{lemma}
\begin{proof} Observe that the argument used in Lemma \ref{lem4.2} can be adopted. Decompose $\rho_{\delta}u_\delta=X_\delta+Y_\delta$, where
\begin{eqnarray*}
&&X_\delta=m_{0,\delta}+\int_{0}^{t}-{\rm div}(\rho_\delta u_\delta\otimes u_\delta)-\nabla\rho^\gamma_\delta
 +\mu_1\Delta u_\delta+(\mu_1+\mu_2)\nabla ({\rm div} u_\delta)\\&&\qquad+\nabla\cdot({\rm F}(Q_\delta){\rm I}_3-\nabla Q_\delta\odot \nabla Q_\delta)+\nabla\cdot(Q_\delta\triangle Q_\delta-\triangle Q_\delta Q_\delta)+\sigma^*\nabla\cdot(c_\delta^2Q_\delta)ds\\ &&\qquad+\int_{0}^{t}\rho_\delta f(\rho_\delta,\rho_\delta u_\delta, c_\delta, Q_\delta)d\mathcal{W},
\end{eqnarray*}
and
\begin{eqnarray*}
Y_\delta=\delta\int_{0}^{t}\nabla \rho_\delta^\beta ds.
\end{eqnarray*}
Also, the process $X_\delta$ has the bound
\begin{eqnarray*}
\mathbb{E}\|X_{\delta}(t)\|_{C^{\alpha}([0,T]; H^{-l}(\mathcal{D}))}\leq C~ {\rm for~ any} ~\alpha\in [0,\frac{1}{2}), ~l\geq\frac{5}{2}.
\end{eqnarray*}
For the process $Y_\delta$, using the bound (\ref{5.2}), we have as $\delta\rightarrow 0$
\begin{eqnarray*}
\delta\rho_{\delta}^\beta\rightarrow 0 ~{\rm in}~L^{\frac{\beta+\theta}{\beta}}((0,T)\times\mathcal{D}),~\mathbb{P}~ \mbox{a.s.}
\end{eqnarray*}
which implies
\begin{eqnarray*}
\|Y_{\delta}\|_{C([0,T];H^{-1,\frac{\beta+\theta}{\beta}})}=\sup_{t\in [0,T]}\left\|\int_{0}^{t}\delta \nabla\rho_{\delta}ds\right\|_{H^{-1,\frac{\beta+\theta}{\beta}}}\\
\leq \int_{0}^{T}\|\delta \nabla\rho_{\delta}\|_{H^{-1,\frac{\beta+\theta}{\beta}}}dt\rightarrow 0,~ {\rm as}~ \delta\rightarrow 0,~\mathbb{P}~ \mbox{a.s.}
\end{eqnarray*}
this convergence gives $Y_{\delta}\rightarrow 0$ in $C([0,T], H^{-1,\frac{\beta+\theta}{\beta}}(\mathcal{D}))$ in the sense of distribution.

On the other hand, using the boundness of $T_k$ and $(\tilde{\rho}_{\delta}T^{'}_k(\tilde{\rho}_{\delta})-T_k(\tilde{\rho}_{\delta})){\rm div}\tilde{u}_{\delta}$ on spaces $C([0,T]; L^p(\mathcal{D}))$ and $L^{2}(0,T;L^2(\mathcal{D}))$ respectively, we can show the sequence of probability measures $\mathbb{P}\circ(T_k(\tilde{\rho}_{\delta}))^{-1}$ and $\mathbb{P}\circ((\tilde{\rho}_{\delta}T^{'}_k(\tilde{\rho}_{\delta})-T_k(\tilde{\rho}_{\delta})){\rm div}\tilde{u}_{\delta})^{-1}$ are tight on path spaces $C_{w}([0,T]; L^p(\mathcal{D}))$ and $L_w^{2}(0,T;L^2(\mathcal{D}))$, see Claim 2, Lemma \ref{lem3.4}. This completes the proof.
\end{proof}
\textit{Proof of Proposition \ref{pro5.1}}. The proof follows the same line of the Proposition \ref{pro4.1}.

\subsection{Passing limit for $\delta\rightarrow 0$}\

Note that,
\begin{align*}
\widetilde{\mathbb{E}}\int_0^t\langle\delta\tilde{\rho}_{\delta}^\beta, \nabla \phi\rangle ds&\leq \delta^\frac{\theta}{\theta+\beta}\|\nabla \phi\|_{L^\infty}\widetilde{\mathbb{E}}\int_0^t\int_{\mathcal{D}}\delta^{\frac{\beta}{\theta+\beta}}\tilde{\rho}_{\delta}^\beta dx ds\\
&\leq C\delta^\frac{\theta}{\theta+\beta}\|\nabla \phi\|_{L^\infty}\widetilde{\mathbb{E}}\int_0^t\int_{\mathcal{D}}\delta\tilde{\rho}_{\delta}^{\beta+\theta} dxds\\
&\leq C\delta^\frac{\theta}{\theta+\beta}\|\nabla \phi\|_{L^\infty}\rightarrow 0, ~{\rm as}~\delta\rightarrow 0,
\end{align*}
this convergence result together with Proposition \ref{pro5.1}, using the same argument as subsection 3.3, to conclude that there exists an $L_2(\mathcal{H},H^{-l})$-valued process $\widetilde{W}$ and $L^{\frac{\gamma+\theta}{\gamma}}$-valued pressure $\rho^*$ such that $(\tilde{u}, \tilde{\rho} \tilde{u}, \tilde{c}, \widetilde{Q} , \widetilde{W})$ satisfies the momentum equation $\widetilde{\mathbb{P}}$ a.s.
\begin{eqnarray}\label{5.20}
&&\partial_{t}(\tilde{\rho} \tilde{u})+{\rm div}(\tilde{\rho }\tilde{u}\otimes \tilde{u})+\nabla \rho^*
 =\mu_1\Delta \tilde{u}+(\mu_1+\mu_2)\nabla ({\rm div} \tilde{u})+\nabla\cdot({\rm F}(\widetilde{Q}){\rm I}_3-\nabla \widetilde{Q}\odot \nabla \widetilde{Q})\nonumber\\&&\qquad \qquad\qquad \qquad\qquad \qquad\quad+\nabla\cdot(\widetilde{Q}\triangle \widetilde{Q}-\triangle \widetilde{Q} \widetilde{Q})+\sigma^*\nabla\cdot(\tilde{c}^2\widetilde{Q}) +\frac{d\widetilde{W}}{dt},
\end{eqnarray}
in the weak sense.

The proof of Theorem {\ref{thm2.1}} will be completed once we build the strong convergence of density, to identify the pressure term and the stochastic term in equation (\ref{5.20}). Following the idea of \cite{Feireisl,Lions}, the proof of strong convergence of density shall be obtained by three steps.

\noindent{\bf Step 1}. Weak continuity of the effective viscous flow.

Choosing $b=T_k(\tilde{\rho}_\delta)$ in the re-normalized continuity equation, it holds $\widetilde{\mathbb{P}}$ a.s. in the weak sense
\begin{eqnarray}\label{5.21}
dT_k(\tilde{\rho}_\delta)+{\rm div}(T_k(\tilde{\rho}_\delta)\tilde{u}_\delta)dt+(T'_k(\tilde{\rho}_\delta)\tilde{\rho}_\delta-T_k(\tilde{\rho}_\delta)){\rm div}\tilde{u}_{\delta}dt=0.
\end{eqnarray}
In addition, (\ref{5.12}) implies
\begin{eqnarray}\label{5.22*}
T_k(\tilde{\rho}_\delta)\rightarrow\widetilde{T}_{1,k}~ {\rm in }~C([0,T]; H^{-1}(\mathcal{D})).
\end{eqnarray}
Then, combining (\ref{5.10}), (\ref{5.22*}) and (\ref{5.13}), letting $\delta\rightarrow 0$ in (\ref{5.21}), to obtain that
\begin{eqnarray}\label{5.22}
d\widetilde{T}_{1,k}+{\rm div} (\widetilde{T}_{1,k}\tilde{u})dt+\widetilde{T}_{2,k}dt=0,
\end{eqnarray}
holds $\widetilde{\mathbb{P}}$ a.s. in the weak sense.  {We aim to get}
\begin{eqnarray}\label{5.23}
&&\lim_{\delta\rightarrow 0}\widetilde{\mathbb{E}}\int_{0}^{T}\tilde{\psi}(t)\int_{\mathcal{D}}\tilde{\phi}(\tilde{\rho}_{\delta}^{\gamma}-(\mu_2+2\mu_1){\rm div} \tilde{u}_{\delta})T_{k}(\tilde{\rho}_{\delta})dxdt\nonumber \\&&=
\widetilde{\mathbb{E}}\int_{0}^{T}\tilde{\psi}(t)\int_{\mathcal{D}}\tilde{\phi}(\rho^*-(\mu_2+2\mu_1){\rm div} \tilde{u})\widetilde{T}_{1,k}dxdt,
\end{eqnarray}
where the functions $\tilde{\psi}, \tilde{\phi}$ are the same as in \eqref{4.29}. The proof of (\ref{5.23}) is very similar to that of the argument (\ref{4.29}). Here, we skip it.

\noindent{\bf Step 2}. Re-normalized solution.

Define the oscillations defect measure related to the family $\{\tilde{\rho}_{\delta}\}$ by
\begin{eqnarray*}
\mathcal{O}_{\gamma+1}[\tilde{\rho}_{\delta}\rightarrow\tilde{\rho}](\mathcal{D})=\sup_{k\geq 1}\left(\limsup_{\delta\rightarrow 0^+}\widetilde{\mathbb{E}}\int_{0}^T\int_{\mathcal{D}}|T_{k}(\tilde{\rho}_{\delta})-T_{k}(\tilde\rho)|^{\gamma+1}dxdt\right).
\end{eqnarray*}
\begin{lemma}\!\!\!{\rm \cite[Lemma 5.3]{DWang}}\label{lem5.3} There exists a constant $C$ independent of $k$ such that
\begin{eqnarray*}
\mathcal{O}_{\gamma+1}[\tilde{\rho}_{\delta}\rightarrow\tilde{\rho}](\mathcal{D})\leq C.
\end{eqnarray*}
\end{lemma}

With the help of the Lemma \ref{lem5.3}, we may show that the limit $(\tilde \rho, \tilde u)$ satisfies the renormalized continuity equation using the same argument as Lemma 5.4 in \cite{DWang}
\begin{eqnarray}\label{5.25}
\partial_t b(\tilde\rho)+{\rm div }(b(\tilde \rho) \tilde u)+(b'(\tilde \rho)\tilde\rho-b(\tilde \rho)){\rm div} \tilde u=0,
\end{eqnarray}
$\widetilde{\mathbb{P}}$ a.s. in the weak sense.

\noindent{\bf Step 3}. The strong convergence of density.

The proof is also standard, we refer the reader to \cite{Feireisl,DWang,16,Hofmanova} for the deterministic and stochastic case. The proof of Theorem \ref{thm2.1} is completed.

\section{Appendix}

In the appendix, we present some lemmas that will be used frequently in this paper.
\begin{lemma}\!\!\! {{\rm \cite[Theorem 2.1]{Flan2}}}\label{lem6.1} Suppose that $X_{1}\subset X_{0}\subset X_{2}$ are Banach spaces, $X_{1}$ and $X_{2}$ are reflexive,  and the embedding of $X_{1}$ into $X_{0}$ is compact.
Then for any $1<p<\infty,~ 0<\alpha<1$, the embedding
\begin{equation*}
L^{p}(0,T;X_{1})\cap W^{\alpha,p}(0,T;X_{2})\hookrightarrow L^{p}(0,T;X_{0})
\end{equation*}
is compact.
\end{lemma}
\begin{lemma}\!\!\!{\rm \cite[Theorem 1.1.1]{bergh}}\label{lem6.2} Let $\mathcal{L}: L^{p_1}(0,T)\rightarrow L^{p_2}(\mathcal{D})$ and $L^{q_1}(0,T)\rightarrow L^{q_2}(\mathcal{D})$ be a linear operator with $q_1>p_1$ and $q_2<p_2$. Then, for any $s\in (0,1)$, the operator $\mathcal{L}: L^{r_1}(0,T)\rightarrow L^{r_2}(\mathcal{D})$, where $r_1=\frac{1}{s/p_1+(1-s)/q_1}$, $r_2=\frac{1}{s/p_2+(1-s)/q_2}$.
\end{lemma}
\begin{theorem}\!\!\!{\rm \cite[Chapter 3]{kall}}\label{thm6.1} Let $p\geq 1$, $\{X_n\}_{n\geq 1}\in L^p$ and $X_n\rightarrow X$ in probability. Then, the following are equivalent\\
{\rm (1)}. $X_n\rightarrow X$ in $L^p$;\\
{\rm (2)}. the sequence $|X_n|^p$ is uniformly integrable;\\
{\rm (3)}. $\mathbb{E}|X_n|^p\rightarrow \mathbb{E}|X|^p$.
\end{theorem}
\begin{theorem}\!\!\!{\rm \cite[Theorem 1]{jak}}\label{thm6.2}  {Let $X$ be a quasi-Polish space}. If the set of probability measures $\{\nu_n\}_{n\geq 1}$ on $\mathcal{B}(X)$ is tight, then there exists a probability space $(\Omega, \mathcal{F}, \mathbb{P})$ and a sequence of random variables $u_n, u$ such that theirs laws are $\nu_n$, $\nu$ and $u_n\rightarrow u$, $\mathbb{P}$ a.s. as $n\rightarrow \infty$.
\end{theorem}

\section*{Acknowledgments}
 {We are thankful to the referee} for his/her careful reading and many detailed comments and suggestions that help improved the paper.  Z. Qiu's research was supported by the CSC under grant No.201806160015.

\end{document}